 \documentclass[a4paper,10pt, oneside]{article}
\usepackage{syntonly,indentfirst,mathrsfs,latexsym,float}
\usepackage{bm,amsmath,amsbsy,amsthm,amssymb,amsfonts}
\usepackage{array}
\usepackage{multirow}
\usepackage{mathrsfs}
\usepackage{type1cm}
\usepackage{wasysym}
\usepackage{cite,fancyhdr,dsfont,enumerate,graphicx,color,stmaryrd}
\usepackage[colorlinks,linkcolor=black,citecolor=blue]{hyperref}
\usepackage[capitalise, nosort]{cleveref}

\usepackage{indentfirst}
\usepackage{stmaryrd}
\usepackage{subfigure}
\usepackage{graphicx}
\usepackage[utf8]{inputenc}
\usepackage{pgf,tikz}
\usepackage{threeparttable}
\usetikzlibrary{arrows}
\numberwithin{equation}{section}
\newcommand{\tabcaption}{\def\@captype{table}\caption}
\newtheorem{lem}{Lemma}[section]
\newtheorem{thm}{Theorem}[section]

\newtheorem{rem}{Remark}[section]

\newtheorem{exmp}{Example}[section]

\title{EXtended HDG methods for  second order elliptic interface problems\thanks
	{
		This work was supported by  National Natural Science Foundation of China (11771312).
}}
\author{
Yihui Han\thanks{School of Mathematics, Sichuan University, Chengdu, 610064, China, (Email: yihui$\_$han@qq.com)},
Huangxin Chen\thanks{School of Mathematical Sciences and Fujian Provincial Key Laboratory on Mathematical Modeling and High Performance Scientific Computing, Xiamen University, Fujian, 361005, China, (Email: chx@xmu.edu.cn)},
Xiao-Ping Wang\thanks{Department of Mathematics, The Hong Kong University of Science and Technology, Clear Water Bay, Kowloon, Hong Kong, China, (Email: mawang@ust.hk)},
  Xiaoping Xie \thanks{School of Mathematics, Sichuan University, Chengdu, 610064, China, (Email: xpxie@scu.edu.cn)}}
\date{}

\begin{document}
	\maketitle	
\begin{abstract}
	In this paper, we propose  two    arbitrary order  eXtended hybridizable Discontinuous Galerkin (X-HDG) methods for   second order elliptic interface problems in two and three dimensions. The first X-HDG method applies to any piecewise $C^2$ smooth interface. It  uses piecewise polynomials of degrees $k$ $(k\geq 1)$ and $k-1$  respectively  for the potential and flux approximations in the interior  of elements   inside the subdomains,   and  piecewise polynomials of degree $  k$  for the numerical traces of potential on the inter-element boundaries  inside the subdomains.    Double value numerical traces on the  parts of  interface   inside elements are adopted to deal with the jump condition. The second X-HDG method is a modified version of the first one and applies to any fold line/plane interface, which uses piecewise polynomials of degree $  k-1$ for the numerical traces of   potential.  The X-HDG methods are of the  local elimination property, then  lead to reduced systems  which only involve the  unknowns of   numerical traces of potential on the inter-element boundaries and the interface. Optimal error estimates are derived for the flux approximation in $L^2$ norm and for the potential approximation in piecewise $H^1$ seminorm without requiring   ``sufficiently large” stabilization parameters in the schemes. In addition, error estimation for  the potential approximation  in $L^2$ norm is performed using dual arguments.   Finally, we provide several numerical examples   to verify the theoretical results.
	
	\bigskip
	
\end{abstract}
\section{Introduction} 
Elliptic interface problems are widely used in many multi-physics problems and multiphase applications in science computing and engineering \cite{chen2011mibpb,hadley2002high,hesthaven2003high,layton2009using,hou1997hybrid,zhao2010high}. In this paper, we consider the following second order elliptic interface problem: find $u$ satisfying
\begin{align}
\label{pb1}
\left \{
\begin{array}{rl}
- \nabla\cdot(\alpha\nabla u) = f &   {\rm in} \ \Omega_1\cup\Omega_2,\\
u =  g&  {\rm on} \  \partial\Omega,\\
\llbracket u \rrbracket=  g_D, \ \llbracket \alpha\frac{\partial u}{\partial \bm{n}}\rrbracket =  g_N  & {\rm on} \  \Gamma.
\end{array}
\right.
\end{align}
Here $\Omega \subset \mathbb{R}^d$ $(d = 2, 3) $ is a   polygonal/polyhedral domain, which is divided into two subdomains, $\Omega_i$ $(i = 1,2)$,  by a piecewise $C^2$ smooth interface $\Gamma$ (see Figure \ref{domain}).
The coefficient  $\alpha$ is piecewise constant with $\alpha|_{\Omega_i}=\alpha_i>0$ for $i=1,2$. The jump of a function $w$ across the interface $\Gamma$ is defined by $ \llbracket  w  \rrbracket = (w|_{\Omega_{1}})|_\Gamma-(w|_{\Omega_{2}})|_\Gamma$, and $\bm{n}$ denotes the unit normal vector along $\Gamma$ pointing to $\Omega_{2}$.

 \begin{figure}[htp]	
	\centering
	\includegraphics[height = 5.5 cm,width= 6.5 cm]{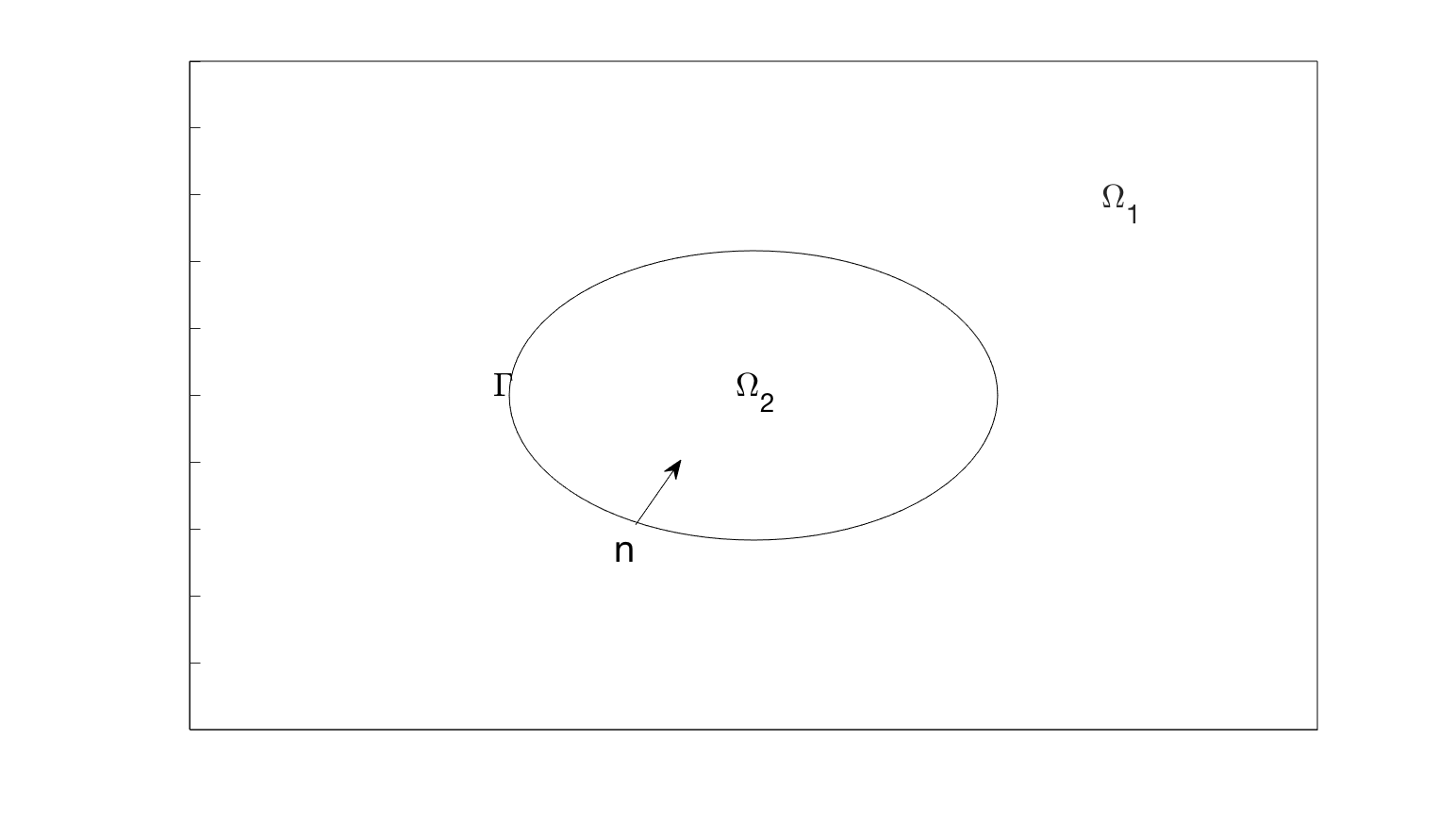} 
	\includegraphics[height = 5.5 cm,width= 6.5 cm]{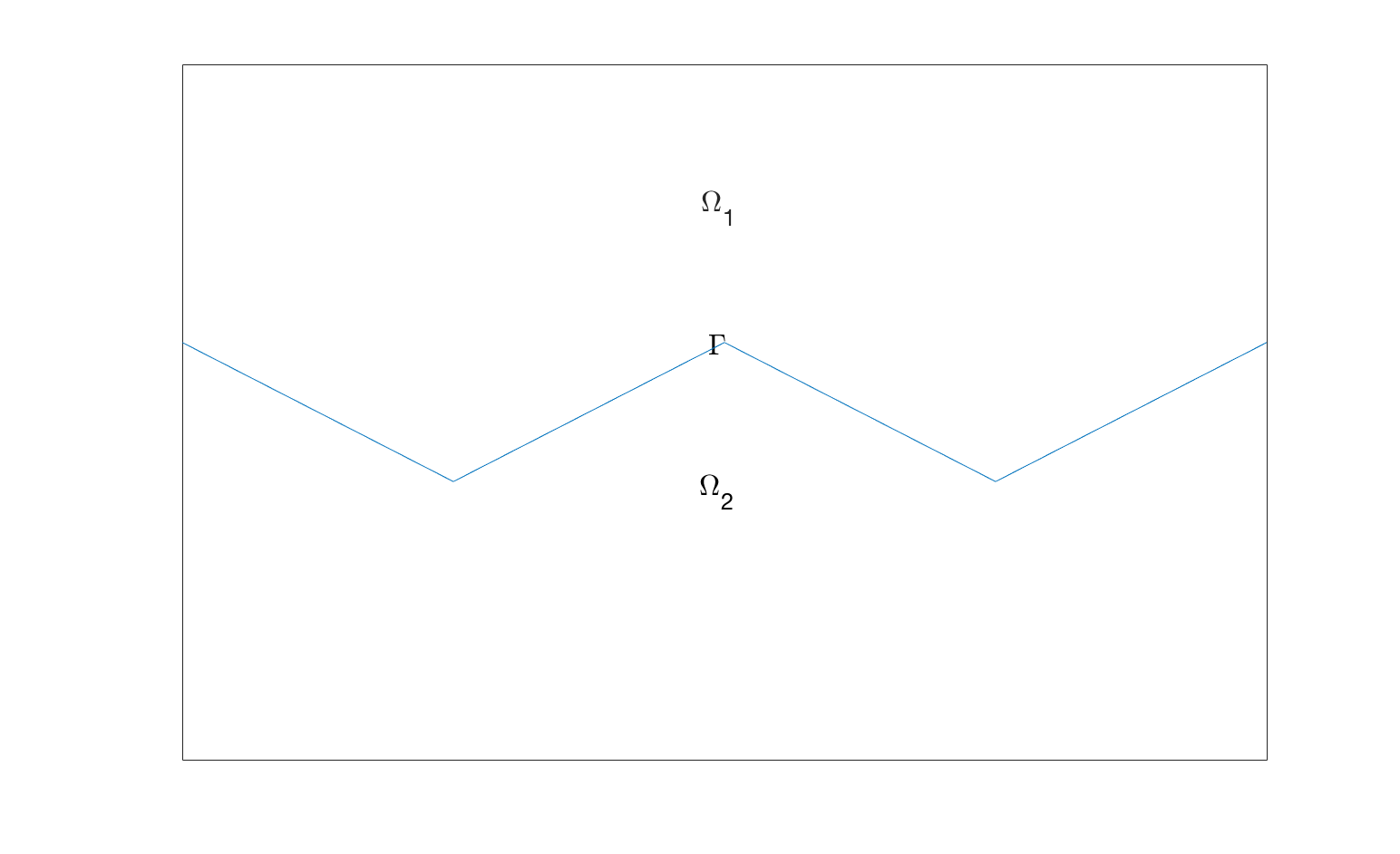} 
	\caption{The geometry of domain with circle interface or fold line interface}\label{domain}
\end{figure}

Due to the  discontinuity of coefficient, the global regularity of the solution to the elliptic interface problem  is generally very low.  This low regularity may result in reduced accuracy of finite element discretization  \cite{babuvska1970finite,xu2013estimate}.  One strategy  for this situation is to use interface(or body)-fitted meshes   to dominate the approximation error caused by the non-smoothness of solution  \cite{barrett1987fittedandunfitted, Brambles96fin, Chen1998Finite, Huang2002Some,   Plum2003Optimal, Li2010Optimal, Cai2011Discontinuous,Cai2017Discontinuous}; see Figure \ref{fitted mesh} for an example. However, the generation of interface-fitted meshes is usually expensive, especially when the interface is of complicated geometry   or   moving with time or iteration. 

Another strategy  avoiding the  loss of numerical accuracy   is to use  certain types of  modification in the finite element discretization for approximating functions around the interface. The resultant finite element methods do not need interface-fitted meshes. One representative of such interface-unfitted methods is the  eXtended Finite Element Method (XFEM) or Generalized Finite Element Method (GFEM) ,  where additional basis functions characterizing the singularity of solution around the interface are enriched into the corresponding approximation space. We  refer to \cite{babuvska1994special,moes1999finite,strouboulis2000design,strouboulis2006generalized,belytschko2009review,babuska2010Optimal,nicaise2011optimal,babuska2011Stable,Burman2012Fictitious,Gupta2015Stable} for some development of   XFEM/GFEM.  It should be pointed out that  a special XFEM   based on the Nitsche's method, called Nitsche-XFEM   was proposed in \cite{hansbo2002unfitted} for the elliptic interface problems. This method
	 enriches    the standard linear finite element space with  additional cut basis functions and   generalizes the results in \cite{babuvska1970finite,barrett1987fittedandunfitted}.  
 Recently, the Nitsche-XFEM was extended to the discretization of  optimal control problems of elliptic interface equations (\cite{wangtao2018nitsche,yangcc2018interface}).   We note that the technique of using cut basis functions as the enrichment  was also applied   in \cite{wu2010unfitted,Massjung2012An,wangchen2014unfitted, wang2016high} to develop  interface-unfitted discontinuous Galerkin methods  for the elliptic interface problems.

The immersed finite element method (IFEM) is another type of interface-unfitted methods,  where   special finite element basis  functions  are constructed to satisfy the interface jump conditions; see, e.g. \cite{leveque1994immersed,lizhilin1998immersed,ewing1999immersed, zhang2004immersed,lizhilin2006immersed, lin2007error,lin2015partially,cao2015iterative} for some development of IFEM.  

\begin{figure}[H]
	\begin{minipage}[t]{0.4\linewidth}
		\centering
		\includegraphics[height = 5.1 cm,width= 7.5 cm, clip, trim = 3.5cm 2cm 3cm 1cm]{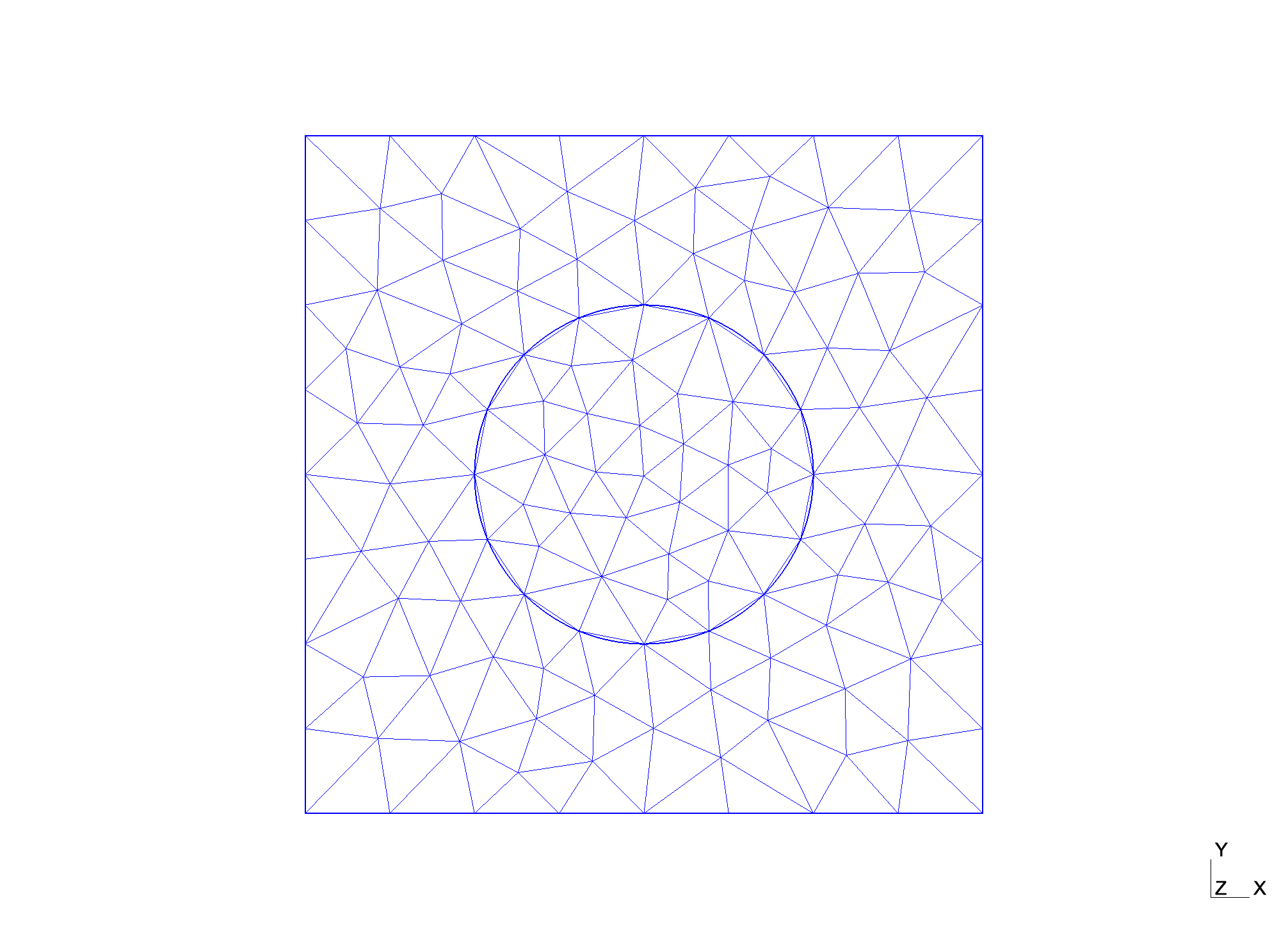}   
		\caption{Fitted mesh}
		\label{fitted mesh}
	\end{minipage}%
	\begin{minipage}[t]{0.5\linewidth}
		\centering
		\includegraphics[width=2.5in]{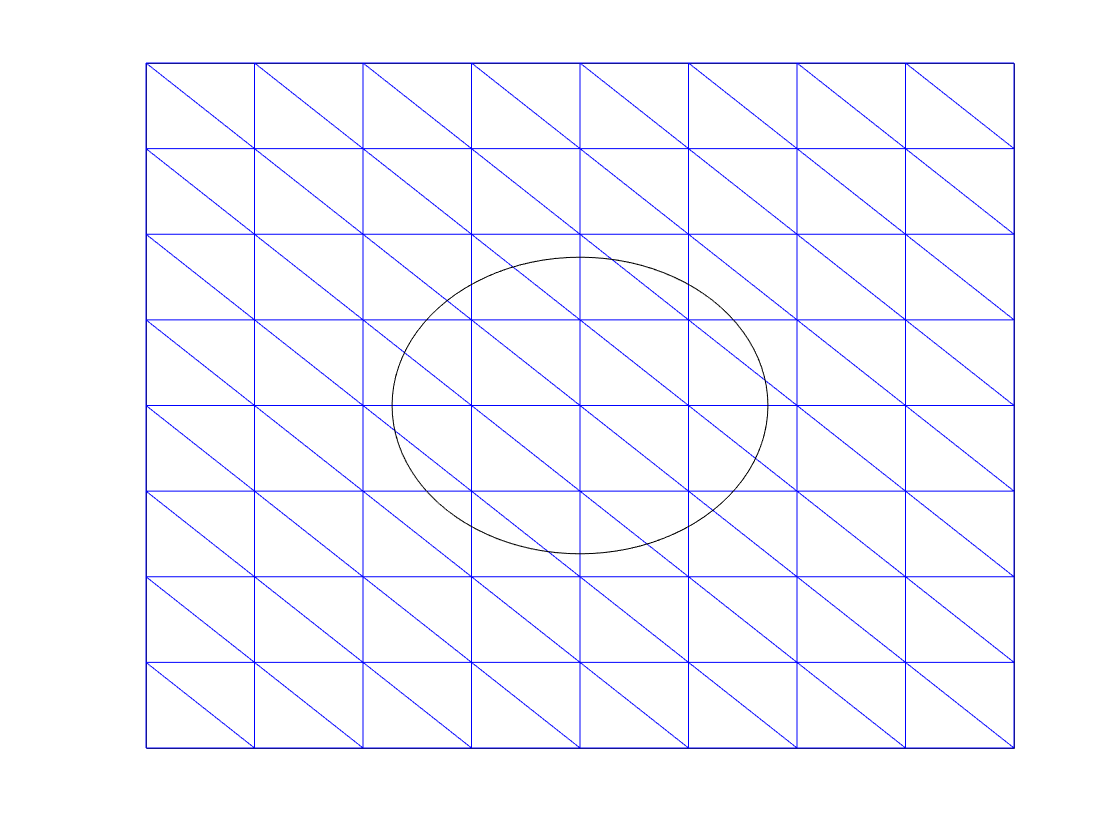}  
		\caption{Unfitted mesh}
		\label{unfitted mesh}
	\end{minipage}
\end{figure}

The hybridizable discontinuous Galerkin (HDG) framework, proposed in \cite{Cockburn2009Unified} for second order elliptic problems, provides a unifying strategy for hybridization of finite element methods. 
It is designed to relax the constraint of function continuity on the inter-element boundaries by introducing Lagrange multipliers defined on the the inter-element boundaries. Thus, it allows piecewise-independent approximation to the potential or flux solution. By the local elimination of the unknowns defined in the interior of elements, the HDG method finally leads to a system where the unknowns are only the globally coupled degrees of freedom describing the Lagrange multipliers. In \cite{Huynh2013A,Wang2013Hybridizable}   high-order interface-fitted HDG methods for solving elliptic and stokes interface problems were proposed. An unfitted HDG method for the Poisson interface problem was presented  in \cite{Dong2018An} by constructing a novel ansatz function in the vicinity of the interface.    Based on the eXtended Finite Element philosophy and a level set description of interfaces, an equal order HDG method was applied in  \cite{G2016eXtendedvoid,G2016eXtended1} to discretize heat bimaterial problems with homogenous interface coditions $g_D = g_N = 0$, where the Heaviside enrichment on cut elements and cut faces is used to represent discontinuities across the interface.  We refer to \cite{cockburn2010comparison,Cockburn2010HDGstokes,Cockburn2011HDGstokes,cockburn2014divergence,chen2014robust,chen2015robust, Li-X-Z2016analysis,Li-X2016analysis,Li-X2016SIAM, chen2017superconvergent} 
for some other developments and applications of the HDG method. 

In this paper we aim to propose two  arbitrary order eXtended HDG (X-HDG) methods for the elliptic interface problem \eqref{pb1}. Compared with \cite{Dong2018An, G2016eXtended1,G2016eXtendedvoid},  our new X-HDG methods are of  the following features:

\begin{itemize}
\item  The first X-HDG method applies to any piecewise $C^2$ smooth interface. It uses piecewise polynomials of degrees $k$ $(k\geq 1)$ and $k-1$  respectively  for the potential and flux approximations in the interior  of elements inside each subdomain,  and piecewise  polynomials of degree $k$ for the numerical traces of potential on the inter-element boundaries  inside each subdomain. 
This means that  the   cut basis functions of   corresponding degrees are applied to enrich the approximation spaces of potential, flux and potential traces around the interface. 

\item The second X-HDG method is a modified version of the first one and applies to any fold line/plane interface. Different from the first one, this modified  method use piecewise polynomials of degree $  k-1$, instead of $k$, for the numerical traces of the potential. 

\item  To deal with the non-homogeneous interface jump condition $\llbracket u \rrbracket=  g_D \neq 0$, we introduce    double value numerical traces with piecewise  polynomials of degree $k$ and $k-1$  on the parts of interface inside elements for the first and second methods,  respectively.

\item For both methods,  optimal error estimates are derived for the flux approximation in $L^2$ norm and for the potential approximation in piecewise $H^1$ seminorm without requiring   ``sufficiently large” stabilization parameters in the schemes. In addition, error estimates for  the potential approximation  in $L^2$ norm are obtained by using dual arguments. 

\item By    local elimination, the proposed X-HDG schemes leads to positive definite systems 
 only involving the unknowns of   numerical traces of potential on the  inter-element boundaries  and  the  interface $\Gamma$. 

\end{itemize}
The rest of the paper is organized as follows.  Section 2  introduces   eXtended finite element (XFE) spaces and eXtended HDG schemes for the elliptic interface problem, and shows the wellposedness of the schemes. Section 3 is devoted to  a priori error estimates for the methods. Numerical examples are provided in Section 4 to verify the theoretical results.

\section{ X-HDG schemes for interface problems}
\subsection{Notations}
For any bounded domain $D\subset \mathbb{R}^s$ $(s= d, d-1)$ and nonnegative integer $m$, let $H^m(D)$ and $H_0^m(D)$ be the usual $m$-th order Sobolev spaces on $D$, with norm $\lVert \cdot\rVert_{m,D}$ and semi-norm   $\lvert\cdot\rvert_{m,D}$.  In particular, $L^2(D):=H^0(D)$ is the space of square integrable functions, with the inner product   $(\cdot,\cdot)_{D}$.
When $D \subset R^{d-1}$, we use $\langle\cdot,\cdot\rangle_D$ to replace $(\cdot,\cdot)_D$.
 We set 
\begin{align*}
H^m(\Omega_{1}\cup\Omega_{2}):= \{v\in L^2(\Omega),v|_{\Omega_{1}} \in H^m(\Omega_{1}),\  and\  v|_{\Omega_{2}} \in H^m(\Omega_{2})\}, \\
\lVert\cdot\rVert_m := \lVert\cdot\rVert_{m,\Omega_1\cup\Omega_2} = \sum\limits_{i=1}^{2} \lVert\cdot\rVert_{m,\Omega_i}, \quad
\lvert\cdot\rvert_m :=\lvert\cdot\rvert_{m,\Omega_1\cup\Omega_2}= \sum\limits_{i=1}^{2} \lvert\cdot\rvert_{m,\Omega_i}. 
\end{align*} 
  For integer $k\geqslant0$,    $P_k(D)$   denotes the set of all polynomials on D with degree no more than $k$.

Let $\mathcal{T}_h=\cup\{K\}$ be a shape-regular triangulation of the domain $\Omega$ consisting of open triangles/tetrahedrons. 
We define the set of all elements intersected by the interface $\Gamma$ as 
\begin{align*}
\mathcal{T}_h^{\Gamma} :=\{K\in \mathcal{T}_h: K\cap\Gamma \neq \emptyset\}.
\end{align*} 
For any $K\in \mathcal{T}_h^{\Gamma}$, called an interface element, let  $\Gamma_K := K\cap \Gamma$ be the part of $ \Gamma$ in $K$, $K_i = K\cap \Omega_{i}$ be the part of $K$ in $\Omega_{i}(i = 1,2)$, and   $\Gamma_{K,h}$ be the straight line/plane segment connecting the intersection between $\Gamma_K$ and $\partial K$. 
To ensure that $\Gamma$ is reasonably resolved by $\mathcal{T}_h$, we make the following standard assumptions on $\mathcal{T}_h$ and interface $\Gamma$: 

\noindent  {\bf(A1)}.  For $K\in \mathcal{T}_h^{\Gamma}$ and any edge/face $F\subset \partial K$ which intersects $\Gamma$,  $F_\Gamma:=\Gamma\cap F$ is simply connected with either   $F_\Gamma =F$ or $meas(F_\Gamma)=0$. 

\noindent 	{\bf(A2)}. For $K\in \mathcal{T}_h^{\Gamma}$, there is a smooth function $\psi$ which maps $\Gamma_{K,h}$ onto $\Gamma_K$.

\noindent 	{\bf(A3)}. For any two different points $\bm{x},\bm{y}\in \Gamma_{K}$, the unit normal vectors $\bm{n}(\bm{x})$ and $\bm{n}(\bm{y})$, pointing to $\Omega_{2}$, at $\bm{x}$ and $\bm{y}$ satisfy
\begin{align}\label{gamma}
\lvert \bm{n}(\bm{x})-\bm{n}(\bm{y})\rvert \leq \gamma h_K
\end{align}
with $\gamma\geq 0$(cf.\cite{Chen1998Finite,xu2013estimate}). Note that   $\gamma = 0$ when $\Gamma_{K}$ is a straight line/plane segment.

  Let  $\mathcal{\varepsilon}_h$ be the set of all edges (faces) of all elements in $\mathcal{T}_h$ and   $\mathcal{\varepsilon}_h^{\Gamma}$ be  the partition of $\Gamma$ with respect to $\mathcal{T}_h$, i.e.
 $$\mathcal{\varepsilon}_h^{\Gamma} := \{ F:\ F=\Gamma_{K}, \text{ or } F= \Gamma\cap \partial K \text{ if $\Gamma\cap \partial K$ is an edge/face of $K$},\forall   K\in \mathcal{T}_h\},$$
 and  set 
 $\mathcal{\varepsilon}_h^* :=\mathcal{\varepsilon}_h\setminus\mathcal{\varepsilon}_h^{\Gamma} $. For any $K\in \mathcal{T}_h$ and   $F\in \mathcal{\varepsilon}_h^*\cup \mathcal{\varepsilon}_h^{\Gamma},$    $h_K$  and $h_F$ denote respectively   the diameters of $K$ and   $F$, and $\bm{n}_K$ denotes the unit outward  normal vector along   $\partial K$. We denote by 
$h: = \max\limits_{K\in \mathcal{T}_h}h_K$  the mesh size of $\mathcal{T}_h$, and by   $\nabla_h$ and $\nabla_h\cdot$ the piecewise-defined gradient and divergence operators with respect to $\mathcal{T}_h$, respectively.

Throughout the paper, we use $a\apprle b$ $(a\apprge b)$ to denote $a\leq Cb$ $(a\geq Cb)$,
where   $C$ is a generic positive constant   independent of mesh parameters $h, h_K, h_e$,   the coefficients $\alpha_i$ $(i =1,2)$ and the  location of the interface relative to the mesh.

\subsection{Two X-HDG schemes}

The X-HDG method is based on a   first-order system of  the elliptic interface problem (\ref{pb1}): 
\begin{subequations}\label{firstorderscheme}
\begin{align}
\bm{q} = \alpha \nabla u &\quad   {\rm in} \ \Omega_1\cup\Omega_2,\\ 
 - \nabla\cdot \bm{q} = f & \quad  {\rm in} \ \Omega_1\cup\Omega_2,\\
 u=g& \quad {\rm on} \  \partial\Omega,\label{2.2c}\\
 \llbracket u \rrbracket= g_D, \ \llbracket \bm{q}\cdot\bm{n}\rrbracket = g_N&\quad {\rm on} \  \Gamma. \label{2.2d}
\end{align}
\end{subequations}

For  $i=1,2$, let $\chi_i$ be the characteristic function on $\Omega_{i}$, and for any  $K\in \mathcal{T}_h$, $F\in \mathcal{\varepsilon}_h^*\cup \mathcal{\varepsilon}_h^{\Gamma}$ and   integer $r\geq 0$,  
let $Q_{r}^b: L^2(D)\rightarrow P_r(D)$ be the standard $L^2$ orthogonal projection operator with 
$D=F\cap\bar{\Omega}_i$. 
Set $$ \oplus\chi_iP_{r}(K) := \chi_1P_{r}(K)+\chi_2P_{r}(K).$$
\subsubsection{X-HDG scheme for a generic piecewise $C^2$ interface}
We  introduce  the following X-HDG finite element spaces:
\begin{align*}
\bm{W}_h :=& \{\bm{w}\in {L^2}(\Omega)^d: \bm{w}|_K \in {P}_{k-1}(K)^d \ {\rm if} \  K\in   \mathcal{T}_h\setminus \mathcal{T}_h^{\Gamma}; \bm{w}|_K \in \left(\oplus\chi_i {P}_{k-1}(K)\right)^d \ {\rm if} \  K\in  \mathcal{T}_h^{\Gamma}\}, \\
V_h :=& \{v\in L^2(\Omega): v|_K \in P_{k}(K) \ {\rm if} \  K\in   \mathcal{T}_h\setminus  \mathcal{T}_h^{\Gamma}; v|_K \in \oplus\chi_i P_{k}(K) \ {\rm if} \  K\in  \mathcal{T}_h^{\Gamma}\}, \\
M_h :=& \{\mu\in L^2(\varepsilon_h^*): \forall F\in \varepsilon_h^*, \mu|_F \in P_{k}(F) \  {\rm if} \  F\cap \Gamma = \emptyset; \mu|_F \in \oplus\chi_i P_{k}(F) \  {\rm if} \   F\cap \Gamma \neq \emptyset\},\\
M_h(g): =& \{\mu\in M_h : \ \mu|_F=Q_{k}^b(g|_F),\  \forall F\in \varepsilon_h^* \ {\it with}\ F  \subset {\partial \Omega}\}, \\
\tilde{M}_h := &\{\tilde{\mu} = \{\tilde{\mu}_1,\tilde{\mu}_2\}: 
\tilde{\mu}_i:=\tilde{\mu}|_{F\cap\bar{\Omega}_i}\in P_{k}(K)|_F,\forall F\in\mathcal{\varepsilon}_h^{\Gamma}, F\subset \bar{K} \text{ for some } K\in   \mathcal{T}_h, i=1,2\},\\
	\tilde{M}_{h}(g_D): = &\{\tilde{\mu} \in \tilde{M}_h: \  \langle \llbracket \tilde{\mu}\rrbracket, {\mu}^*\rangle_{F}=\langle g_D, {\mu}^*\rangle_{F} ,\ \forall  F\in \mathcal{\varepsilon}_h^{\Gamma},\mu^* \in P_{k}(K)|_F \text{ with }  F\subset \bar{K}   \text{ for some } K\in   \mathcal{T}_h\}.
\end{align*}
 It is easy to see that 
$$\tilde{M}_{h}(0)=\{\tilde{\mu} 
\in \tilde{M}_h: \   \llbracket \tilde{\mu}\rrbracket_F = 0,\   \forall F\in \mathcal{\varepsilon}_h^{\Gamma} \}.$$
To describe the X-HDG scheme,   we also define 
\begin{align*}
(\cdot,\cdot)_{\mathcal{T}_h}:=\sum\limits_{K\in  \mathcal{T}_h}(\cdot,\cdot)_K,\quad  \langle\cdot,\cdot\rangle_{\partial\mathcal{T}_h}:=\sum\limits_{K\in  \mathcal{T}_h}\langle\cdot,\cdot\rangle_{\partial K},
\end{align*}
and, for  scalars $w,v$ and vector $\bm{w}$ with $w_i=w|_{F\cap\bar{\Omega}_i}, v_i=v|_{F\cap\bar{\Omega}_i}$ and $\bm{w}_i=\bm{w}|_{F\cap\bar{\Omega}_i}$, 
\begin{align}
\langle w,v\rangle_{*,\Gamma} :&=\sum\limits_{F\in \varepsilon_h^{\Gamma}}\int_{F} (w_1v_1+ w_2v_2)ds,\label{<>_G}\\
\langle \bm{w},v\bm{n}\rangle_{*,\Gamma} :&=\sum\limits_{F\in \varepsilon_h^{\Gamma}}\int_{F} (v_1\bm{w}_1\cdot\bm{n}_1+v_2\bm{w}_2\cdot\bm{n}_2)ds,\label{<n>_G}
\end{align}
where $\bm{n}_i$ denotes  the unit normal vector along $\Gamma$ pointing from   $\Omega_{i}$ to $\Omega_{j}$ with $i,j=1,2$ and $i\neq j$. 
Then the X-HDG method is given as follows:  seek $(\bm{q}_h,u_h,\hat{u}_h,\tilde{u}_h)\in \bm{W}_h\times V_h\times M_h(g)\times \tilde{M}_h(g_D)$ such that
\begin{subequations}\label{X-HDGscheme}
	\begin{align}
	(\alpha^{-1}\bm{q}_h,\bm{w})_{\mathcal{T}_h} + (u_h,\nabla_h\cdot \bm{w})_{\mathcal{T}_h} - \langle \hat{u}_h,\bm{w}\cdot \bm{n}\rangle_{\partial\mathcal{T}_h\setminus \mathcal{\varepsilon}_h^{\Gamma}} -\langle\bm{w}, \tilde{u}_h\bm{n}\rangle_{*,\Gamma} &= 0, \label{X-HDG1}\\
	-(\nabla_h\cdot\bm{q}_h,v)_{\mathcal{T}_h}+ \langle \tau(u_h-\hat{u}_h),v\rangle_{\partial\mathcal{T}_h\setminus \mathcal{\varepsilon}_h^{\Gamma}} +  \langle \eta(u_h-\tilde{u}_h),v\rangle_{*,\Gamma}  &= (f,v)_{\mathcal{T}_h} ,\label{X-HDG2}\\
	\langle \bm{q}_h\cdot\bm{n},\mu\rangle_{\partial \mathcal{T}_h\setminus  \mathcal{\varepsilon}_h^{\Gamma}}-\langle \tau(u_h-\hat{u}_h),\mu\rangle_{\partial \mathcal{T}_h\setminus  \mathcal{\varepsilon}_h^{\Gamma}} &= 0,\label{X-HDG4}\\
	\langle \bm{q}_h,\tilde{\mu}\bm{n}\rangle_{*,\Gamma}-\langle \eta( u_h-\tilde{u}_h),\tilde{\mu}\rangle_{*,\Gamma}&= \langle g_N,\tilde{\mu}\rangle_{*,\Gamma} \label{X-HDG5}
	\end{align}
\end{subequations}
hold for any $(\bm{w},v,\mu,\tilde{\mu})\in \bm{W}_h\times V_h\times M_h(0)\times \tilde{M}_{h}(0)$, 
and the stabilization functions $\tau$, $\eta$ are defined as following:  for any  $K\in \mathcal{T}_h, \ F \in  \partial \mathcal{T}_h$ 
and $i=1,2$,
\begin{align}
\tau|_{F\cap \bar\Omega_{i}} &= \alpha_i h_K^{-1} ,\quad {\rm for} \ F\in \partial \mathcal{T}_h\setminus  \mathcal{\varepsilon}_h^{\Gamma} \ {\rm with} \ F\cap \bar\Omega_{i}\neq\emptyset ,  \label{stablizationpara1}\\ 
 \eta|_{F\cap \bar{ \Omega}_i}&=\alpha_i h_{K}^{-1}  \quad   \text{ for } F\in \varepsilon_h^{\Gamma} \ {\rm with }\ F=\Gamma_K \text{ or } F\subset \partial (K \cap \Omega_i). \label{stablizationpara2}
\end{align}

	\begin{rem}
		We note that this X-HDG scheme is  ``parameter-friendly"  in the sense that there is no need   to choose any ``sufficiently large"  factors in the stabilization functions $\tau$, $\eta$.
	\end{rem}
\begin{rem}
Note that   $\bm{q}_h$ and $u_h$ can be eliminated locally from 
 the X-HDG system (\ref{X-HDGscheme}), which leads to a discrete   system only involving the parameters of numerical traces $\hat{u}_h$ and $\tilde{u}_h$ as unknowns. 
\end{rem}

\begin{thm}\label{thm2.1}
	For $k\geq 1$, the X-HDG system (\ref{X-HDGscheme}) admits a unique solution $(\bm{q}_h,u_h,\hat{u}_h,\tilde{u}_h)$.
\end{thm}
\begin{proof}
	Since (\ref{X-HDGscheme}) is a linear square system, it suffices to show that if all of the given data vanish, i.e. $f = g =  g_D = g_N = 0$, then we get the zero solution. In fact, taking  $(\bm{w},v,\mu,\tilde{\mu}) = (\bm{q}_h,u_h,\hat{u}_h,\tilde{u}_h)$ in \eqref{X-HDG1}-\eqref{X-HDG5} and adding these equations together, we have
	\begin{align*}
	(\alpha^{-1}\bm{q}_h,\bm{q}_h)_{\mathcal{T}_h} + \langle \tau(u_h-\hat{u}_h),(u_h-\hat{u}_h)\rangle_{\partial\mathcal{T}_h\setminus \mathcal{\varepsilon}_h^{\Gamma}} 
	+ \langle \eta(u_h-\tilde{u}_h),(u_h-\tilde{u}_h)\rangle_{*,\Gamma} = 0 .
	\end{align*}
This implies
	\begin{align}
	\bm{q}_h &= \bm{0} \quad {\rm on} \  \mathcal{T}_h  \label{qh27} \\
	 u_h-\hat{u}_h &= 0 \quad {\rm on} \  \partial \mathcal{T}_h\setminus \mathcal{\varepsilon}_h^{\Gamma},  \label{Qh28}\\
	 \{\{(u_h-\tilde{u}_h)^2\}\} &= 0  \quad {\rm on} \   \Gamma. \label{Qlh29}
	\end{align}
where $\{\{\cdot\}\} $ is defined by  
\begin{align}\label{average}
\{\{w\}\} = \frac{1}{2}(w_1+w_2) \quad \text{ with } w_i=w|_{\Gamma\cap\bar{\Omega}_i} , \ i=1,2.
\end{align}
  In light of the above three  relations and  integration by parts the equation \eqref{X-HDG1}   yields
\begin{eqnarray*}
	0=-(\nabla_h u_h, \bm{w})_{\mathcal{T}_h} + \langle u_h-\hat{u}_h,\bm{w}\cdot \bm{n}\rangle_{\partial\mathcal{T}_h\setminus \mathcal{\varepsilon}_h^{\Gamma}} +\langle (u_h-\tilde{u}_h)\bm{n},\bm{w}\rangle_{*,\Gamma}, \ \forall \bm{w} \in \bm{W}_h,
\end{eqnarray*}
which indicates $(\nabla_h u_h, \nabla_h u_h)_{\mathcal{T}_h}=0.$ Thus, $\nabla_h u_h = 0$ and $u_h$ is piecewise constant.  On the  other hand,  the fact $g= g_D  =0$ implies $\hat{u}_h|_{\partial \Omega} = 0$ and $\llbracket \tilde{u}_h\rrbracket _\Gamma= 0$, respectively.      As a result, from \eqref{Qh28} and \eqref{Qlh29}   it follows 
 $u_h = 0, \hat{u}_h=0$ and $ \tilde{u}_h|_{\Gamma\cap\bar\Omega_i}=0$. This completes the proof.
\end{proof}
\subsubsection{Modified X-HDG scheme  for a fold line/plane interface}\label{rem-modified}
	We note that the X-HDG scheme \eqref{X-HDGscheme} applies to any   piecewise $C^2$ interface,  and that $\hat{u}_h\in   M_h(g)$ and $\tilde{u}_h\in  \tilde{M}_h$  are both piecewise polynomials of degree no more than $k$.    In fact, when the interface $\Gamma$ is a fold line/plane, we can use  lower order polynomial approximations for  $\hat{u}_h$ and $  \tilde{u}_h$ to get  a modified X-HDG scheme:  seek $(\bm{q}_h,u_h,\hat{u}_h,\tilde{u}_h)\in \bm{W}_h\times V_h\times M_h^*(g)\times \tilde{M}_h^*(g_D)$ such that
	\begin{subequations}\label{X-HDGscheme1}
		\begin{align}
		(\alpha^{-1}\bm{q}_h,\bm{w})_{\mathcal{T}_h} + (u_h,\nabla_h\cdot \bm{w})_{\mathcal{T}_h} - \langle \hat{u}_h,\bm{w}\cdot \bm{n}\rangle_{\partial\mathcal{T}_h\setminus \mathcal{\varepsilon}_h^{\Gamma}} -\langle\bm{w}, \tilde{u}_h\bm{n}\rangle_{*,\Gamma} &= 0, \label{X-HDG11}\\
		-(\nabla_h\cdot\bm{q}_h,v)_{\mathcal{T}_h}+ \langle \tau(Q_{k-1}^bu_h-\hat{u}_h),v\rangle_{\partial\mathcal{T}_h\setminus \mathcal{\varepsilon}_h^{\Gamma}} +  \langle \eta(Q_{k-1}^bu_h-\tilde{u}_h),v\rangle_{*,\Gamma}  &= (f,v)_{\mathcal{T}_h} ,\label{X-HDG21}\\
		\langle \bm{q}_h\cdot\bm{n},\mu\rangle_{\partial \mathcal{T}_h\setminus  \mathcal{\varepsilon}_h^{\Gamma}}-\langle \tau(Q_{k-1}^bu_h-\hat{u}_h),\mu\rangle_{\partial \mathcal{T}_h\setminus  \mathcal{\varepsilon}_h^{\Gamma}} &= 0,\label{X-HDG41}\\
		\langle \bm{q}_h,\tilde{\mu}\bm{n}\rangle_{*,\Gamma}-\langle \eta( Q_{k-1}^bu_h-\tilde{u}_h),\tilde{\mu}\rangle_{*,\Gamma}&= \langle g_N,\tilde{\mu}\rangle_{*,\Gamma} \label{X-HDG51}
		\end{align}
	\end{subequations}
hold for any $(\bm{w},v,\mu,\tilde{\mu})\in \bm{W}_h\times V_h\times M_h^*(0)\times \tilde{M}_{h}^*(0)$,	
where the modified spaces $M_h^*(g), \tilde{M}_h^*(g_D)$ are defined by
\begin{align*}
M_h^*(g): =& \{\mu\in M_h^* : \ \mu|_F=Q_{k-1}^b(g|_F),\  \forall F\in \varepsilon_h^* \ {\it with}\ F  \subset {\partial \Omega}\}, \\
M_h^* :=& \{\mu\in L^2(\varepsilon_h^*): \forall F\in \varepsilon_h^*, \mu|_F \in P_{k-1}(F) \  {\rm if} \  F\cap \Gamma = \emptyset; \mu|_F \in \oplus\chi_i P_{k-1}(F) \  {\rm if} \   F\cap \Gamma \neq \emptyset\},\\
\tilde{M}_{h}^*(g_D): = &\{\tilde{\mu} \in \tilde{M}_h^*: \  \langle \llbracket \tilde{\mu}\rrbracket, {\mu}^*\rangle_{F}=\langle g_D, {\mu}^*\rangle_{F} ,\ \forall  F\in \mathcal{\varepsilon}_h^{\Gamma},\mu^* \in P_{k}(K)|_F \text{ with }  F\subset \bar{K}   \text{ for some } K\in   \mathcal{T}_h\},\\
\tilde{M}_h^* := &\{\tilde{\mu} = \{\tilde{\mu}_1,\tilde{\mu}_2\}: 
\tilde{\mu}_i:=\tilde{\mu}|_{F\cap\bar{\Omega}_i}\in P_{k-1}(K)|_F,\forall F\in\mathcal{\varepsilon}_h^{\Gamma}, F\subset \bar{K} \text{ for some } K\in   \mathcal{T}_h, i=1,2\}.
\end{align*}
\begin{rem}
	The existence and uniqueness of the solution to \eqref{X-HDGscheme1} can be obtained by following the same routine as in the proof of  Theorem \ref{thm2.1}.  It is easy to see that the size of this modified system is smaller than that of the original system \eqref{X-HDGscheme}.
\end{rem}

\section{A priori error estimates}

This section is devoted to the error estimation for the   X-HDG scheme (\ref{X-HDGscheme}) and the modified scheme (\ref{X-HDGscheme1}).	
Let $Q_{r}: L^2(D)\rightarrow P_r(D)$ be the standard $L^2$ orthogonal projection operator with $D=K\cap\Omega_i$ for any $K\in \mathcal{T}_h$ and $i=1,2$. 
Recall that  $Q_{r}^b$ is the standard $L^2$ orthogonal projection operator from $L^2(F\cap\bar{\Omega}_i) $ onto $P_r(F\cap\bar{\Omega}_i)$ for any $F\in \mathcal{\varepsilon}_h^*\cup \mathcal{\varepsilon}_h^{\Gamma}$.

The following lemma from \cite{wu2010unfitted}  will be used to derive an error estimate of the projection $Q_r$ on the interface $\Gamma$ (cf. Lemma \ref{ineq}).

\begin{lem}\label{trace}
	There exists a positive constant $h_0$ depending only on the interface $\Gamma$, the shape regularity of the mesh $\mathcal{T}_h$, and $\gamma$ in (\ref{gamma}), such that for any $h \in (0, h_0]$ and   $ K\in \mathcal{T}_h^{\Gamma}$, the following estimates hold: 
	\begin{align}
	\lVert v\rVert_{0,\Gamma_K} &\apprle h_K^{-1/2}\lVert v\rVert_{0,K\cap\Omega_{i}} +\lVert v\rVert_{0, K\cap\Omega_{i}}^{1/2} \lVert \nabla v\rVert_{0,K\cap\Omega_{i}}^{1/2}, \quad \forall  v\in H^1(K\cap\Omega_{i}),  \ i=1,2,\\
	\lVert v_h\rVert_{0,\Gamma_K} &\apprle h_K^{-1/2}\lVert v_h\rVert_{0,K\cap\Omega_{i}}, \quad \forall v_h\in P_r(K) .
	\end{align}
\end{lem}

\begin{rem}\label{rem-h0}
	We note that the condition $h\in(0,h_0]$ for some $h_0$ in this lemma  is not required when $\Gamma_{K}$ is a straight line/plane segment, and this condition  is easy to satisfy when $\Gamma_{K}$ is a curved line/surface segment.
\end{rem}

Based on standard  properties of the projection operator and  Lemma \ref{trace}, we have the following estimates.
\begin{lem}\label{ineq} 
	Let $s$ be an integer with $1\leq s\leq r+1$. For any $K\in \mathcal{T}_h$, $h\in (0, h_0]$ and $ v\in H^s\left((K\cap\Omega_{1})\cup (K\cap\Omega_{2}) \right)$, we have
	\begin{align*}
	\lVert v-Q_{r}v\rVert_{0,K}+h\lVert v-Q_{r}v\rVert_{1,K}&\apprle h_K^s\lVert v\rVert_{s,K}, \\
	\lVert v-Q_{r}v\rVert_{0,\partial K}+\lVert v-Q_{r}v\rVert_{0,\Gamma_{K}}&\apprle h_K^{s-1/2}\lVert v\rVert_{s,K}, \\ 
		\lVert v-Q_{r}^bv\rVert_{0,\partial K} 
		&\apprle h_K^{s-1/2}\lVert v\rVert_{s,K},
	\end{align*}
	where the notations $\lVert \cdot\rVert_{s,K} $ and $\lVert\cdot\rVert_{0,\partial K}$ are understood respectively  as   $\lVert \cdot\rVert_{s,K} = \sum\limits_{i=1}^2\lVert \cdot\rVert_{s,K\cap\Omega_{i}}$ and $\lVert \cdot\rVert_{s,\partial K} = \sum\limits_{i=1}^{2}\lVert \cdot\rVert_{s,\partial K\cap \bar \Omega_{i}}$ when  $K\in  \mathcal{T}_h^\Gamma$. 
\end{lem}

In what follows, we shall derive the error estimation for the X-HDG scheme (\ref{X-HDGscheme}) in two cases:

\begin{itemize}
\item[Case 1.]	    the interface $\Gamma$ is a  fold line/plane such that $\Gamma_K$ is a straight  line/plane segment, i.e.  $\Gamma_K=\Gamma_{K,h} $,  for any $K\in  \mathcal{T}_h^{\Gamma}$; 

\item[Case 2.]  $g_D = 0$ when $\Gamma$ is not a fold line/plane. 

\end{itemize}

We set 
\begin{align}\label{erroroperator}
\bm{e}_h^q: = \bm{q}_h - \bm{Q}_{k-1}\bm{q}, \quad  e_h^u: = u_h - Q_{k}u, \quad  e_h^{\hat{u}}: = \hat{u}_h - Q_{k}^bu, \quad  e_h^{\tilde{u}}: = \tilde{u}_h - Q_{k}^{\Gamma}u.
\end{align}
Here 
\begin{align}
(\bm{Q}_{k-1}\bm{q})|_{K\cap\Omega_{i}}:=& \bm{Q}_{k-1}(\bm{q}|_{K\cap\Omega_{i}}), \quad 
(Q_{k}u)|_{K\cap\Omega_{i}}:= Q_{k}(u|_{K\cap\Omega_{i}}), \ \  \forall K\in\mathcal{T}_h , \ i=1,2,\nonumber\\
(Q_{k}^bu)|_{F\cap\Omega_{i}}:=& Q_{k}^b(u|_{F\cap\Omega_{i}}), \ \  \forall F\in\varepsilon_h^{*}, \ i=1,2,\nonumber\\
(Q_{k}^{\Gamma}u)|_F:=&\left\{ \begin{array}{ll}
\{Q_{k}^b(u|_{\bar{\Omega}_{1}\cap F}),Q_{k}^b(u|_{\bar{\Omega}_{2}\cap F})\}, \ \forall F\in\varepsilon_h^{\Gamma} \ \text{and} \ F \text{ is a straight segment},\\
\{u^*_F, u^*_F\}, \ \forall F\in\varepsilon_h^{\Gamma} \ \text{and}\ F \text{ is not a straight segment},
\end{array}\right. \label{Q^Ga}
\end{align} 
where  $\bm{Q}_r$  denotes the 
vector  analogue of $Q_{r}$,    and $u^*_F:=\frac{1}{2}(Q_{k}(u|_{K\cap \Omega_{1}} )|_{F}+Q_{k}(u|_{K\cap \Omega_{2}})|_{F})$.

We also define, for any $\psi \in H^1(\Omega_1\cup\Omega_2)\cup \bm{W}_h\cup V_h\cup  M_h\cup  \tilde{M}_{h}(0)$, 
\begin{align*}
L_1(\psi) &:= \langle ((Q_{k}^{\Gamma}u-u)\bm{n},\psi \rangle_{ *,\Gamma} , \\
L_2(\psi) &:= \langle (\bm{Q}_{k-1}\bm{q}-\bm{q})\cdot\bm{n},\psi \rangle_{\partial\mathcal{T}_h\setminus \mathcal{\varepsilon}_h^{\Gamma}}+\langle \tau (Q_{k}^bu-Q_{k}u),\psi \rangle_{\partial\mathcal{T}_h\setminus \mathcal{\varepsilon}_h^{\Gamma}} , \\
L_3(\psi) &:= \langle (\bm{Q}_{k-1}\bm{q}-\bm{q}),\psi\bm{n} \rangle_{*,\Gamma}+ \langle \eta (Q_{k}^{\Gamma}u-Q_{k}u),\psi \rangle_{*,\Gamma}.
\end{align*}
Then we  have the following results.
\begin{lem}\label{errorequation}
	For any $(\bm{w},v,\mu,\tilde{\mu})\in \bm{W}_h\times V_h\times M_h(0)\times \tilde{M}_{h}(0)$,  it holds
	\begin{subequations}\label{3.4abcd}
		\begin{align}
		(\alpha^{-1}\bm{e}_h^q,\bm{w})_{\mathcal{T}_h}+(e_h^u,\nabla_h\cdot \bm{w})_{\mathcal{T}_h}-\langle e_h^{\hat{u}},\bm{w}\cdot\bm{n} \rangle_{\partial \mathcal{T}_h\setminus\mathcal{\varepsilon}_h^{\Gamma}} -\langle e_h^{\tilde{u}}\bm{n},\bm{w} \rangle_{*,\Gamma} &= L_1(\bm{w}), \label{errorequation1}\\
		-(\nabla_h\cdot\bm{e}_h^q,v)_{\mathcal{T}_h}+ \langle \tau(e_h^u-e_h^{\hat{u}}),v\rangle_{\partial\mathcal{T}_h\setminus \mathcal{\varepsilon}_h^{\Gamma}} +  \langle \eta(e_h^u-e_h^{\tilde{u}}),v\rangle_{*,\Gamma}  &= L_2(v) +L_3(v),\label{errorequation2}\\
		\langle \bm{e}_h^q\cdot\bm{n},\mu\rangle_{\partial\mathcal{T}_h\setminus \mathcal{\varepsilon}_h^{\Gamma}}-\langle \tau(e_h^u-e_h^{\hat{u}}),\mu\rangle_{\partial\mathcal{T}_h\setminus  \mathcal{\varepsilon}_h^{\Gamma}}&= -L_2(\mu),\label{errorequation3}\\
		\langle \bm{e}_h^q,\tilde{\mu}\bm{n}\rangle_{*,\Gamma}-\langle \eta(e_h^u-e_h^{\tilde{u}}),\tilde{\mu}\rangle_{*,\Gamma}&= -L_3(\tilde{\mu}) \label{errorequation4}.
		\end{align}
	\end{subequations}

\end{lem}
\begin{proof}
	From \eqref{firstorderscheme},  the  definitions  of the projections, and integration by parts, 
	we obtain 
	\begin{eqnarray*}
		(\alpha^{-1}\bm{Q}_{k-1}\bm{q},\bm{w})_{\mathcal{T}_h}+(Q_{k}u,\nabla_h\cdot \bm{w})_{\mathcal{T}_h}-\langle Q_{k}^bu,\bm{w}\cdot\bm{n} \rangle_{\partial \mathcal{T}_h\setminus\mathcal{\varepsilon}_h^{\Gamma}} -\langle Q_{k}^{\Gamma}u\bm{n},\bm{w}\rangle_{*,\Gamma}=\langle (u-Q_{k}^{\Gamma}u)\bm{n},\bm{w} \rangle_{ *,\Gamma},  \forall \bm{w}\in \bm{W},
		\\
		-(\nabla_h\cdot\bm{Q}_{k-1}\bm{q},v)_{\mathcal{T}_h} +\langle (\bm{Q}_{k-1}\bm{q}-\bm{q})\cdot\bm{n},v \rangle_{\partial \mathcal{T}_h\setminus\mathcal{\varepsilon}_h^{\Gamma}}+\langle (\bm{Q}_{k-1}\bm{q}-\bm{q}),v\bm{n} \rangle_{*,\Gamma}= (f,v)_{\mathcal{T}_h}, \quad \forall  v\in  V_h.
	\end{eqnarray*}
	Subtracting (\ref{X-HDG1}) and (\ref{X-HDG2}) from the above two equations respectively yields (\ref{errorequation1}) and (\ref{errorequation2}). Similarly, the relations (\ref{errorequation3}) and  (\ref{errorequation4}) follow from  \eqref{firstorderscheme},(\ref{X-HDG4}) and  (\ref{X-HDG5}).
\end{proof}

Define a semi-norm $\interleave \cdot\interleave$ on $ \bm{W}_h\times V_h\times M_h\times \tilde{M}_h$ by 
\begin{align}\label{|||-norm}
\interleave (\bm{w},v,\mu,\tilde{\mu})\interleave^2 : = \lVert \alpha^{-1/2}\bm{w}\rVert_{0,\mathcal{T}_h}^2 +\lVert \tau^{1/2}(v-\mu)\rVert_{0,\partial\mathcal{T}_h\setminus \mathcal{\varepsilon}_h^{\Gamma}}^2+\lVert \eta^{1/2}(v-\tilde{\mu})\rVert_{*,\Gamma}^2   
\end{align}
for any $ (\bm{w},v,\mu,\tilde{\mu})\in \bm{W}_h\times V_h\times M_h\times \tilde{M}_h,$ where
	\begin{align*}
	\lVert w\rVert_{0,\mathcal{T}_h}^2 := \sum_{K\in\mathcal{T}_h}\sum_{i = 1}^{2}\lVert w\rVert_{0,K\cap\Omega_i}^2,  \quad 
	\lVert w\rVert_{0,\partial\mathcal{T}_h\setminus \mathcal{\varepsilon}_h^{\Gamma}}^2 := \sum_{F\in\partial\mathcal{T}_h\setminus\varepsilon_h^{\Gamma}}\sum_{i = 1}^{2}\lVert w\rVert_{0,F\cap\Omega_i}^2,  \quad 
	\lVert w\rVert_{*,\Gamma}^2 := \langle w,w\rangle_{*,\Gamma}.  
	\end{align*}

\begin{lem}\label{errorterm} 
	For any $h \in (0, h_0]$, it hold
	\begin{align} \label{id3.5-estimategrade_1}
	 \interleave (\bm{e}_h^q,e_h^u,e_h^{\hat{u}},e_h^{\tilde{\mu}})\interleave  &=\left( \sum_{i=1}^{3}E_i\right)^{1/2}, \\
		\lVert  \alpha^{1/2}\nabla_h e_h^u\rVert_{0,\mathcal{T}_h} &\apprle 
		\left\{ \begin{array}{ll}
		\interleave (\bm{e}_h^q,e_h^u,e_h^{\hat{u}},e_h^{\tilde{\mu}})\interleave +h^{-1/2}\lVert \nu^{1/2}(u-Q_k^{\Gamma}u) \rVert_{*,\Gamma}, \quad $ if interface is not a fold line$,\\
		\interleave (\bm{e}_h^q,e_h^u,e_h^{\hat{u}},e_h^{\tilde{\mu}})\interleave , \quad $ if interface is a fold line$
		\end{array}\right.
	\end{align}
	where 
	\begin{align*}
	E_1&= \langle (Q_{k}^{\Gamma}u-u)\bm{n},\bm{e}_h^q \rangle_{*,\Gamma} ,\\ 
	E_2&= \langle (\bm{Q}_{k-1}\bm{q}-\bm{q})\cdot\bm{n}, e_h^u-e_h^{\hat{u}}\rangle_{\partial \mathcal{T}_h\setminus\mathcal{\varepsilon}_h^{\Gamma}}+\langle (\bm{Q}_{k-1}\bm{q}-\bm{q}),(e_h^u-e_h^{\tilde{u}})\bm{n} \rangle_{*,\Gamma} , \\
	E_3 &= \langle \tau (Q_{k}^bu-Q_{k}u),e_h^u-e_h^{\hat{u}} \rangle_{\partial \mathcal{T}_h\setminus\mathcal{\varepsilon}_h^{\Gamma}}+ \langle \eta (Q_{k}^{\Gamma}u-Q_{k}u),e_h^u-e_h^{\tilde{u}} \rangle_{*,\Gamma}.
	\end{align*} 
\end{lem}

\begin{proof}
	We  first show the identity
	$$\interleave (\bm{e}_h^q,e_h^u,e_h^{\hat{u}},e_h^{\tilde{\mu}})\interleave  =\left( \sum_{i=1}^{3}E_i\right)^{1/2}.$$ Take $(\bm{w},v,\mu ,\tilde{\mu}) = (\bm{e}_h^q,e_h^u,e_h^{\hat{u}},e_h^{\tilde{\mu}})$ in \eqref{3.4abcd} and sum the obtained four error equations, we then get   
	\begin{align*}
	\interleave (\bm{e}_h^q,e_h^u,e_h^{\hat{u}},e_h^{\tilde{\mu}})\interleave^2 = L_1(\bm{e}_h^q)+ L_2(e_h^u-e_h^{\hat{u}})+L_3(e_h^u-e_h^{\hat{u}}),
	\end{align*}
	which, together with  the definitions of $L_i(\cdot)$ ($i=1,2,3$), yields the desired identity.
	
  The case of fold line interface is easier, so we only prove the case when interface is not a fold line, i.e. 
	\begin{align}\label{left<=}
\lVert  \alpha^{1/2}\nabla_h e_h^u\rVert_{0,\mathcal{T}_h} &\apprle \interleave (\bm{e}_h^q,e_h^u,e_h^{\hat{u}},e_h^{\tilde{\mu}})\interleave +h^{-\frac{1}{2}}\lVert \nu^{1/2}(u-Q_k^{\Gamma}u) \rVert_{*,\Gamma}.
	\end{align}
	On one hand, taking $\bm{w} = \alpha\nabla e_h^u$ in (\ref{errorequation1}) and applying integration by parts, we obtain
	\begin{align*}
	(\bm{e}_h^q,\nabla e_h^u)_{\mathcal{T}_h}-(\alpha\nabla e_h^u, \nabla e_h^u)_{\mathcal{T}_h}+\langle \alpha (e_h^u-e_h^{\hat{u}}),\nabla e_h^u\cdot\bm{n} \rangle_{\partial\mathcal{T}_h\setminus \mathcal{\varepsilon}_h^{\Gamma}} +\langle \alpha (e_h^u-e_h^{\tilde{u}})\bm{n},\nabla e_h^u \rangle_{*,\Gamma}&= \langle (Q_k^{\Gamma}u-u)\bm{n},\nabla e_h^u\rangle_{*,\Gamma},
	\end{align*}
	Then, by the  Cauchy-Schwarz inequality and Lemma \ref{trace} we have
	\begin{align*}
	\lVert \alpha^{1/2}\nabla e_h^u\rVert_{0,\mathcal{T}_h} \leq \lVert \alpha^{-1/2}\bm{e}_h^q\rVert_{0,\mathcal{T}_h}+\lVert  \tau^{1/2}(e_h^u-e_h^{\hat{u}})\rVert_{0,\partial\mathcal{T}_h\setminus \mathcal{\varepsilon}_h^{\Gamma}} +\lVert  \eta^{1/2}(e_h^u-e_h^{\tilde{u}})\rVert_{*,\Gamma}+h^{-\frac{1}{2}} \lVert \nu^{1/2}(u-Q_k^{\Gamma}u) \rVert_{*,\Gamma}.
	\end{align*}
	Combining the definition of $\interleave \cdot\interleave$, which indicates  (\ref{left<=}).  This  completes the proof.
\end{proof}

In light of Lemma \ref{ineq}, Lemma \ref{errorterm} and the definition of $\interleave \cdot\interleave$, we can derive the following optimal error estimates.
\begin{thm}\label{estimateofenergynorm}
	Let $(u,\bm{q})\in H^{k+1}(\Omega_1\cup \Omega_2)\times H^k(\Omega_1\cup \Omega_2)^d $ and $(\bm{q}_h,u_h,\hat{u}_h,\tilde{u}_h)\in \bm{W}_h\times V_h\times M_h(g)\times \tilde{M}_h$ be the solutions of the  problem (\ref{firstorderscheme}) and the X-HDG scheme (\ref{X-HDGscheme}), respectively.
	Then  the following error estimate holds for any $h \in (0, h_0]$:
	\begin{align}
	\interleave (\bm{e}_h^q,e_h^u,e_h^{\hat{u}},e_h^{\tilde{\mu}})\interleave 
	\apprle h^k\lvert \alpha^{1/2}u\rvert_{k+1,\Omega_1\cup \Omega_2}. \label{est_add0}
	\end{align}
	Further more, it holds 
	\begin{align} 
	\lVert \alpha^{-1/2}(\bm{q}-\bm{q}_h)\rVert_{0,\mathcal{T}_h} +\lVert \alpha^{1/2}(\nabla_h u-\nabla_h u_h)\rVert_{0,\mathcal{T}_h} \apprle h^k\lvert \alpha^{1/2}u\rvert_{k+1,\Omega_1\cup \Omega_2}. \label{est_add}
	\end{align}
\end{thm}
\begin{proof}
	In view of  Lemma \ref{errorterm},  we need to estimate the terms $E_1, E_2$ and $E_3$.
with Cauchy-Schwarz inequality and th property of projection, we have
	\begin{align*}
	E_1  = \langle \alpha^{1/2}(Q_{k}^{\Gamma}u-u)\bm{n},\alpha^{-1/2}\bm{e}_h^q \rangle_{*,\Gamma} \apprle h^k\lvert \alpha^{1/2}u\rvert_{k+1,\Omega_1\cup \Omega_ 2} \lVert \alpha^{-1/2}\bm{e}_h^q\rVert_{0,\Omega}.
	\end{align*}
	Similarly, we can obtain
	\begin{align*}
	E_2  \apprle & h^k\lvert \alpha^{-1/2}q\rvert_{k,\Omega_1\cup \Omega_ 2} (\lVert \tau^{1/2}(e_h^u-e_h^{\hat{u}})\rVert_{0,\partial\mathcal{T}_h\setminus \mathcal{\varepsilon}_h^{\Gamma}}+\lVert \eta^{1/2}(e_h^u-e_h^{\tilde{u}})\rVert_{*,\Gamma}),\\
	E_3  = &\langle \tau(Q_{k}^bu-Q_{k}u),e_h^u-e_h^{\hat{u}} \rangle_{\partial \mathcal{T}_h\setminus\mathcal{\varepsilon}_h^{\Gamma}}+ \langle \eta (Q_{k}^{\Gamma}u-Q_ku),e_h^u-e_h^{\tilde{u}} \rangle_{*,\Gamma} \\
	\leq &\left(\lVert \tau^{1/2}(Q_{k}^bu-Q_{k}u)\rVert_{0,\partial\mathcal{T}_h\setminus \mathcal{\varepsilon}_h^{\Gamma}}+ \lVert \eta^{1/2}(Q_{k}^{\Gamma}u-Q_{k}u)\rVert _{*,\Gamma}\right)\left(\lVert \tau^{1/2}(e_h^u-e_h^{\hat{u}})\rVert_{0,\partial\mathcal{T}_h\setminus \mathcal{\varepsilon}_h^{\Gamma}}+\lVert \eta^{1/2}(e_h^u-e_h^{\tilde{u}})\rVert_{*,\Gamma}\right) \\
	\apprle & h^{k} \lvert \alpha^{1/2}u\rvert_{k+1,\Omega_1\cup \Omega_ 2} \left(\lVert \tau^{1/2}(e_h^u-e_h^{\hat{u}})\rVert_{0,\partial\mathcal{T}_h\setminus \mathcal{\varepsilon}_h^{\Gamma}}+\lVert \eta^{1/2}(e_h^u-e_h^{\tilde{u}})\rVert_{*,\Gamma}\right).
	\end{align*}
	The above three inequalities and Lemma \ref{errorterm} imply the estimate \eqref{est_add0}. And the estimate \eqref{est_add} follows from  \eqref{est_add0}, the triangle inequality and Lemma \ref{ineq}.
\end{proof}

\begin{rem}\label{est_1norm}
	From (\ref{est_add}) we easily get 
	\begin{align}
	\lVert \bm{q}-\bm{q}_h\rVert_{0,\mathcal{T}_h}\apprle & \alpha_{max}h^k\lvert u\rvert_{k+1,\Omega_1\cup \Omega_2},\label{est_add3}\\
	\lVert \nabla_h u- \nabla_h u_h\rVert_{0,\mathcal{T}_h} \apprle &(\frac{\alpha_{max}}{\alpha_{min}})^{1/2}h^k\lvert u\rvert_{k+1,\Omega_1\cup \Omega_2}\label{est_add2}.
	\end{align}
	Here   $\alpha_{max} = \max\limits_{i=1,2}{\alpha_i}$ and $\alpha_{min} = \min\limits_{i=1,2}{\alpha_i}$.  We recall that  $a\apprle b$   denotes $a\leq Cb$  with   $C$  being  a generic positive constant   independent of mesh parameters $h, h_K, h_e$,   the coefficients $\alpha_i$ $(i =1,2)$ and the  location of the interface relative to the mesh.
\end{rem}

\begin{rem}
We note that the condition $h\in(0,h_0]$ for some $h_0$ in Theorem \ref{estimateofenergynorm}  is not required when $\Gamma$ is a fold line/plane; see Remark \ref{rem-h0}.
\end{rem}
   To analyze  the modified X-HDG scheme \eqref{X-HDGscheme1}, we need to modify the semi-norm $\interleave (\cdot,\cdot,\cdot,\cdot)\interleave$ in \eqref{|||-norm}  and the errors  $\bm{e}_h^q,e_h^u,e_h^{\hat{u}},e_h^{\tilde{\mu}}$ in \eqref{erroroperator} respectively as
	\begin{align}
	\interleave (\bm{w},v,\mu,\tilde{\mu})\interleave^2 : = \lVert \alpha^{-1/2}\bm{w}\rVert_{0,\mathcal{T}_h}^2 +\lVert \tau^{1/2}(Q_{k-1}^bv-\mu)\rVert_{0,\partial\mathcal{T}_h\setminus \mathcal{\varepsilon}_h^{\Gamma}}^2+\lVert \eta^{1/2}(Q_{k-1}^bv-\tilde{\mu})\rVert_{*,\Gamma}^2,\\
		\bm{e}_h^q: = \bm{q}_h - \bm{Q}_{k-1}\bm{q}, \quad  e_h^u: = u_h - Q_{k}u, \quad  e_h^{\hat{u}}: = \hat{u}_h - Q_{k-1}^bu, \quad  e_h^{\tilde{u}}: = \tilde{u}_h - Q_{k-1}^bu. \label{e_h^q}
		\end{align}
Then, by following the same line as in the proof of  Theorem \ref{estimateofenergynorm},    we can obtain the following conclusion.

\begin{thm}\label{energy-modified}
	Let $(u,\bm{q})\in H^{k+1}(\Omega_1\cup \Omega_2)\times H^k(\Omega_1\cup \Omega_2)^d $ and $(\bm{q}_h,u_h,\hat{u}_h,\tilde{u}_h)\in \bm{W}_h\times V_h\times M_h(g)\times \tilde{M}_h$ be the solutions of the  problem (\ref{firstorderscheme}) and the modified X-HDG scheme (\ref{X-HDGscheme1}), respectively. Then the estimates  \eqref{est_add0} -\eqref{est_add2}  still hold. 
\end{thm}

\section{$L^2$ error  estimation for the numerical potential}

In this section, we shall   perform the Aubin-Nitsche duality argument  to derive the $L^2$ error estimation for the potential approximation $u_h$ in the schemes (\ref{X-HDGscheme}) and   (\ref{X-HDGscheme1}). 

For the scheme (\ref{X-HDGscheme}), let us introduce the auxiliary problem
\begin{align}
\label{dual problem}
\left \{
\begin{array}{rl}
\bm{\Phi} = \alpha \nabla \phi &\quad   {\rm in} \ \Omega_1\cup\Omega_2,\\
- \nabla\cdot \bm{\Phi} = u-u_h & \quad  {\rm in} \ \Omega_1\cup\Omega_2,\\
\phi=0& \quad {\rm on} \ \partial\Omega,\\
\llbracket \phi\rrbracket= 0, \  \llbracket\bm{\Phi}\cdot \bm{n}\rrbracket = 0&\quad {\rm on} \  \Gamma,
\end{array}
\right.
\end{align} 
and assume the   regularity estimate
\begin{align}\label{regularestimate}
\lVert \bm{\Phi}\rVert_{1,\Omega_1\cup\Omega_2} +\lVert \alpha\phi\rVert_{2,\Omega_1\cup\Omega_2} \apprle \lVert u-u_h\rVert_{0,\mathcal{T}_h}.
\end{align}
We note that this regularity result holds when $\Omega$ is convex and   $\Gamma$ is $C^2$ (cf.\cite{Huang2012Uniform},Theorem 4.5), and it is   sharp in terms of the coefficient $\alpha$.

In light of the auxiliary problem \eqref{dual problem}, we can obtain the following conclusion:
\begin{thm}\label{L2}
	Let $(u,\bm{q})\in H^{k+1}(\Omega_1\cup \Omega_2)\times H^k(\Omega_1\cup \Omega_2)^d $ and $(\bm{q}_h,u_h,\hat{u}_h,\tilde{u}_h)\in \bm{W}_h\times V_h\times M_h(g)\times \tilde{M}_h$ be the solutions of the  problem (\ref{firstorderscheme}) and the X-HDG scheme (\ref{X-HDGscheme}), respectively. Under the regularity assumption \eqref{regularestimate},   for any $h \in (0, h_0]$ it holds
	\begin{align}\label{estimategeneralinterface}
	\lVert u-u_h\rVert_{0,\mathcal{T}_h} \apprle (\frac{\alpha_{max}}{\alpha_{min}})^{1/2}h^{k+1} \lVert u\rVert_{k+1,\Omega_1\cup\Omega_2}
	\end{align}
for either of the  following two cases:  (1)  the interface $\Gamma$ is a  fold line/plane such that $\Gamma_K$ is a straight  line/plane segment, i.e.  $\Gamma_K=\Gamma_{K,h} $,  for any $K\in  \mathcal{T}_h^{\Gamma}$; (2) $g_D = 0$ when $\Gamma$ is not a fold line/plane.
\end{thm}
\begin{proof}
 For simplicity, we define, for any $(\bm{\sigma},\xi,\hat{\xi},\tilde{\xi}),(\bm{w},v,\mu,\tilde{\mu})\in L^2(\Omega)^d\times L^2(\Omega)\times L^2(\varepsilon_h^*)\times L^2(\mathcal{\varepsilon}_h^{\Gamma})$, 
	\begin{align*}
	B_h(\bm{\sigma},\xi,\hat{\xi},\tilde{\xi};\bm{w},v,\mu,\tilde{\mu}) :
	=& 	(\alpha^{-1}\bm{\sigma},\bm{w})_{\mathcal{T}_h} + (\xi,\nabla_h\cdot \bm{w})_{\mathcal{T}_h} - \langle \hat{\xi},\bm{w}\cdot \bm{n}\rangle_{\partial\mathcal{T}_h\setminus \mathcal{\varepsilon}_h^{\Gamma}} -\langle\bm{w}, \tilde{\xi}\bm{n}\rangle_{*,\Gamma}  \\
	&-(\nabla_h\cdot\bm{\sigma},v)_{\mathcal{T}_h}+ \langle \tau(\xi-\hat{\xi}),v-\mu\rangle_{\partial\mathcal{T}_h\setminus \mathcal{\varepsilon}_h^{\Gamma}} +  \langle \eta(\xi-\tilde{\xi}),v-\tilde{\mu}\rangle_{*,\Gamma}  \\
&+	\langle \bm{\sigma}\cdot\bm{n},\mu\rangle_{\partial \mathcal{T}_h\setminus (\partial \Omega\cup \mathcal{\varepsilon}_h^{\Gamma})} +\langle \bm{\sigma},\tilde{\mu}\bm{n}\rangle_{*,\Gamma} .
	\end{align*}
From (\ref{firstorderscheme}) and  (\ref{X-HDGscheme}) it follows
	\begin{align*} B_h(\bm{q}-\bm{q}_h,u-u_h,u-\hat{u}_h,u-\tilde{u}_h;\bm{w}_h,v_h,\mu_h,\tilde{\mu}_h) = 0,\  \forall (\bm{w}_h,v_h,\mu_h,\tilde{\mu}_h)\in \bm{W}_h\times V_h\times M_h(0)\times \tilde{M}_{h}(0).
	\end{align*} 
	By \eqref{dual problem}  we have, for  $(\bm{w},v,\mu,\tilde{\mu})\in L^2(\Omega)^d\times L^2(\Omega)\times L^2(\varepsilon_h^*)\times L^2(\mathcal{\varepsilon}_h^{\Gamma})$, 
	\begin{subequations}\label{X-HDGscheme-a}
		\begin{align}
		(\alpha^{-1}\bm{\Phi},\bm{w})_{\mathcal{T}_h} + (\phi,\nabla_h\cdot \bm{w})_{\mathcal{T}_h} - \langle \phi,\bm{w}\cdot \bm{n}\rangle_{\partial\mathcal{T}_h\setminus \mathcal{\varepsilon}_h^{\Gamma}} -\langle\bm{w}, \phi\bm{n}\rangle_{*,\Gamma} &= 0, \label{1}\\
		-(\nabla_h\cdot\bm{\Phi},v)_{\mathcal{T}_h}+ \langle \tau(\phi-\phi),v\rangle_{\partial\mathcal{T}_h\setminus \mathcal{\varepsilon}_h^{\Gamma}} +  \langle \eta(\phi-\phi),v\rangle_{*,\Gamma}  &= (u-u_h,v)_{\mathcal{T}_h} ,\label{2}\\
		\langle \bm{\Phi}\cdot\bm{n},\mu\rangle_{\partial \mathcal{T}_h\setminus  \mathcal{\varepsilon}_h^{\Gamma}}-\langle \tau(\phi-\phi),\mu\rangle_{\partial \mathcal{T}_h\setminus  \mathcal{\varepsilon}_h^{\Gamma}} &= \langle \bm{\Phi}\cdot\bm{n},\mu\rangle_{\partial\Omega},\label{4}\\
		\langle \bm{\Phi},\tilde{\mu}\bm{n}\rangle_{*,\Gamma}-\langle \eta( \phi-\phi),\tilde{\mu}\rangle_{*,\Gamma}&= \langle \bm{\Phi},\tilde{\mu}\bm{n}\rangle_{*,\Gamma} \label{5}.
		\end{align}
	\end{subequations}
Take   $\bm{w}=\bm{q}-\bm{q}_h,v=u-u_h,\mu=u-\hat{u}_h,\tilde{\mu}=u-\tilde{u}_h:= \{u|_{{\Omega_{1}}\cap\Gamma}-u_{h1},u|_{{\Omega_{2}}\cap\Gamma}-u_{h2}\}$ in the above four equations, and add the equations all together, then we have 
	\begin{align*}
	&\lVert u-u_h\rVert_{0,\mathcal{T}_h}^2 \\
	=& B_h(\bm{\Phi},\phi,\phi,\phi;\bm{q}-\bm{q}_h,u-u_h,u-\hat{u}_h,u-\tilde{u}_h)- \langle \bm{\Phi},(u-\tilde{u}_h)\bm{n}\rangle_{*,\Gamma} -\langle \bm{\Phi}\cdot\bm{n},u-\hat{u}_h\rangle_{\partial\Omega}   \\
	=&  B_h(-(\bm{q}-\bm{q}_h),u-u_h,u-\hat{u}_h,u-\tilde{u}_h;-\bm{\Phi},\phi,\phi,\phi)  -\langle \bm{\Phi},(u-\tilde{u}_h)\bm{n}\rangle_{*,\Gamma}-\langle \bm{\Phi}\cdot\bm{n},u-\hat{u}_h\rangle_{\partial\Omega}\\
	=& B_h(-(\bm{q}-\bm{q}_h),u-u_h,u-\hat{u}_h,u-\tilde{u}_h;-(\bm{\Phi}-\bm{Q}_{k-1}\bm{\Phi}),\phi-Q_{k}\phi,\phi-Q_{k}^b\phi,\phi-\{\{Q_{k}^{\Gamma}\phi\}\})  \\
	& - \langle \bm{\Phi},(u-\tilde{u}_h)\bm{n}\rangle_{*,\Gamma} -\langle \bm{\Phi}\cdot\bm{n},(u-\hat{u}_h)\rangle_{\partial\Omega}\\
=& \sum_{i=1}^{5}I_i,
	\end{align*}
	where 
	\begin{align*}
	I_1 &= (\alpha^{-1}(\bm{q}-\bm{q}_h),(\bm{\Phi}-\bm{Q}_{k-1}\bm{\Phi}))_{\mathcal{T}_h}, \\
	I_2 &=  -(u-u_h,\nabla_h\cdot (\bm{\Phi}-\bm{Q}_{k-1}\bm{\Phi}))_{\mathcal{T}_h} + \langle u-\hat{u}_h,(\bm{\Phi}-\bm{Q}_{k-1}\bm{\Phi})\cdot \bm{n}\rangle_{\partial\mathcal{T}_h\setminus \mathcal{\varepsilon}_h^{\Gamma}} 
	+\langle (\bm{\Phi}-\bm{Q}_{k-1}\bm{\Phi}), (u-\tilde{u}_h)\bm{n}\rangle_{*,\Gamma},  \\
	I_3 &=  (\nabla_h\cdot(\bm{q}-\bm{q}_h),\phi-Q_{k}\phi)_{\mathcal{T}_h}  + \langle (\bm{q}-\bm{q}_h)\cdot\bm{n},\phi-Q_{k}^b\phi\rangle_{\partial \mathcal{T}_h\setminus (\partial \Omega\cup \mathcal{\varepsilon}_h^{\Gamma})} +\langle \bm{q}-\bm{q}_h,(\phi-\{\{Q_{k}^{\Gamma}\phi\}\})\bm{n}\rangle_{*,\Gamma},  \\
	I_4 &= \langle \tau(\hat{u}_h-u_h),Q_{k}^b\phi-Q_{k}\phi
	\rangle_{\partial\mathcal{T}_h\setminus \mathcal{\varepsilon}_h^{\Gamma}} + \langle
	\eta(\hat{u}_h-u_h),\{\{Q_{k}^{\Gamma}\phi\}\}-Q_{k}\phi\rangle_{*,\Gamma} ,\\
	I_5 &= -\langle \bm{\Phi},(u-\tilde{u}_h)\bm{n}\rangle_{*,\Gamma}-\langle \bm{\Phi}\cdot\bm{n},u-\hat{u}_h\rangle_{\partial\Omega},
	\end{align*}
and	we recall that  $\{\{\cdot\}\}$ and  $Q_{k}^{\Gamma}$ are defined in \eqref{average} and \eqref{Q^Ga}, respectively. 

The thing left is to  estimate the terms $I_i$ one by one. 
The Cauchy-Schwarz inequality and the projection property (cf. Lemma \ref{ineq}) indicate
	\begin{align*}
	I_1 \leq \lVert \alpha^{-1}(\bm{q}-\bm{q}_h)\rVert_{0,\mathcal{T}_h}  \lVert \bm{\Phi}-\bm{Q}_{k-1}\bm{\Phi}\rVert_{0,\mathcal{T}_h} \apprle h\lVert \alpha^{-1}(\bm{q}-\bm{q}_h)\rVert_{0,\mathcal{T}_h}  \lVert \bm{\Phi}\rVert_{1,\Omega_1\cup\Omega_2}.
	\end{align*}
 Similarly,    from integration by parts it follows
   \begin{align*}
   I_2
   =& (\nabla_h(u-u_h), (\bm{\Phi}-\bm{Q}_{k-1}\bm{\Phi}))_{\mathcal{T}_h} + \langle u_h-\hat{u}_h,(\bm{\Phi}-\bm{Q}_{k-1}\bm{\Phi})\cdot \bm{n}\rangle_{\partial\mathcal{T}_h\setminus \mathcal{\varepsilon}_h^{\Gamma}} 
   +\langle (\bm{\Phi}-\bm{Q}_{k-1}\bm{\Phi}), (u_h-\tilde{u}_h)\bm{n}\rangle_{*,\Gamma}  \\
   \apprle& h\lVert \nabla_h(u-u_h)\rVert_{0,\mathcal{T}_h} \lVert \bm{\Phi}\rVert_{1,\Omega_1\cup\Omega_2} +  h\interleave  \alpha^{-1/2} (\bm{e}_h^q,e_h^u,e_h^{\hat{u}},e_h^{\tilde{\mu}})\interleave \lVert \bm{\Phi}\rVert_{1,\Omega_1\cup\Omega_2},\\
    I_3 
    =& -(\bm{q}-\bm{q}_h,\nabla_h(\phi-Q_{k}\phi))_{\mathcal{T}_h} +\langle (\bm{q}-\bm{q}_h)\cdot\bm{n},Q_{k}\phi-Q_{k}^b\phi\rangle_{\partial \mathcal{T}_h\setminus (\partial \Omega\cup \mathcal{\varepsilon}_h^{\Gamma})} +\langle \bm{q}-\bm{q}_h,(Q_{k}\phi-\{\{Q_{k}^{\Gamma}\phi\}\})\bm{n}\rangle_{*,\Gamma}  \\
   \apprle & h\lVert \alpha^{-1}(\bm{q}-\bm{q}_h)\rVert_{0,\mathcal{T}_h}  \lVert \alpha \phi\rVert_{2,\Omega_1\cup\Omega_2}.
    \end{align*}
And, by the definitions of $e_h^u$ and $ e_h^{\hat{u}}$ in \eqref{e_h^q}, we obtain
    \begin{align*}
    I_4 
    =&\langle \tau(e_h^u-e_h^{\hat{u}}),Q_{k}^b\phi-Q_{k}\phi\rangle_{\partial\mathcal{T}_h\setminus \mathcal{\varepsilon}_h^{\Gamma}} + \langle \tau((u-Q_ku)-(u-Q_{k}^bu)),Q_{k}^b\phi-Q_{k}\phi \rangle_{\partial\mathcal{T}_h\setminus \mathcal{\varepsilon}_h^{\Gamma}} \\
    &+\langle \eta(e_h^u-e_h^{\hat{u}}),\{\{Q_{k}^{\Gamma}\phi\}\}-Q_{k}\phi \rangle_{*,\Gamma} + \langle \eta((u-Q_ku)-(u-Q_{k}^{\Gamma}u)),\{\{Q_{k}^{\Gamma}\phi\}\}-Q_{k}\phi\rangle_{*,\Gamma} \\
    \apprle & h^{-1/2}\lVert \alpha(Q_{k}^b\phi-Q_{k}\phi)\rVert_{\partial\mathcal{T}_h\setminus \mathcal{\varepsilon}_h^{\Gamma}} \left(h^{-1/2}\lVert u-Q_ku\rVert_{\partial\mathcal{T}_h\setminus \mathcal{\varepsilon}_h^{\Gamma}} +  \lVert \alpha^{-1/2}\tau^{1/2}(e_h^u-e_h^{\hat{u}})\rVert_{\partial\mathcal{T}_h\setminus \mathcal{\varepsilon}_h^{\Gamma}}  \right) \\
    &+ h^{-1/2}\lVert \alpha(\{\{Q_{k}^{\Gamma}\phi\}\}-Q_{k}\phi)\rVert_{*,\Gamma} \left( h^{-1/2}\lVert u-Q_ku\rVert_{*,\Gamma} +  \lVert \alpha^{-1/2}\eta^{1/2}(e_h^u-e_h^{\hat{u}})\rVert_{*,\Gamma} \right) \\
    \apprle & h\lVert \alpha \phi\rVert_{2,\Omega_1\cup\Omega_2} \left(\interleave  \alpha^{-1/2} (\bm{e}_h^q,e_h^u,e_h^{\hat{u}},e_h^{\tilde{\mu}})\interleave +h^k\lVert u\rVert_{k+1,\Omega_1\cup\Omega_2}\right).
    \end{align*}
 Recall that $\Gamma_{K,h}$ is the straight line/plane segment connecting the intersection between $\Gamma_K$ and $\partial K$.  To estimate $I_5$, we   assume  $\bm{n}_c $ to be the unit normal vector along  $\Gamma_{K,h}$.  Note that for case (1) it holds that 
 \begin{align}\label{ncn}
 \Gamma_{K,h}= \Gamma_{K}, \quad \bm{n}_c=\bm{n}   .
 \end{align}
  Thus, by \eqref{<n>_G}  we get
    \begin{align*}
    I_5 =& -\langle \bm{\Phi}-\{\{\bm{Q}_{k-1}^{\Gamma}\bm{\Phi}\}\},(u-\tilde{u}_h)\bm{n}\rangle_{*,\Gamma} -\langle (\bm{\Phi}-\bm{Q}_{k-1}\bm{\Phi})\cdot\bm{n},(u-\hat{u}_h)\rangle_{\partial\Omega}\\
    &\quad -\langle \{\{\bm{Q}_{k-1}^{\Gamma}\bm{\Phi}\}\},(u-\tilde{u}_h)\bm{n}\rangle_{*,\Gamma} -\langle \bm{Q}_{k-1}\bm{\Phi}\cdot\bm{n},(u-\hat{u}_h)\rangle_{\partial\Omega} \\
    =&-\langle \bm{\Phi}-\{\{\bm{Q}_{k-1}^{\Gamma}\bm{\Phi}\}\},(u-\tilde{u}_h)\bm{n}\rangle_{*,\Gamma} -\langle \{\{\bm{Q}_{k-1}^{\Gamma}\bm{\Phi}\}\},(u-\tilde{u}_h)(\bm{n}-\bm{n}_c)\rangle_{*,\Gamma} -\langle (\bm{\Phi}-\bm{Q}_{k-1}\bm{\Phi})\cdot\bm{n},(u-\hat{u}_h)\rangle_{\partial\Omega}\\
    &\quad -\langle \{\{\bm{Q}_{k-1}^{\Gamma}\bm{\Phi}\}\},(u-\tilde{u}_h)\bm{n}_c\rangle_{*,\Gamma} -\langle \bm{Q}_{k-1}\bm{\Phi}\cdot\bm{n},(u-\hat{u}_h)\rangle_{\partial\Omega} \\
      =&-\langle \bm{\Phi}-\{\{\bm{Q}_{k-1}^{\Gamma}\bm{\Phi}\}\},(u-\tilde{u}_h)\bm{n}\rangle_{*,\Gamma} -\langle \{\{\bm{Q}_{k-1}^{\Gamma}\bm{\Phi}\}\},(u-\tilde{u}_h)(\bm{n}-\bm{n}_c)\rangle_{*,\Gamma} -\langle (\bm{\Phi}-\bm{Q}_{k-1}\bm{\Phi})\cdot\bm{n},(u-\hat{u}_h)\rangle_{\partial\Omega}\\
    &\quad -\langle \{\{\bm{Q}_{k-1}^{\Gamma}\bm{\Phi}\}\},(u-\tilde{u}_h)\bm{n}_c\rangle_{*,\Gamma},   
    \end{align*}
where we have used the fact that $ \langle \bm{Q}_{k-1}\bm{\Phi}\cdot\bm{n},(u-\hat{u}_h)\rangle_{\partial\Omega} =0$ due to the boundary condition \eqref{2.2c}  and the definition of $M_h(g)$.  On one hand,  from \eqref{ncn} and \eqref{gamma} it follows
 $$\langle \{\{\bm{Q}_{k-1}^{\Gamma}\bm{\Phi}\}\},(u-\tilde{u}_h)(\bm{n}-\bm{n}_c)\rangle_{*,\Gamma}\left\{
 \begin{array}{ll}
 =0 & \text{ for case (1)} ,\\
  \apprle   h\lVert \bm{\Phi}\rVert_{1,\Omega_{1}\cup\Omega_2} \interleave \alpha^{-1/2} (\bm{e}_h^q,e_h^u,e_h^{\hat{u}},e_h^{\tilde{\mu}})\interleave  & \text{otherwise.}
 \end{array}
  \right. $$
 On the other hand,  the relation
    $$ \langle \{\{\bm{Q}_{k-1}^{\Gamma}\bm{\Phi}\}\},(u-\tilde{u}_h)\bm{n}_c\rangle_{*,\Gamma}=0$$
    holds  for case (1) due to the definition of $\tilde M_h(g_D)$ and  the interface condition \eqref{2.2d}, and 
 for case (2)  due to    $g_D=0$.   As a result, by the projection property we obtain
   \begin{align*}
    I_5 
   \apprle & h\lVert \bm{\Phi}\rVert_{1,\Omega_{1}\cup\Omega_2}(h^k\lVert \alpha^{1/2}u\rVert_{k+1,\Omega_{1}\cup\Omega_2}+\interleave \alpha^{-1/2} (\bm{e}_h^q,e_h^u,e_h^{\hat{u}},e_h^{\tilde{\mu}})\interleave).
    \end{align*}

    The above estimates of $I_i$, together with the results \eqref{est_add0}-\eqref{est_add2} and the regularity assumption (\ref{regularestimate}), yield the desired conclusion \eqref{estimategeneralinterface}.
\end{proof}

		For the modified scheme (\ref{X-HDGscheme1}) where the interface $\Gamma$ is a fold line/plane, we  replace the second equation in the auxiliary  problem (\ref{dual problem}) with 
		\begin{equation}\label{rme1}
		- \nabla\cdot \bm{\Phi} = e_h^u \quad  {\rm in} \ \Omega_1\cup\Omega_2,
		\end{equation}
	and assume the modified regularity assumption
		\begin{align}\label{regularestimate1}
		\lVert \bm{\Phi}\rVert_{1,\Omega_1\cup\Omega_2} +\lVert \alpha\phi\rVert_{2,\Omega_1\cup\Omega_2} \apprle \lVert e_h^u \rVert_{0,\mathcal{T}_h}.
		\end{align}
		 We can follow the same  line as in the proof of Theorem \ref{L2} to derive  the following $L^2$ error estimation for  the modified scheme \eqref{X-HDGscheme1}. 
		   \begin{thm}	\label{L2-modified}
		   	Let $(u,\bm{q})\in H^{k+1}(\Omega_1\cup \Omega_2)\times H^k(\Omega_1\cup \Omega_2)^d $ and $(\bm{q}_h,u_h,\hat{u}_h,\tilde{u}_h)\in \bm{W}_h\times V_h\times M_h(g)\times \tilde{M}_h$ be the solutions of the  problem (\ref{firstorderscheme}) and the modified X-HDG scheme (\ref{X-HDGscheme1}), respectively.  Under the regularity assumption \eqref{regularestimate1},  it holds
		   	\begin{align}\label{err_L2modified}
		   	\lVert u-u_h\rVert_{0,\mathcal{T}_h} \apprle \left\{ \begin{array}{ll}
		   	(\frac{\alpha_{max}}{\alpha_{min}})^{1/2}h^{k+1} \lVert u\rVert_{H^{k+1}(\Omega_1\cup\Omega_2)} & \text{ if } g_N=0,\\\\
		   	(\frac{\alpha_{max}}{\alpha_{min}})^{1/2}h^{k} \lVert u\rVert_{H^{k+1}(\Omega_1\cup\Omega_2)} &\text{ otherwise},
		   	\end{array}
		   	\right.
		   	\end{align}
		   	where the interface $\Gamma$ is a fold line/plane.
		   \end{thm}
	   \begin{proof}
		   In fact, from the modified  auxiliary  problem with \eqref{rme1}, the definitions of $L_2(\cdot) $ and $ L_3(\cdot)$, and the projection properties we can similarly derive
			\begin{align*}
		\lVert e_h^u\rVert_{0,\mathcal{T}_h}^2
		=&(\alpha^{-1}\bm{e}_h^q,\bm{Q}_{k-1}\bm{\Phi}-\bm{\Phi})_{\mathcal{T}_h}+ L_2(Q_{k}\phi-Q_{k-1}^b\phi)+L_3(Q_{k}\phi-Q_{k-1}^b\phi)  \\
		&+\langle e_h^u-e_h^{\hat{u}},(\bm{Q}_{k-1}\bm{\Phi}-\bm{\Phi})\cdot \bm{n}\rangle_{\partial\mathcal{T}_h\setminus \mathcal{\varepsilon}_h^{\Gamma}} +\langle (e_h^u-e_h^{\tilde{u}})\bm{n},(\bm{Q}_{k-1}\bm{\Phi}-\bm{\Phi})\rangle_{*,\Gamma} \\
		&+\langle \tau(Q_{k-1}^be_h^u-e_h^{\hat{u}}),Q_{k-1}^b\phi-Q_{k}\phi\rangle_{\partial\mathcal{T}_h\setminus \mathcal{\varepsilon}_h^{\Gamma}}+\langle \eta(Q_{k-1}^be_h^u-e_h^{\hat{u}}),Q_{k-1}^b\phi-Q_{k}\phi\rangle_{*,\Gamma} \\
		=& \sum_{i=1}^{4}I_i,
		\end{align*}
		where
		\begin{align*}
		I_1 =&  (\alpha^{-1}\bm{e}_h^q,\bm{Q}_{k-1}\bm{\Phi}-\bm{\Phi})_{\mathcal{T}_h}, \\
		I_2 =& \langle (\bm{Q}_{k-1}\bm{q}-\bm{q})\cdot\bm{n},Q_{k}\phi - Q_{k-1}^b\phi \rangle_{\partial\mathcal{T}_h\setminus \mathcal{\varepsilon}_h^{\Gamma}}+\langle \tau Q_{k-1}^b(u-Q_{k}u),Q_{k}\phi - Q_{k-1}^b\phi \rangle_{\partial\mathcal{T}_h\setminus \mathcal{\varepsilon}_h^{\Gamma}} \\
		&+ \langle (\bm{Q}_{k-1}\bm{q}-\bm{q}),(Q_{k}\phi - Q_{k-1}^b\phi)\bm{n} \rangle_{*,\Gamma}+ \langle \eta Q_{k-1}^b(u-Q_{k}u),Q_{k}\phi - Q_{k-1}^b\phi \rangle_{*,\Gamma}, \\
		I_3 =& \langle e_h^u-e_h^{\hat{u}},(\bm{Q}_{k-1}\bm{\Phi}-\bm{\Phi})\cdot \bm{n}\rangle_{\partial\mathcal{T}_h\setminus \mathcal{\varepsilon}_h^{\Gamma}} +\langle (e_h^u-e_h^{\tilde{u}})\bm{n},(\bm{Q}_{k-1}\bm{\Phi}-\bm{\Phi})\rangle_{*,\Gamma}, \\
		I_4 =& \langle \tau(Q_{k-1}^be_h^u-e_h^{\hat{u}}),Q_{k-1}^b\phi-Q_{k}\phi\rangle_{\partial\mathcal{T}_h\setminus \mathcal{\varepsilon}_h^{\Gamma}}+\langle \eta(Q_{k-1}^be_h^u-e_h^{\hat{u}}),Q_{k-1}^b\phi-Q_{k}\phi\rangle_{*,\Gamma}.
		\end{align*}
		By the definition of $\interleave \cdot\interleave$, it is easy to obtain 
		\begin{align*}
		I_j 
		\apprle 
		&h  \lVert \bm{\Phi} \rVert_{1,\Omega_{1}\cup\Omega_{2}} \interleave \alpha^{-1/2} (\bm{e}_h^q,e_h^u,e_h^{\hat{u}},e_h^{\tilde{\mu}})\interleave , \ j=1,3,\\
			I_4 \apprle& h\lVert \alpha\phi \rVert_{2,\Omega_{1}\cup\Omega_{2}}  \interleave  \alpha^{-1/2} (\bm{e}_h^q,e_h^u,e_h^{\hat{u}},e_h^{\tilde{\mu}})\interleave.
		\end{align*}
		In light of the orthogonal property of projection $Q_{k-1}^b$,  it holds
	\begin{align*}
		I_2  
		= & \langle (\bm{Q}_{k-1}\bm{q}-\bm{q})\cdot\bm{n},Q_{k}\phi - Q_{k-1}^b\phi \rangle_{\partial\mathcal{T}_h\setminus \mathcal{\varepsilon}_h^{\Gamma}}
		+ \langle (\bm{Q}_{k-1}\bm{q}-\bm{q}),(Q_{k}\phi - Q_{k-1}^b\phi)\bm{n} \rangle_{*,\Gamma} \\
		&+\langle \tau Q_{k-1}^b(u-Q_{k}u),Q_{k}\phi - \phi \rangle_{\partial\mathcal{T}_h\setminus \mathcal{\varepsilon}_h^{\Gamma}}+ \langle \eta Q_{k-1}^b(u-Q_{k}u),Q_{k}\phi - \phi \rangle_{*,\Gamma}. \\
		 = &\left(\langle (\bm{Q}_{k-1}\bm{q}-\bm{q})\cdot\bm{n},Q_{k}\phi - \phi \rangle_{\partial\mathcal{T}_h\setminus \mathcal{\varepsilon}_h^{\Gamma}}
		 + \langle (\bm{Q}_{k-1}\bm{q}-\bm{q}),(Q_{k}\phi - \phi)\bm{n} \rangle_{*,\Gamma}\right) \\
		  &+ \langle (\bm{Q}_{k-1}\bm{q}-\bm{q})\cdot\bm{n},\phi - Q_{k-1}^b\phi \rangle_{\partial\mathcal{T}_h\setminus \mathcal{\varepsilon}_h^{\Gamma}}\\
		 & + \langle (\bm{Q}_{k-1}\bm{q}-\bm{q}),(\phi - Q_{k-1}^b\phi)\bm{n} \rangle_{*,\Gamma} \\
		  &+\left(\langle \tau Q_{k-1}^b(u-Q_{k}u),Q_{k}\phi - \phi \rangle_{\partial\mathcal{T}_h\setminus \mathcal{\varepsilon}_h^{\Gamma}}+ \langle \eta Q_{k-1}^b(u-Q_{k}u),Q_{k}\phi - \phi \rangle_{*,\Gamma}\right). \\
		  =:& \tilde I_1 +\tilde I_2 +\tilde I_3+\tilde I_4.
		  \end{align*}
	From the projection properties it follows
		  \begin{align*}
		  \tilde I_1=&\langle (\bm{Q}_{k-1}\bm{q}-\bm{q})\cdot\bm{n},Q_{k}\phi - \phi \rangle_{\partial\mathcal{T}_h\setminus \mathcal{\varepsilon}_h^{\Gamma}}
		 + \langle (\bm{Q}_{k-1}\bm{q}-\bm{q}),(Q_{k}\phi - \phi)\bm{n} \rangle_{*,\Gamma}\\
		  \leq&  \left(\lVert \alpha^{-1}(\bm{Q}_{k-1}\bm{q}-\bm{q})\rVert_{0,\partial\mathcal{T}_h\setminus \mathcal{\varepsilon}_h^{\Gamma}} +\lVert \alpha^{-1}(\bm{Q}_{k-1}\bm{q}-\bm{q})\rVert_{*,\Gamma}\right)\left(\lVert \alpha(\phi-Q_k\phi)\rVert_{0,\partial\mathcal{T}_h\setminus \mathcal{\varepsilon}_h^{\Gamma}} +\lVert \alpha(\phi-Q_k\phi)\rVert_{*,\Gamma}\right) \\
		  \apprle &h^{k+1}\lVert \alpha^{-1}\bm{q} \rVert_{k,\Omega_{1}\cup\Omega_{2}}  \lVert \alpha\phi\rVert_{2,\Omega_{1}\cup\Omega_{2}}.
		  \end{align*}
By the continuity of $\bm{q}\cdot\bm{n}$ on $F\in {\partial\mathcal{T}_h\setminus \mathcal{\varepsilon}_h^{\Gamma}}$ and the  orthogonal  property of projection $Q_{k-1}^b$, we easily get  
$$\tilde I_2 =\langle (\bm{Q}_{k-1}\bm{q}-\bm{q})\cdot\bm{n},\phi - Q_{k-1}^b\phi \rangle_{\partial\mathcal{T}_h\setminus \mathcal{\varepsilon}_h^{\Gamma}}= 0.$$
		  Similarly, if $g_N = 0$, then we have  the continuity of $\bm{q}\cdot\bm{n}$ on $F\in \varepsilon_h^*$, which implies
		  $$\tilde I_3=\langle (\bm{Q}_{k-1}\bm{q}-\bm{q}),(\phi - Q_{k-1}^b\phi)\bm{n} \rangle_{*,\Gamma} =0;$$
	Otherwise, it holds 
		  \begin{align*}
		  \tilde I_3 \leq&  \lVert \alpha^{-1}(\bm{Q}_{k-1}\bm{q}-\bm{q})\rVert_{*,\Gamma}\lVert \alpha(\phi-Q_{k-1}^b\phi)\rVert_{*,\Gamma}) \\
		  \apprle &h^{k}\lVert \alpha^{-1}\bm{q} \rVert_{k,\Omega_{1}\cup\Omega_{2}}  \lVert \alpha\phi\rVert_{2,\Omega_{1}\cup\Omega_{2}}.
		  \end{align*}
		  For the term $\tilde I_4$,  we similarly have 
		  \begin{align*}
		  \tilde I_4 \leq&  h^{k+1}\lVert u \rVert_{k+1,\Omega_{1}\cup\Omega_{2}}  \lVert \alpha\phi\rVert_{2,\Omega_{1}\cup\Omega_{2}}.
		  \end{align*}
		Combining the above estimates of $\tilde I_j$  ($j=1,2,3,4$) yields 
		$$I_2= \tilde I_1 +\tilde I_2 +\tilde I_3+\tilde I_4\apprle \left\{ \begin{array}{ll}
		h^{k+1}\lVert \alpha^{-1}\bm{q} \rVert_{k,\Omega_{1}\cup\Omega_{2}}  \lVert \alpha\phi\rVert_{2,\Omega_{1}\cup\Omega_{2}} & \text{ if } g_N=0,\\
		h^{k}\lVert \alpha^{-1}\bm{q} \rVert_{k,\Omega_{1}\cup\Omega_{2}}  \lVert \alpha\phi\rVert_{2,\Omega_{1}\cup\Omega_{2}} &\text{ otherwise.} 
		\end{array}
		\right.$$
		As a result, the estimates of $ I_i$  ($i=1,2,3,4$), together with the modified regularity assumption \eqref{regularestimate1}
and Theorem \ref{estimateofenergynorm}, yields the results \eqref{err_L2modified}.
		\end{proof}

\section{Numerical experiments}
In this section, we shall provide several numerical examples to verify the performance of the proposed   X-HDG schemes (\ref{X-HDGscheme}) and (\ref{X-HDGscheme1}). We recall that the modified scheme (\ref{X-HDGscheme1}) is only for the case that the interface $\Gamma$ is a fold line/plane. 

\begin{exmp}\label{circle} \textbf{Circular interface with homogeneous jump conditions} (cf. \cite{Huynh2013A}).  

Set   $\Omega = [0,1]\times [0,1]$ in the model problem \eqref{pb1} with a circular interface (Figure \ref{circledomain}). 
	The exact solution is given by
	\begin{align*}
	u(x,y) = \left \{
	\begin{array}{rl}
	\frac{r^5}{\alpha_2},  \quad \quad \quad \quad \quad \,{\rm if} \  r < r_0,\\
	\frac{r^5}{\alpha_1}-\frac{r_0^5}{\alpha_1}+\frac{r_0^5}{\alpha_2},  \,{\rm if} \   r > r_0 ,
	\end{array}
	\right.
	\end{align*}
	where $r = \sqrt{(x-1/2)^2+(y-1/2)^2}$ and  $r_0 = \sqrt{3}/8$.  It is easy to know that the homogeneous jump conditions  $g_D = g_N = 0$ hold. Different the coefficient $\alpha$, we consider four cases: $\alpha_1=10, \alpha_2=1; $  $\alpha_1=1, \alpha_2=10; $ $\alpha_1=1000, \alpha_2=1; $ $\alpha_1=1, \alpha_2=1000 $.  
\end{exmp}

We use $N\times N$ uniform  triangular meshes  for the computation (cf. Figure \ref{circledomain}). Tables \ref{numericalresultcircle11} and  \ref{numericalresultcircle21} list    the  results of the relative errors between $(u,\bm{q} = \alpha \nabla u)$ and $(u_h,\bm{q}_h)$ at different coefficients, and  Figures \ref{1000:1}-\ref{1:1000} show 
 the numerical solutions $u_h$  at $128\times 128$ mesh with $k=1$ and  $\alpha_1:\alpha_2 = 1000:1,1:1000$.
 We can see that  for $k=1,2$  the X-HDG method yields optimal convergence orders, i.e. $k$-th order  rates of convergence for  the error $\lVert \nabla u-\nabla_h u_h \rVert_0, \lVert \bm{q}-\bm{q}_h \rVert_0$, and $(k+1)$-th order rates of convergence for $\lVert u-u_{h}\rVert_0$. This is consistent with our theoretical results in Theorems \ref{estimateofenergynorm}-\ref{L2}. 

\begin{figure}[htbp]\label{circledomain}	
	\centering	
	\includegraphics[height = 4.3 cm,width=5.5 cm]{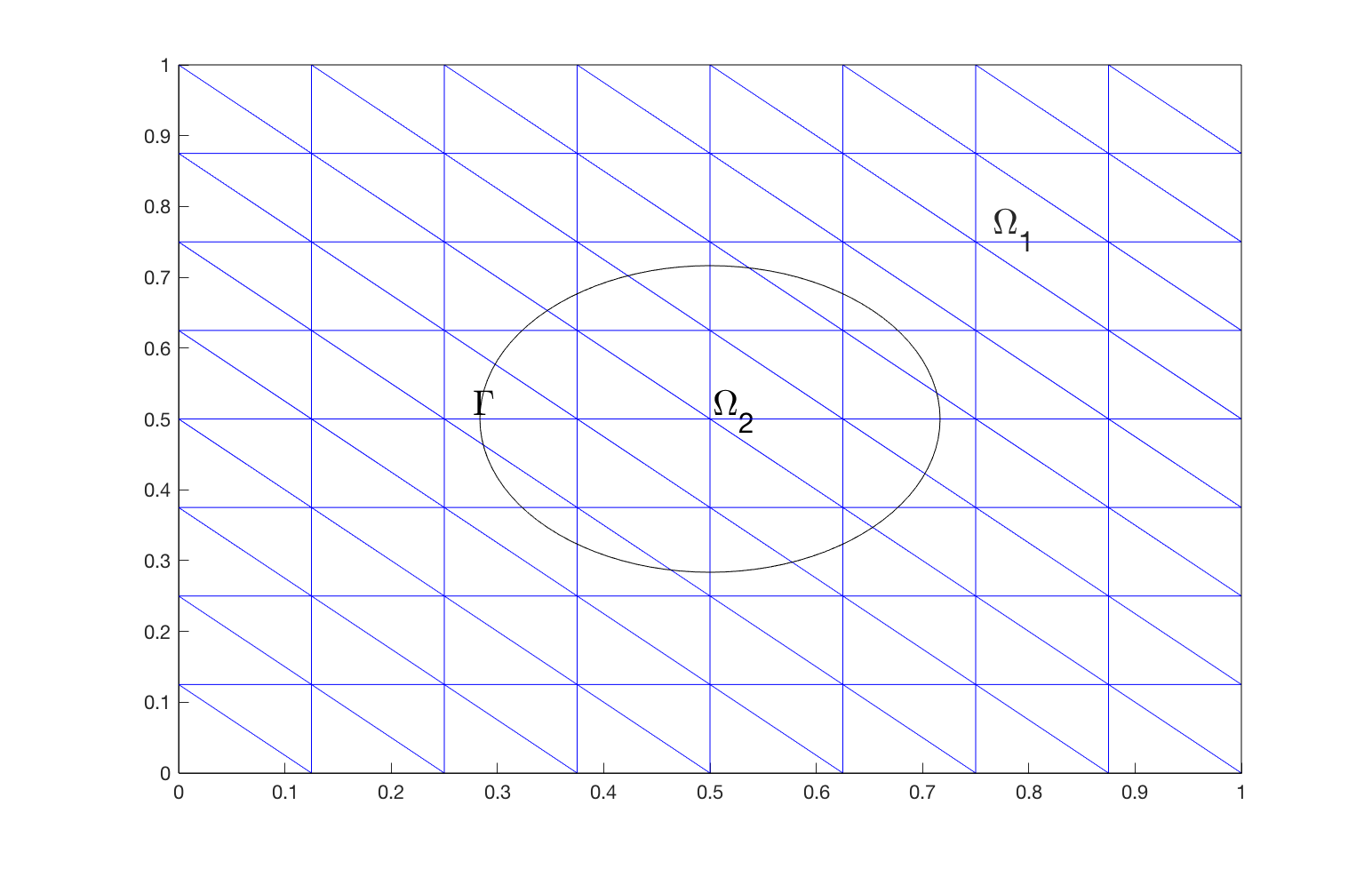} 				
	\caption{The domain with a circular interface: $8\times 8$ mesh}
\end{figure}

\begin{table}[H]
	\normalsize
	\caption{History of convergence for   Example \ref{circle}:   X-HDG scheme (\ref{X-HDGscheme})  with  $k = 1$}\label{numericalresultcircle11}
	\centering
	\footnotesize
	\subtable[ $\alpha_1:\alpha_2=10:1$]{
		\begin{tabular}{p{1.15cm}<{\centering}|p{1.2cm}<{\centering}|p{0.45cm}<{\centering}|p{1.2cm}<{\centering}|p{0.45cm}<{\centering}|p{1.2cm}<{\centering}|p{0.45cm}<{\centering}}
			\hline   
			\multirow{2}{*}{mesh}
			&\multicolumn{2}{c|}{$\frac{\lVert u-u_{h}\rVert_0}{\lVert u\rVert_0}$ }&\multicolumn{2}{c|}{$\frac{\lVert \bm{q}-\bm{q}_h \rVert_0}{\lVert \bm{q}\rVert_0}$}&\multicolumn{2}{c}{$\frac{\lVert \nabla u-\nabla_h u_h \rVert_0}{\lVert \nabla u\rVert_0}$}\cr\cline{2-7}  
			&error&order&error&order&error&order\cr  
			\cline{1-7}
			$8\times 8$	    &1.50E-01    &	--    &2.27E-01	  &  --   &2.93E-01	  &  --   \\
			\hline
			$16\times 16$	 &3.76E-02   &1.99	&1.16E-01	  &0.97  &1.37E-01	  &1.10	 \\
			\hline
			$32\times 32$	&9.44E-03   &2.00	   &5.84E-02	  &0.99	  &6.74E-02	  &1.03  \\
			\hline
			$64\times 64$	&2.36E-03    &2.00	&2.93E-02	  &1.00		 &3.36E-02	  &1.01 \\
			\hline
			$128\times 128$	&5.90E-04    &2.00	&1.46E-02	  &1.00		&1.68E-02	  &1.00  \\
			\hline
		\end{tabular}
	}
\subtable[ $\alpha_1:\alpha_2=1:10$ ]{
	\begin{tabular}{p{1.15cm}<{\centering}|p{1.2cm}<{\centering}|p{0.45cm}<{\centering}|p{1.2cm}<{\centering}|p{0.45cm}<{\centering}|p{1.2cm}<{\centering}|p{0.45cm}<{\centering}}
		\hline   
		\multirow{2}{*}{mesh}
		&\multicolumn{2}{c|}{$\frac{\lVert u-u_{h}\rVert_0}{\lVert u\rVert_0}$ }&\multicolumn{2}{c|}{$\frac{\lVert \bm{q}-\bm{q}_h \rVert_0}{\lVert \bm{q}\rVert_0}$}&\multicolumn{2}{c}{$\frac{\lVert \nabla u-\nabla_h u_h \rVert_0}{\lVert \nabla u\rVert_0}$}\cr\cline{2-7}  
		&error&order&error&order&error&order\cr  
		\cline{1-7}
		$8\times 8$	    &1.63E-01    &	--    &2.25E-01    &--   &2.87E-01	  &   --   \\
		\hline
		$16\times 16$	 &4.09E-02   &2.00	&1.15E-01  &0.97	 &1.35E-01  &1.08\\
		\hline
		$32\times 32$	&1.02E-02    &2.00	     &5.76E-02	   &0.99    &6.65E-02	  &1.02  \\
		\hline
		$64\times 64$	&2.56E-03    &2.00	&2.88E-02	  &1.00		 &3.31E-02	  &1.01  \\
		\hline
		$128\times 128$	&6.39E-04    &2.00	&1.44E-02	  &1.00  &1.65E-02	  &1.00 \\
		\hline
	\end{tabular}
}
	\subtable[ $\alpha_1:\alpha_2=1000:1$]{
		\begin{tabular}{p{1.15cm}<{\centering}|p{1.2cm}<{\centering}|p{0.45cm}<{\centering}|p{1.2cm}<{\centering}|p{0.45cm}<{\centering}|p{1.2cm}<{\centering}|p{0.45cm}<{\centering}}
			\hline   
			\multirow{2}{*}{mesh}
			&\multicolumn{2}{c|}{$\frac{\lVert u-u_{h}\rVert_0}{\lVert u\rVert_0}$ }&\multicolumn{2}{c|}{$\frac{\lVert \bm{q}-\bm{q}_h \rVert_0}{\lVert \bm{q}\rVert_0}$}&\multicolumn{2}{c}{$\frac{\lVert \nabla u-\nabla_h u_h \rVert_0}{\lVert \nabla u\rVert_0}$}\cr\cline{2-7}  
			&error&order&error&order&error&order\cr  
			\cline{1-7}
			$8\times 8$	    &1.99E-01    &	--    &4.45E-01	  &  --  	 &8.79E-01  & --  \\
			\hline
			$16\times 16$	 &5.27E-02   &1.93	&2.70E-01	  &0.72	  &3.60E-01	  &1.29	 \\
			\hline
			$32\times 32$	&1.42E-02   &1.90	   &1.51E-01	  &0.84	  &1.70E-01	  &1.08    \\
			\hline
			$64\times 64$	&3.57E-03    &1.99	&7.92E-02	  &0.93   &8.01E-02	  &1.09 \\
			\hline
			$128\times 128$	&8.98E-04    &1.99	&4.06E-02	  &0.97	  	&4.01E-02	  &1.00  \\
			\hline
		\end{tabular}
	}
	\subtable[$\alpha_1:\alpha_2=1:1000$]{
\begin{tabular}{p{1.15cm}<{\centering}|p{1.2cm}<{\centering}|p{0.45cm}<{\centering}|p{1.2cm}<{\centering}|p{0.45cm}<{\centering}|p{1.2cm}<{\centering}|p{0.45cm}<{\centering}}
	\hline   
	\multirow{2}{*}{mesh}
	&\multicolumn{2}{c|}{$\frac{\lVert u-u_{h}\rVert_0}{\lVert u\rVert_0}$ }&\multicolumn{2}{c|}{$\frac{\lVert \bm{q}-\bm{q}_h \rVert_0}{\lVert \bm{q}\rVert_0}$}&\multicolumn{2}{c}{$\frac{\lVert \nabla u-\nabla_h u_h \rVert_0}{\lVert \nabla u\rVert_0}$}\cr\cline{2-7}  
	&error&order&error&order&error&order\cr  
	\cline{1-7}
		$8\times 8$	    &1.63E-01    &	--    &2.25E-01	  &   --     &2.87E-01	  &   --   \\
		\hline
		$16\times 16$	 &4.09E-02   &2.00	&1.15E-01	  &0.97	   &1.35E-01  &1.08\\
		\hline
		$32\times 32$	&1.02E-02    &2.00	&5.76E-02	  &0.99  &6.65E-02	  &1.02  \\
		\hline
		$64\times 64$	&2.56E-03    &2.00	&2.88E-02	  &1.00	 &3.31E-02	  &1.01  \\
		\hline
		$128\times 128$	&6.40E-04    &2.00	&1.44E-02	  &1.00	 &1.65E-02	  &1.00 \\
		\hline
	\end{tabular}
}
\end{table}

\begin{table}[H]
	\normalsize
	\caption{History of convergence   for Example \ref{circle}:   X-HDG scheme (\ref{X-HDGscheme})  with  $k = 2$ }\label{numericalresultcircle21}
	\centering
	\footnotesize
	\subtable[ $\alpha_1:\alpha_2=10:1$]{
		\begin{tabular}{p{1.15cm}<{\centering}|p{1.2cm}<{\centering}|p{0.45cm}<{\centering}|p{1.2cm}<{\centering}|p{0.45cm}<{\centering}|p{1.2cm}<{\centering}|p{0.45cm}<{\centering}}
			\hline   
			\multirow{2}{*}{mesh}
			&\multicolumn{2}{c|}{$\frac{\lVert u-u_{h}\rVert_0}{\lVert u\rVert_0}$ }&\multicolumn{2}{c|}{$\frac{\lVert \bm{q}-\bm{q}_h \rVert_0}{\lVert \bm{q}\rVert_0}$}&\multicolumn{2}{c}{$\frac{\lVert \nabla u-\nabla_h u_h \rVert_0}{\lVert \nabla u\rVert_0}$}\cr\cline{2-7}  
			&error&order&error&order&error&order\cr  
			\cline{1-7}
			$8\times 8$	    &1.41E-02    &	--    &2.34E-02	  &  --   &8.92E-02	  &  --   \\
			\hline
			$16\times 16$	 &1.84E-03   &2.94	&6.20E-03	  &1.92  &2.28E-02	  &1.97	 \\
			\hline
			$32\times 32$	&2.35E-04   &2.97	 &1.59E-03	  &1.96	  &5.78E-03	  &1.98   \\
			\hline
			$64\times 64$	&2.96E-05    &2.99	&4.04E-04	  &1.98	   &1.46E-03	&1.99  \\
			\hline
			$128\times 128$	&3.72E-06    &2.99	&1.02E-04	  &1.99  	&3.66E-04	  &1.99  \\
			\hline
		\end{tabular}
	}
	\subtable[$\alpha_1:\alpha_2=1:10$]{
\begin{tabular}{p{1.15cm}<{\centering}|p{1.2cm}<{\centering}|p{0.45cm}<{\centering}|p{1.2cm}<{\centering}|p{0.45cm}<{\centering}|p{1.2cm}<{\centering}|p{0.45cm}<{\centering}}
	\hline   
	\multirow{2}{*}{mesh}
	&\multicolumn{2}{c|}{$\frac{\lVert u-u_{h}\rVert_0}{\lVert u\rVert_0}$ }&\multicolumn{2}{c|}{$\frac{\lVert \bm{q}-\bm{q}_h \rVert_0}{\lVert \bm{q}\rVert_0}$}&\multicolumn{2}{c}{$\frac{\lVert \nabla u-\nabla_h u_h \rVert_0}{\lVert \nabla u\rVert_0}$}\cr\cline{2-7}  
	&error&order&error&order&error&order\cr  
	\cline{1-7}
		$8\times 8$	    &1.46E-02    &	--    &2.26E-02	  &  --   &8.30E-02	  &   --   \\
		\hline
		$16\times 16$	 &1.84E-03   &2.99	&5.76E-03	  &1.97		 &2.08E-02  &2.00\\
		\hline
		$32\times 32$	&2.30E-04   &3.00	&1.45E-03	  &1.99    &5.21E-03	  &2.00  \\
		\hline
		$64\times 64$	&2.88E-05    &3.00	&3.63E-04	  &2.00	 &1.30E-03	  &2.00  \\
		\hline
		$128\times 128$	&3.60E-06    &3.00	&9.08E-05	  &2.00  &3.26E-04	  &2.00 \\
		\hline
	\end{tabular}
}
\subtable[ $\alpha_1:\alpha_2=1000:1$ ]{
	\begin{tabular}{p{1.15cm}<{\centering}|p{1.2cm}<{\centering}|p{0.45cm}<{\centering}|p{1.2cm}<{\centering}|p{0.45cm}<{\centering}|p{1.2cm}<{\centering}|p{0.45cm}<{\centering}}
		\hline   
		\multirow{2}{*}{mesh}
		&\multicolumn{2}{c|}{$\frac{\lVert u-u_{h}\rVert_0}{\lVert u\rVert_0}$ }&\multicolumn{2}{c|}{$\frac{\lVert \bm{q}-\bm{q}_h \rVert_0}{\lVert \bm{q}\rVert_0}$}&\multicolumn{2}{c}{$\frac{\lVert \nabla u-\nabla_h u_h \rVert_0}{\lVert \nabla u\rVert_0}$}\cr\cline{2-7}  
		&error&order&error&order&error&order\cr  
		\cline{1-7}
		$8\times 8$	    &2.84E-02    &	--    &9.13E-02	    &  --  	 &5.00E-01  & --  \\
		\hline
		$16\times 16$	 &5.52E-03   &2.36	&3.49E-02	  &1.39	  &1.41E-01	  &1.83 \\
		\hline
		$32\times 32$	&7.95E-04   &2.80	 &1.01E-02	  &1.79	  &3.83E-02	  &1.88    \\
		\hline
		$64\times 64$	&1.03E-04    &2.95	&2.69E-03	  &1.91   &9.89E-03	  &1.95  \\
		\hline
		$128\times 128$	&1.31E-05    &2.97	&6.97E-04	  &1.95	  	&2.53E-03	  &1.97   \\
		\hline
	\end{tabular}
}
	\subtable[  $\alpha_1:\alpha_2=1:1000$]{
		\begin{tabular}{p{1.15cm}<{\centering}|p{1.2cm}<{\centering}|p{0.45cm}<{\centering}|p{1.2cm}<{\centering}|p{0.45cm}<{\centering}|p{1.2cm}<{\centering}|p{0.45cm}<{\centering}}
			\hline   
			\multirow{2}{*}{mesh}
			&\multicolumn{2}{c|}{$\frac{\lVert u-u_{h}\rVert_0}{\lVert u\rVert_0}$ }&\multicolumn{2}{c|}{$\frac{\lVert \bm{q}-\bm{q}_h \rVert_0}{\lVert \bm{q}\rVert_0}$}&\multicolumn{2}{c}{$\frac{\lVert \nabla u-\nabla_h u_h \rVert_0}{\lVert \nabla u\rVert_0}$}\cr\cline{2-7}  
			&error&order&error&order&error&order\cr  
			\cline{1-7}
			$8\times 8$	    &1.46E-02    &	--    &2.26E-02	  &  --     &8.30E-02	  &   --   \\
			\hline
			$16\times 16$	 &1.84E-03   &2.99	&5.76E-03	  &1.97	   &2.08E-02     &2.00\\
			\hline
			$32\times 32$	&2.30E-04   &3.00	   &1.45E-03	  &1.99  &5.21E-03	  &2.00  \\
			\hline
			$64\times 64$	&2.88E-05    &3.00	&3.63E-04	  &2.00		 &1.30E-03	  &2.00  \\
			\hline
			$128\times 128$	&3.60E-06    &3.00	&9.08E-05	  &2.00	    &3.26E-04	  &2.00 \\
			\hline
		\end{tabular}
	}
\end{table}

\begin{figure}[htbp]	
	\centering	
	\includegraphics[height = 5 cm,width=6 cm]{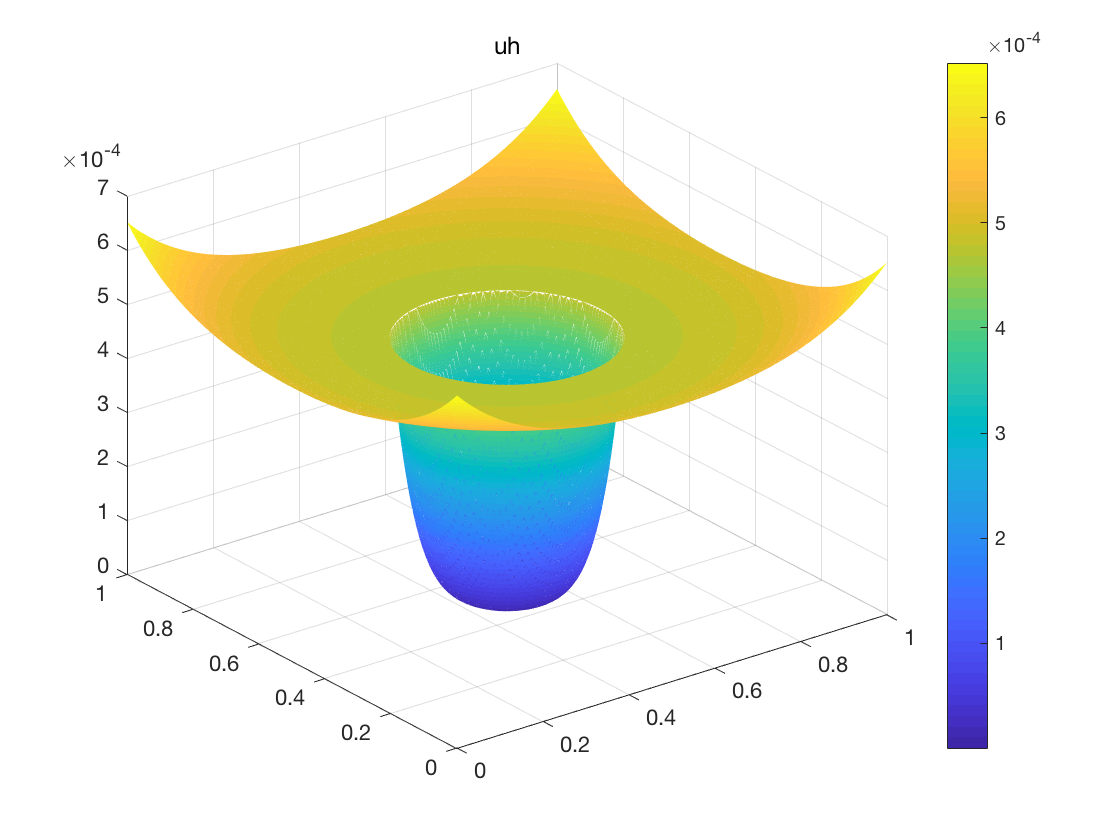} 			
	\includegraphics[height = 5 cm,width=6 cm]{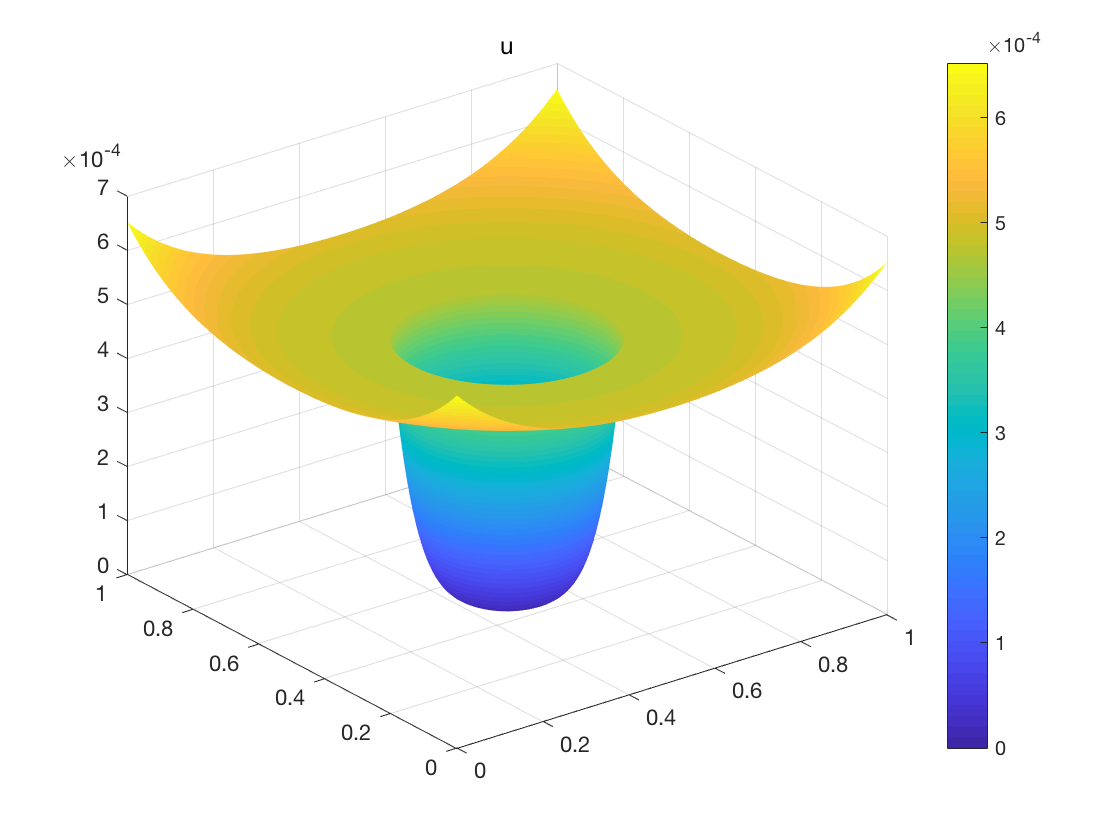} 	
	\tiny\caption{The X-HDG solution (left) and the exact solution (right) for the case of $\alpha_1:\alpha_2=1000:1$ with $k=1$ for Example \ref{circle}.}\label{1000:1}
\end{figure}

\begin{figure}[htbp]	
	\centering	
	\includegraphics[height = 5 cm,width=6 cm]{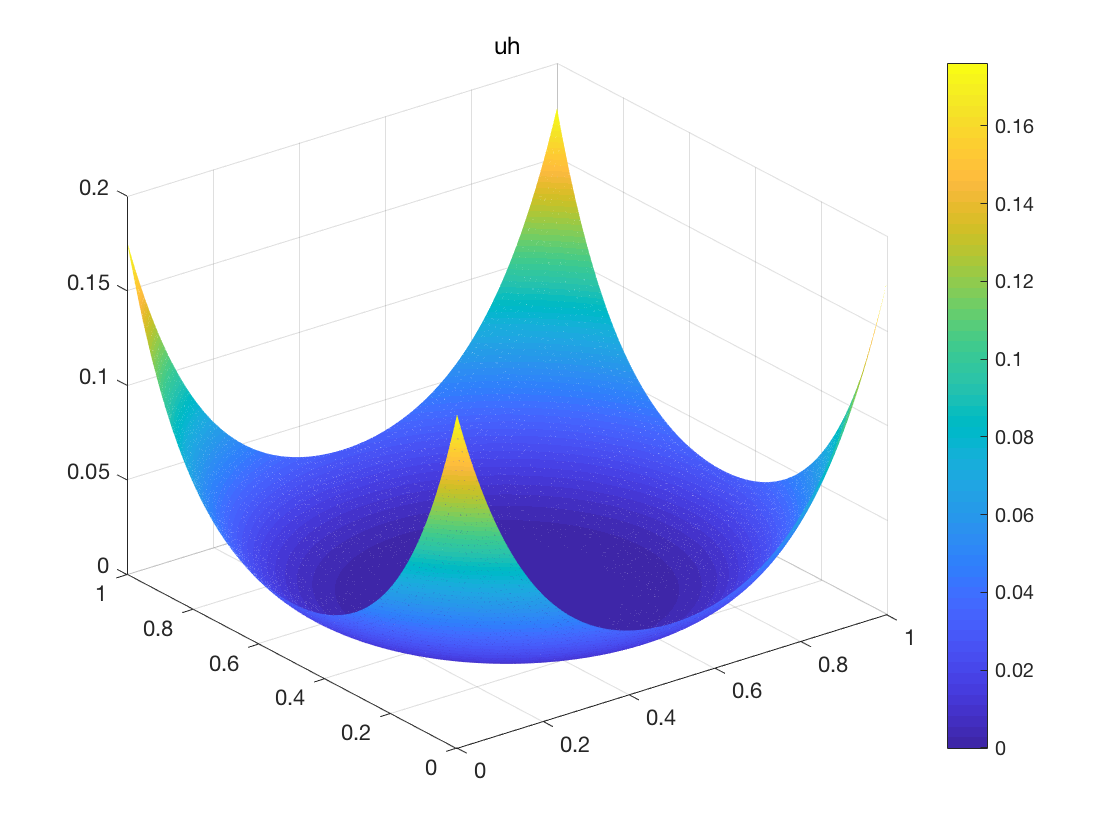} 			
	\includegraphics[height = 5 cm,width=6 cm]{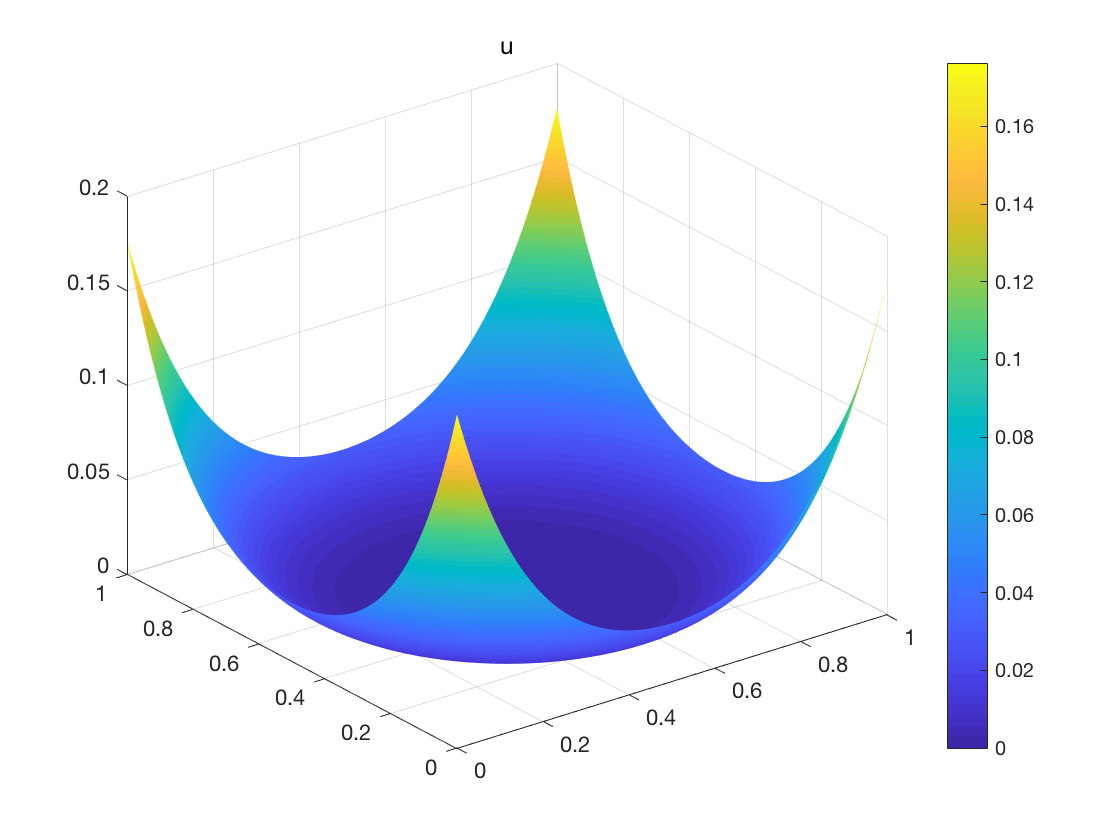} 	
	\tiny\caption{The X-HDG solution (left) and the exact solution (right) for the case of $\alpha_1:\alpha_2=1:1000$ with $k=1$ for Example \ref{circle}.}\label{1:1000}
\end{figure}

\begin{exmp}\label{cicularjump} \textbf{Circular interface with nonhomogeneous jump conditions}.
	
	Set  $\Omega = [0,1]\times [0,1]$ with a circular interface (cf. Figure \ref{circledomain}). The exact solution to  \eqref{pb1}  is given by
	\begin{align*}
	u(x,y) = \left \{
	\begin{array}{rl}
	e^{x}cosy,  \quad \quad \quad \,{\rm if} \  r < r_0,\\
	sin(\pi x)sin(\pi y),  \,{\rm if} \   r > r_0 ,
	\end{array}
	\right.
	\end{align*}
	where $r = \sqrt{(x-1/2)^2+(y-1/2)^2}$, $r_0 = \sqrt{3}/8$. Take $\alpha_1=1000, \alpha_2=1$.  The jump conditions $g_D$ and $g_N$, which  are not zero in this case, can be derived from the analytic solution.
\end{exmp}

\begin{table}[H]
	\normalsize
	\caption{History of convergence for Example \ref{cicularjump}:  X-HDG scheme (\ref{X-HDGscheme}) with $\alpha_1=1000,\alpha_2=1$. }\label{numericaljumpresultline11}
	\centering
	\footnotesize
		\subtable[ $k =1$]{
		\begin{tabular}{p{1.15cm}<{\centering}|p{1.2cm}<{\centering}|p{0.45cm}<{\centering}|p{1.2cm}<{\centering}|p{0.45cm}<{\centering}|p{1.2cm}<{\centering}|p{0.45cm}<{\centering}}
		\hline   
		\multirow{2}{*}{mesh}
		&\multicolumn{2}{c|}{$\frac{\lVert u-u_{h}\rVert_0}{\lVert u\rVert_0}$ }&\multicolumn{2}{c|}{$\frac{\lVert \bm{q}-\bm{q}_h \rVert_0}{\lVert \bm{q}\rVert_0}$}&\multicolumn{2}{c}{$\frac{\lVert \nabla u-\nabla_h u_h \rVert_0}{\lVert \nabla u\rVert_0}$}\cr\cline{2-7}  
		&error&order&error&order&error&order\cr  
		\cline{1-7}
			$8\times 8$	    &2.65E-02    &--       &1.23E-01	  &    --     &1.36E-01	  &  --    \\
			\hline
			$16\times 16$	 &6.77E-03   &1.97	&6.15E-02	  &0.99		  &6.16E-02	  &1.15	\\
			\hline
			$32\times 32$	&1.69E-03    &2.00	&3.09E-02	  &0.99  &3.04E-02	  &1.02  \\
			\hline
			$64\times 64$	&4.25E-04    &2.00	&1.55E-02	  &1.00 &1.41E-02	&1.11  \\
			\hline
			$128\times 128$	&1.06E-04    &2.00	&7.74E-03  &1.00  	&6.90E-03	  &1.03  \\
			\hline
		\end{tabular}
	}
		\subtable[ $k =2$]{
		\begin{tabular}{p{1.15cm}<{\centering}|p{1.2cm}<{\centering}|p{0.45cm}<{\centering}|p{1.2cm}<{\centering}|p{0.45cm}<{\centering}|p{1.2cm}<{\centering}|p{0.45cm}<{\centering}}
			\hline   
			\multirow{2}{*}{mesh}
			&\multicolumn{2}{c|}{$\frac{\lVert u-u_{h}\rVert_0}{\lVert u\rVert_0}$ }&\multicolumn{2}{c|}{$\frac{\lVert \bm{q}-\bm{q}_h \rVert_0}{\lVert \bm{q}\rVert_0}$}&\multicolumn{2}{c}{$\frac{\lVert \nabla u-\nabla_h u_h \rVert_0}{\lVert \nabla u\rVert_0}$}\cr\cline{2-7}  
			&error&order&error&order&error&order\cr  
			\cline{1-7}
			$8\times 8$	    &1.80E-03    &	   --  &1.01E-02	  &    --		 &3.10E-02  &--  \\
			\hline
			$16\times 16$	 &2.38E-04   &2.92	&2.52E-03	  &2.00	   &7.83E-03	  &1.99\\
			\hline
			$32\times 32$	&3.10E-05    &2.94	&6.31E-04	  &2.00	 &1.97E-03	  &1.99  \\
			\hline
			$64\times 64$	&3.89E-06    &2.99	&1.58E-04	  &2.00  &4.92E-04	  &2.00  \\
			\hline
			$128\times 128$	&4.87E-07    &3.00	&3.95E-05    &2.00    &1.23E-04	  &2.00 \\
			\hline
		\end{tabular}
	}
\end{table}

\begin{figure}[htbp]	
	\centering	
	\includegraphics[height = 5.5 cm,width=6.5 cm]{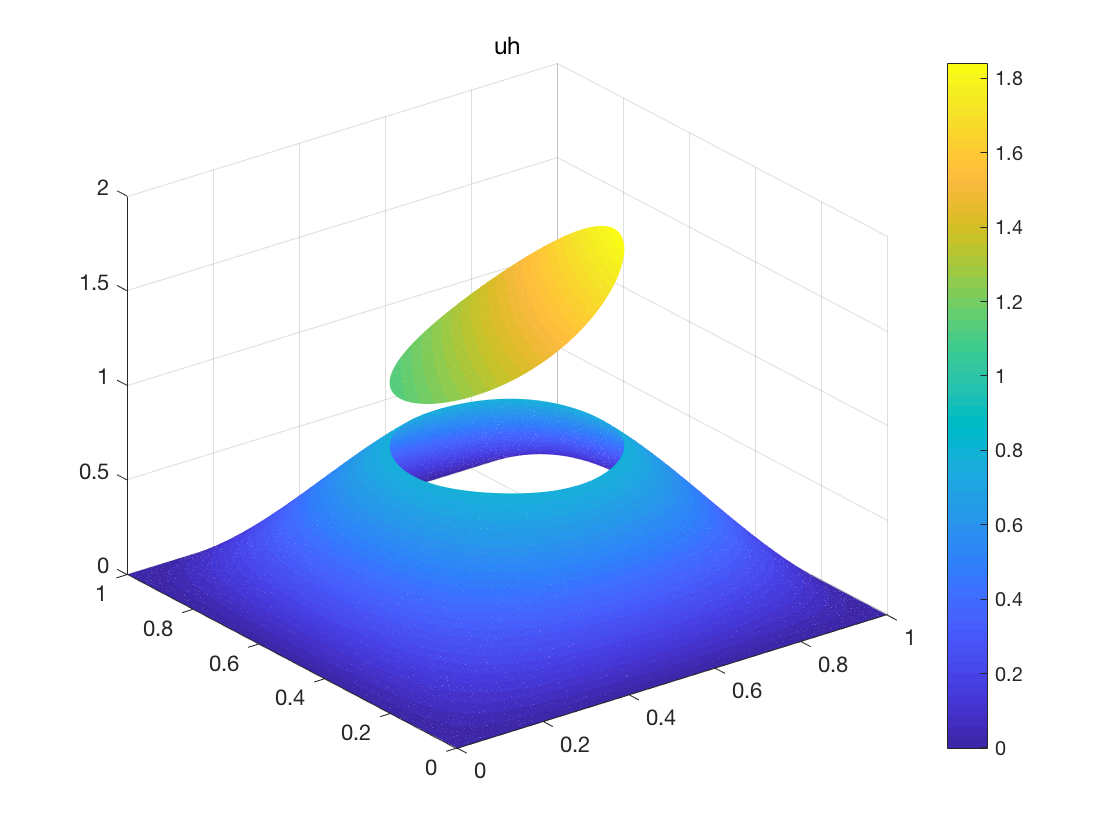} 			
	\includegraphics[height = 5.5 cm,width=6.5 cm]{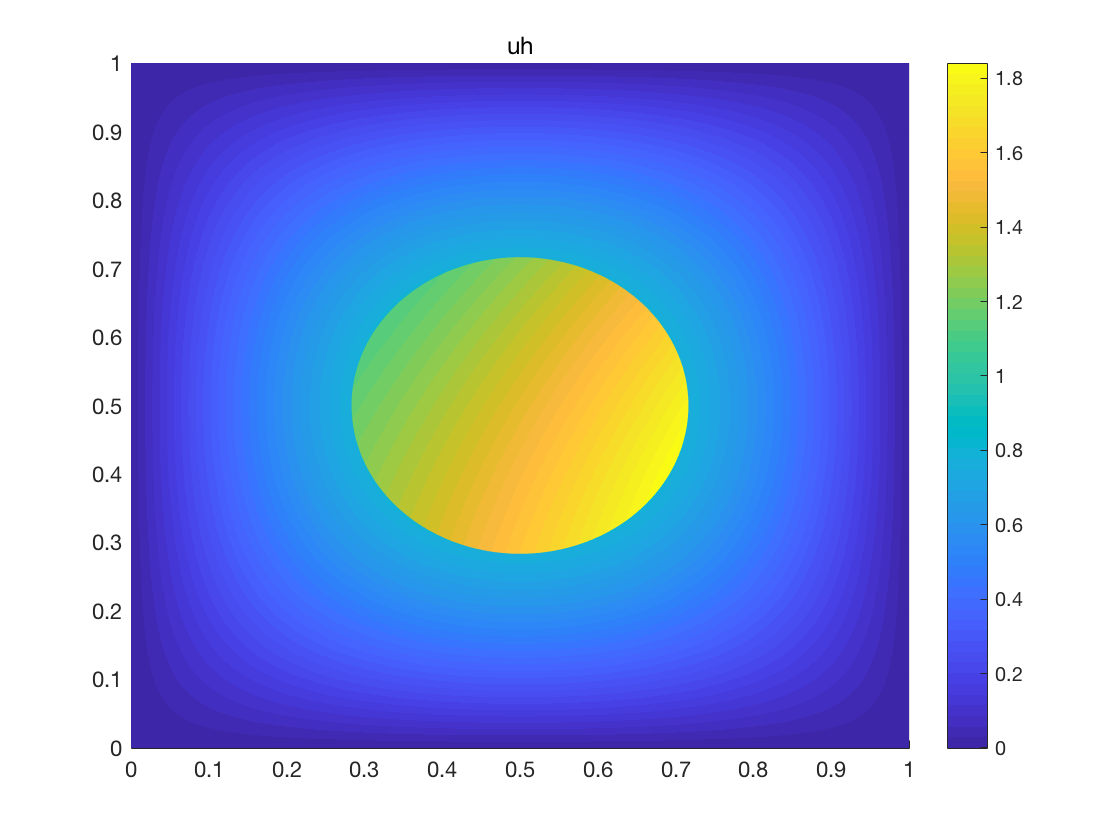} 			
	\tiny\caption{The X-HDG solution with $k=1$ for Example \ref{cicularjump}.}\label{circularjumpfig}
\end{figure}

Table \ref{numericaljumpresultline11} gives the numerical results obtained by the X-HDG scheme (\ref{X-HDGscheme}) with $k=1,k=2$, and  Fig \ref{circularjumpfig} shows he numerical solution with $k=1$  at $128\times 128$ mesh. We can see that the proposed method yields optimal convergence rates, i.e. $k$-th order  rates of convergence for  the error $\lVert \nabla u-\nabla_h u_h \rVert_0, \lVert \bm{q}-\bm{q}_h \rVert_0$, and $(k+1)$-th order rates of convergence for $\lVert u-u_{h}\rVert_0$.  In particular, the convergence rate of  $\lVert u-u_{h}\rVert_0$ is better than the  theoretical result in Theorem \ref{L2}, though  in this case $g_D\neq 0$ and the interface is not fold line.

\begin{exmp}\label{linejump} \textbf{Straight segment interface}.

	Set $\Omega = [0,1]\times [0,1]$ with a straight segment interface (Figure \ref{segment-domain}).   The exact solution  to  \eqref{pb1} is given by
	\begin{align*}
	u(x,y) = \left \{
	\begin{array}{rl}
	&5y^4+1, \,\  {\rm if} \,  y > b_0:= 0.2031,\\
	&y^4+4b_0^4, \ {\rm if} \, y < b_0.
	\end{array}
	\right.
	\end{align*}
	We consider two cases of  the coefficient $\alpha$: $\alpha_1=1000, \alpha_2 =1$ and $\alpha_1=1, \alpha_2 =1000$. 
\end{exmp}

Tables \ref{numericalresultline2}-\ref{numericalresultline1} list   the  numerical results obtained by the X-HDG scheme (\ref{X-HDGscheme}) and the modified X-HDG scheme  (\ref{X-HDGscheme1}) with $k=1,2$, and  Figure \ref{1:10} show 
 the numerical solution $u_h$  at $128\times 128$ mesh with $k=1$ and  $\alpha_1:\alpha_2 = 1000:1$. We can see that both of the methods yield optimal convergence rates. This is conformable to the theoretical results in Theorems \ref{estimateofenergynorm},\ref{L2}, \ref{energy-modified} and \ref{L2-modified}.

\begin{figure}[htbp]\label{segment-domain}	
	\centering	
	\includegraphics[height = 4.5 cm,width=5.5 cm]{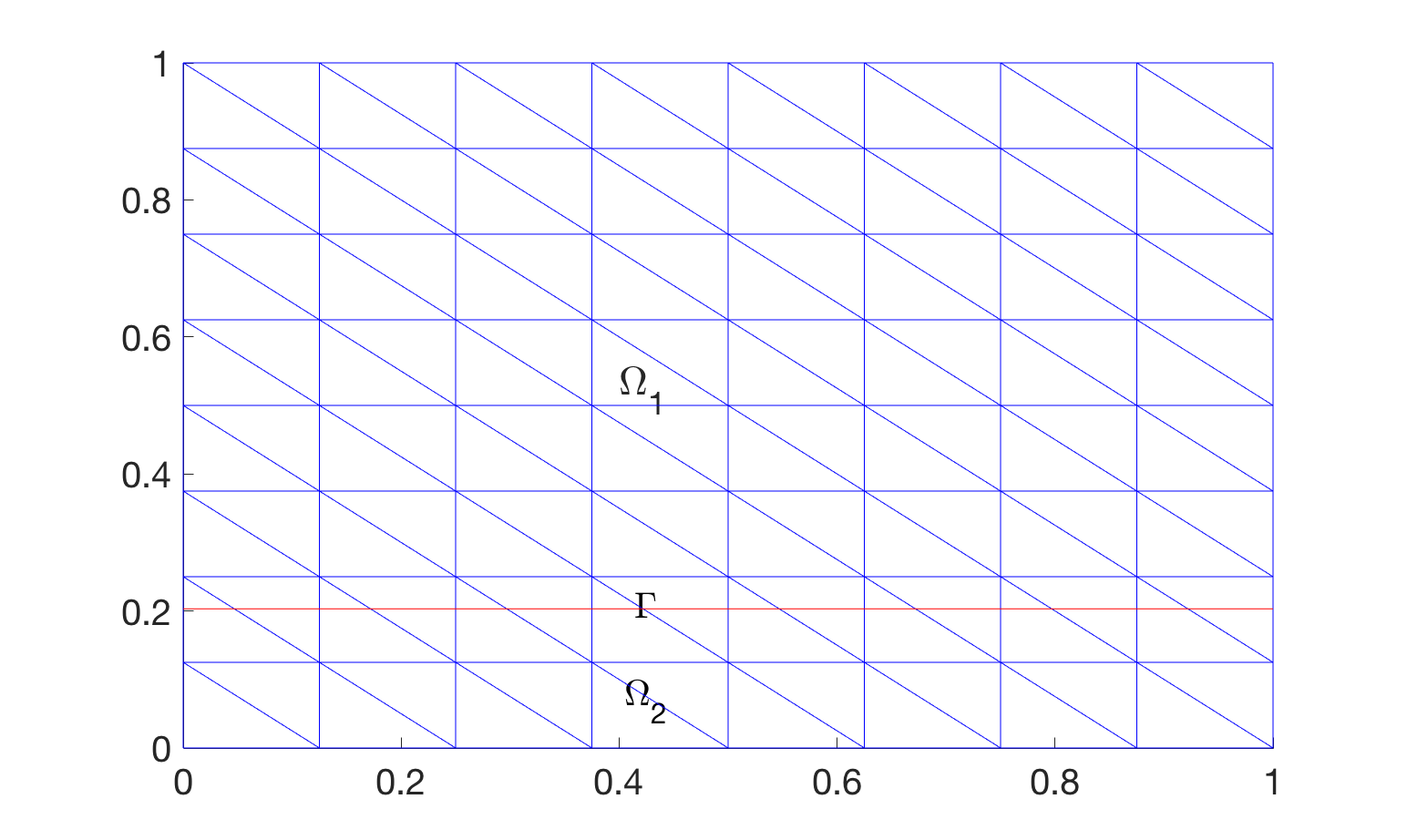} 				
	\caption{The domain with a straight segment interface: $8\times 8$ mesh}
\end{figure}

\begin{table}[H]
	\normalsize
	\caption{History of convergence  for Example \ref{linejump}: X-HDG scheme (\ref{X-HDGscheme}) with   $k=1$ }\label{numericalresultline2}
	\centering
	\footnotesize
		\subtable[ $\alpha_1:\alpha_2=1000:1$]
	{
			\begin{tabular}{p{1.15cm}<{\centering}|p{1.2cm}<{\centering}|p{0.45cm}<{\centering}|p{1.2cm}<{\centering}|p{0.45cm}<{\centering}|p{1.2cm}<{\centering}|p{0.45cm}<{\centering}}
			\hline   
			\multirow{2}{*}{mesh}
			&\multicolumn{2}{c|}{$\frac{\lVert u-u_{h}\rVert_0}{\lVert u\rVert_0}$ }&\multicolumn{2}{c|}{$\frac{\lVert \bm{q}-\bm{q}_h \rVert_0}{\lVert \bm{q}\rVert_0}$}&\multicolumn{2}{c}{$\frac{\lVert \nabla u-\nabla_h u_h \rVert_0}{\lVert \nabla u\rVert_0}$}\cr\cline{2-7}  
			&error&order&error&order&error&order\cr  
			\cline{1-7}
			$8\times 8$	    &2.60E-02    &	   --  &1.24E-01	  &   --   &1.23E-01	  &  --  \\
			\hline
			$16\times 16$	 &6.53E-03   &1.99	&6.27E-02	   &0.99  &6.01E-02	  &1.03	 \\
			\hline
			$32\times 32$	&1.63E-03    &2.00	&3.15E-02	  &1.00		  &2.98E-02	  &1.01  \\
			\hline
			$64\times 64$	&4.09E-04    &2.00	&1.57E-02	  &1.00   &1.49E-02	&1.00  \\
			\hline
			$128\times 128$	&1.02E-04    &2.00	&7.87E-03  &1.00  	&7.44E-03	  &1.00   \\
			\hline
		\end{tabular}
	}
		\subtable[$\alpha_1:\alpha_2=1:1000$]
	{
			\begin{tabular}{p{1.15cm}<{\centering}|p{1.2cm}<{\centering}|p{0.45cm}<{\centering}|p{1.2cm}<{\centering}|p{0.45cm}<{\centering}|p{1.2cm}<{\centering}|p{0.45cm}<{\centering}}
			\hline   
			\multirow{2}{*}{mesh}
			&\multicolumn{2}{c|}{$\frac{\lVert u-u_{h}\rVert_0}{\lVert u\rVert_0}$ }&\multicolumn{2}{c|}{$\frac{\lVert \bm{q}-\bm{q}_h \rVert_0}{\lVert \bm{q}\rVert_0}$}&\multicolumn{2}{c}{$\frac{\lVert \nabla u-\nabla_h u_h \rVert_0}{\lVert \nabla u\rVert_0}$}\cr\cline{2-7}  
			&error&order&error&order&error&order\cr  
			\cline{1-7}
			$8\times 8$	    &2.65E-02    &	   --  &1.24E-01	  &    -- &1.23E-01	  &   --   \\
			\hline
			$16\times 16$	 &6.65E-03   &1.99	&6.27E-02	  &0.99	 &6.00E-02  &1.03\\
			\hline
			$32\times 32$	&1.66E-03    &2.00	&3.14E-02	  &1.00    &2.98E-02	  &1.01  \\
			\hline
			$64\times 64$	&4.16E-04    &2.00	&1.57E-02	  &1.00		 &1.49E-02	  &1.00  \\
			\hline
			$128\times 128$	&1.04E-04    &2.00	&7.87E-03  &1.00 &7.44E-03	  &1.00 \\
			\hline
		\end{tabular}
	}
\end{table}

\begin{table}[H]
	\normalsize
	\caption{History of convergence  for Example \ref{linejump}: X-HDG scheme (\ref{X-HDGscheme}) with   $k=2$ }\label{numericalresultline12}
	\centering
	\footnotesize
		\subtable[ $\alpha_1:\alpha_2=1000:1$]{
	\begin{tabular}{p{1.15cm}<{\centering}|p{1.2cm}<{\centering}|p{0.45cm}<{\centering}|p{1.2cm}<{\centering}|p{0.45cm}<{\centering}|p{1.2cm}<{\centering}|p{0.45cm}<{\centering}}
		\hline   
		\multirow{2}{*}{mesh}
		&\multicolumn{2}{c|}{$\frac{\lVert u-u_{h}\rVert_0}{\lVert u\rVert_0}$ }&\multicolumn{2}{c|}{$\frac{\lVert \bm{q}-\bm{q}_h \rVert_0}{\lVert \bm{q}\rVert_0}$}&\multicolumn{2}{c}{$\frac{\lVert \nabla u-\nabla_h u_h \rVert_0}{\lVert \nabla u\rVert_0}$}\cr\cline{2-7}  
		&error&order&error&order&error&order\cr  
		\cline{1-7}
			$8\times 8$	    &1.01E-03    &	 --    &4.77E-03	  &  --    &1.47E-02	  &  --   \\
			\hline
			$16\times 16$	 &1.27E-04   &2.99	&1.21E-03	  &1.98	  &3.68E-03	  &2.00	 \\
			\hline
			$32\times 32$	&1.59E-05    &3.00	&3.02E-04	  &2.00  &9.20E-04	  &2.00   \\
			\hline
			$64\times 64$	&1.99E-06    &3.00	&7.58E-05	  &2.00  &2.30E-04	&2.00  \\
			\hline
			$128\times 128$	&2.49E-07    &3.00	&1.90E-05  &2.00  	&5.75E-05	  &2.00    \\
			\hline
		\end{tabular}
	}
\subtable[$\alpha_1:\alpha_2=1:1000$]{
	\begin{tabular}{p{1.15cm}<{\centering}|p{1.2cm}<{\centering}|p{0.45cm}<{\centering}|p{1.2cm}<{\centering}|p{0.45cm}<{\centering}|p{1.2cm}<{\centering}|p{0.45cm}<{\centering}}
		\hline   
		\multirow{2}{*}{mesh}
		&\multicolumn{2}{c|}{$\frac{\lVert u-u_{h}\rVert_0}{\lVert u\rVert_0}$ }&\multicolumn{2}{c|}{$\frac{\lVert \bm{q}-\bm{q}_h \rVert_0}{\lVert \bm{q}\rVert_0}$}&\multicolumn{2}{c}{$\frac{\lVert \nabla u-\nabla_h u_h \rVert_0}{\lVert \nabla u\rVert_0}$}\cr\cline{2-7}  
		&error&order&error&order&error&order\cr  
		\cline{1-7}
		$8\times 8$	    &1.01E-03    &	 --    &4.77E-03	  &  --  &1.47E-02	  &   --   \\
		\hline
		$16\times 16$	 &1.27E-04   &2.99	&1.21E-03	  &1.98		 &3.68E-03  &2.00\\
		\hline
		$32\times 32$	&1.59E-05    &3.00	&3.02E-04	  &2.00   &9.20E-04	  &2.00  \\
		\hline
		$64\times 64$	&1.99E-06    &3.00	&7.58E-05	  &2.00	 &2.30E-04	  &2.00  \\
		\hline
		$128\times 128$	&2.49E-07    &3.00	&1.90E-05  &2.00 &5.75E-05	  &2.00 \\
		\hline
	\end{tabular}
}
\end{table}

\begin{table}[H]
	\normalsize
	\caption{History of convergence  for Example \ref{linejump}: modified X-HDG scheme (\ref{X-HDGscheme1}) with   $k=1$ }\label{numericalresultline}
	\centering
	\footnotesize
		\subtable[ $\alpha_1:\alpha_2=1000:1$]{
		\begin{tabular}{p{1.15cm}<{\centering}|p{1.2cm}<{\centering}|p{0.45cm}<{\centering}|p{1.2cm}<{\centering}|p{0.45cm}<{\centering}|p{1.2cm}<{\centering}|p{0.45cm}<{\centering}}
			\hline   
			\multirow{2}{*}{mesh}
			&\multicolumn{2}{c|}{$\frac{\lVert u-u_{h}\rVert_0}{\lVert u\rVert_0}$ }&\multicolumn{2}{c|}{$\frac{\lVert \bm{q}-\bm{q}_h \rVert_0}{\lVert \bm{q}\rVert_0}$}&\multicolumn{2}{c}{$\frac{\lVert \nabla u-\nabla_h u_h \rVert_0}{\lVert \nabla u\rVert_0}$}\cr\cline{2-7}  
			&error&order&error&order&error&order\cr  
			\cline{1-7}
			$8\times 8$	    &3.10E-02    &	   --  &1.44E-01	  &   --   &1.59E-01	  &  --    \\
			\hline
			$16\times 16$	 &7.89E-03   &1.98	&7.33E-02	   &0.97  &7.98E-02	  &1.00	 \\
			\hline
			$32\times 32$	&1.98E-03    &1.99	&3.69E-02	  &0.99		  &3.99E-02	  &1.00   \\
			\hline
			$64\times 64$	&4.97E-04    &2.00	&1.85E-02	  &1.00   &2.00E-02	&1.00  \\
			\hline
			$128\times 128$	&1.24E-04    &2.00	&9.24E-03  &1.00  	&1.00E-02	  &1.00    \\
			\hline
		\end{tabular}
	}
		\subtable[$\alpha_1:\alpha_2=1:1000$]{
		\begin{tabular}{p{1.15cm}<{\centering}|p{1.2cm}<{\centering}|p{0.45cm}<{\centering}|p{1.2cm}<{\centering}|p{0.45cm}<{\centering}|p{1.2cm}<{\centering}|p{0.45cm}<{\centering}}
			\hline   
			\multirow{2}{*}{mesh}
			&\multicolumn{2}{c|}{$\frac{\lVert u-u_{h}\rVert_0}{\lVert u\rVert_0}$ }&\multicolumn{2}{c|}{$\frac{\lVert \bm{q}-\bm{q}_h \rVert_0}{\lVert \bm{q}\rVert_0}$}&\multicolumn{2}{c}{$\frac{\lVert \nabla u-\nabla_h u_h \rVert_0}{\lVert \nabla u\rVert_0}$}\cr\cline{2-7}  
			&error&order&error&order&error&order\cr  
			\cline{1-7}
			$8\times 8$	    &3.15E-02    &	   --  &1.44E-01	  &    --  &1.58E-01	  &   --   \\
			\hline
			$16\times 16$	 &8.01E-03   &1.98	&7.33E-02	  &0.97	 &7.97E-02  &0.99\\
			\hline
			$32\times 32$	&2.01E-03    &1.99	&3.69E-02	  &0.99	    &4.00E-02	  &1.00  \\
			\hline
			$64\times 64$	&5.04E-04    &2.00	&1.85E-02	  &1.00	 &2.00E-02	  &1.00  \\
			\hline
			$128\times 128$	&1.26E-04    &2.00	&9.24E-03  &1.00 &1.00E-02	  &1.00 \\
			\hline
		\end{tabular}
	}
\end{table}

\begin{table}[H]
	\normalsize
	\caption{History of convergence  for Example \ref{linejump}: modified X-HDG scheme (\ref{X-HDGscheme1}) with   $k=2$ }\label{numericalresultline1}
	\centering
	\footnotesize
		\subtable[ $\alpha_1:\alpha_2=1000:1$ ]{
		\begin{tabular}{p{1.15cm}<{\centering}|p{1.2cm}<{\centering}|p{0.45cm}<{\centering}|p{1.2cm}<{\centering}|p{0.45cm}<{\centering}|p{1.2cm}<{\centering}|p{0.45cm}<{\centering}}
			\hline   
			\multirow{2}{*}{mesh}
			&\multicolumn{2}{c|}{$\frac{\lVert u-u_{h}\rVert_0}{\lVert u\rVert_0}$ }&\multicolumn{2}{c|}{$\frac{\lVert \bm{q}-\bm{q}_h \rVert_0}{\lVert \bm{q}\rVert_0}$}&\multicolumn{2}{c}{$\frac{\lVert \nabla u-\nabla_h u_h \rVert_0}{\lVert \nabla u\rVert_0}$}\cr\cline{2-7}  
			&error&order&error&order&error&order\cr  
			\cline{1-7}
			$8\times 8$	    &1.12E-03    &	 --    &4.96E-03	  &  --    &1.56E-02	  &  --   \\
			\hline
			$16\times 16$	 &1.41E-04   &2.99	&1.26E-03	  &1.98  &3.91E-03	  &2.00	 \\
			\hline
			$32\times 32$	&1.77E-05    &2.99	&3.16E-04	  &1.99	  &9.77E-04	  &2.00   \\
			\hline
			$64\times 64$	&2.22E-06    &3.00	&7.93E-05	  &1.99  &2.44E-04	&2.00  \\
			\hline
			$128\times 128$	&2.77E-07    &3.00	&1.98E-05  &2.00  	&6.11E-05	  &2.00   \\
			\hline
		\end{tabular}
	}
		\subtable[$\alpha_1:\alpha_2=1:1000$]{
		\begin{tabular}{p{1.15cm}<{\centering}|p{1.2cm}<{\centering}|p{0.45cm}<{\centering}|p{1.2cm}<{\centering}|p{0.45cm}<{\centering}|p{1.2cm}<{\centering}|p{0.45cm}<{\centering}}
			\hline   
			\multirow{2}{*}{mesh}
			&\multicolumn{2}{c|}{$\frac{\lVert u-u_{h}\rVert_0}{\lVert u\rVert_0}$ }&\multicolumn{2}{c|}{$\frac{\lVert \bm{q}-\bm{q}_h \rVert_0}{\lVert \bm{q}\rVert_0}$}&\multicolumn{2}{c}{$\frac{\lVert \nabla u-\nabla_h u_h \rVert_0}{\lVert \nabla u\rVert_0}$}\cr\cline{2-7}  
			&error&order&error&order&error&order\cr  
			\cline{1-7}
			$8\times 8$	    &1.12E-03    &	   --  &4.96E-03	  &    --   &1.56E-02	  &   --   \\
			\hline
			$16\times 16$	 &1.41E-04   &2.99	&1.25E-03	  &1.98	 &3.91E-03  &2.00\\
			\hline
			$32\times 32$	&1.77E-05    &2.99	&3.16E-04	  &1.99    &9.77E-04	  &2.00  \\
			\hline
			$64\times 64$	&2.22E-06    &3.00	&7.93E-05	  &1.99	 &2.44E-04	  &2.00  \\
			\hline
			$128\times 128$	&2.77E-07    &3.00	&1.98E-05  &2.00 &6.11E-05	  &2.00 \\
			\hline
		\end{tabular}
	}
\end{table}

\begin{figure}[ht]\label{linefig}
	\centering
	\begin{minipage}[t]{0.4\textwidth}		
		\includegraphics[height = 5 cm,width=6 cm]{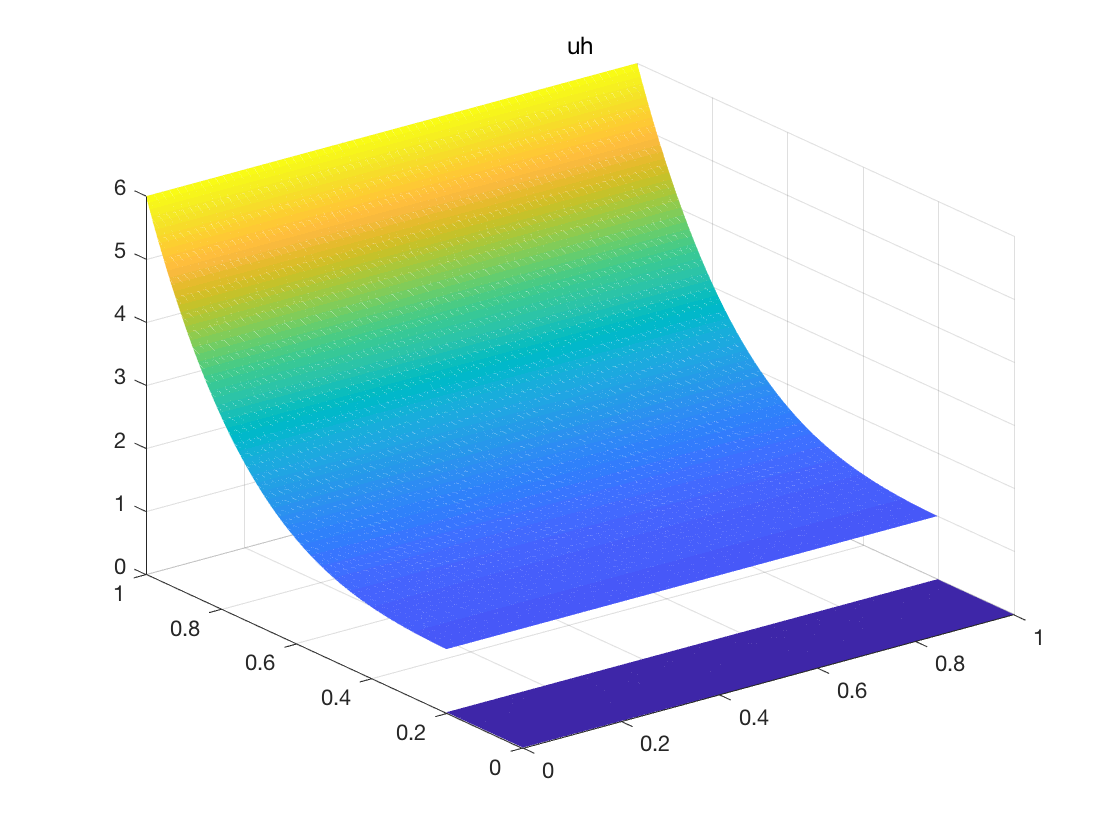} 	
	\end{minipage}
	\tiny\caption{The X-HDG solution at $128\times 128$ mesh with $k=1$    and  $\alpha_1:\alpha_2 = 1000:1$ 
	for Example \ref{linejump}.}	\label{1:10}
\end{figure}

\begin{exmp}\label{polygonalinterface1} \textbf{Polygonal interface with nonhomogeneous  jump conditions}\cite{yangcc2018interface}

	Set $\Omega = [0,2]\times [0,2]$ with a polygonal interface (Figure \ref{polygonaldomain}) 
	\begin{align*}
	\Gamma :=\{(x,y): \phi(x,y) = 0,\ a_0\leq x,y \leq 2-a_0\},
	\end{align*}
	where $\phi(x,y) = (y-(-x+1+a_0))(y-(x-1+a_0))(y-(-x-a_0+3))(y-(x+1-a_0))$, $a_0 = \sqrt{3}/4$. The exact solution to \eqref{pb1}  is given by 
	\begin{align*}
	u(x,y) = \left \{
	\begin{array}{rl}
	sin(x+y)+x^2y^2,  \quad      {\rm outside} \,\Gamma,\\
	e^{(x+y)},  \ \ \ \ \ \ \ \  {\rm inside} \,\Gamma .
	\end{array}
	\right.
	\end{align*}
	For  the coefficient $\alpha$, we take $ \alpha_1 =1000,\alpha_2=1$. 
	We note that the interface jump conditions, derived from the analytical solution, are  non-homogeneous. 
\end{exmp}

\begin{figure}[htbp]	
	\centering	
	\includegraphics[height = 5 cm,width=6 cm]{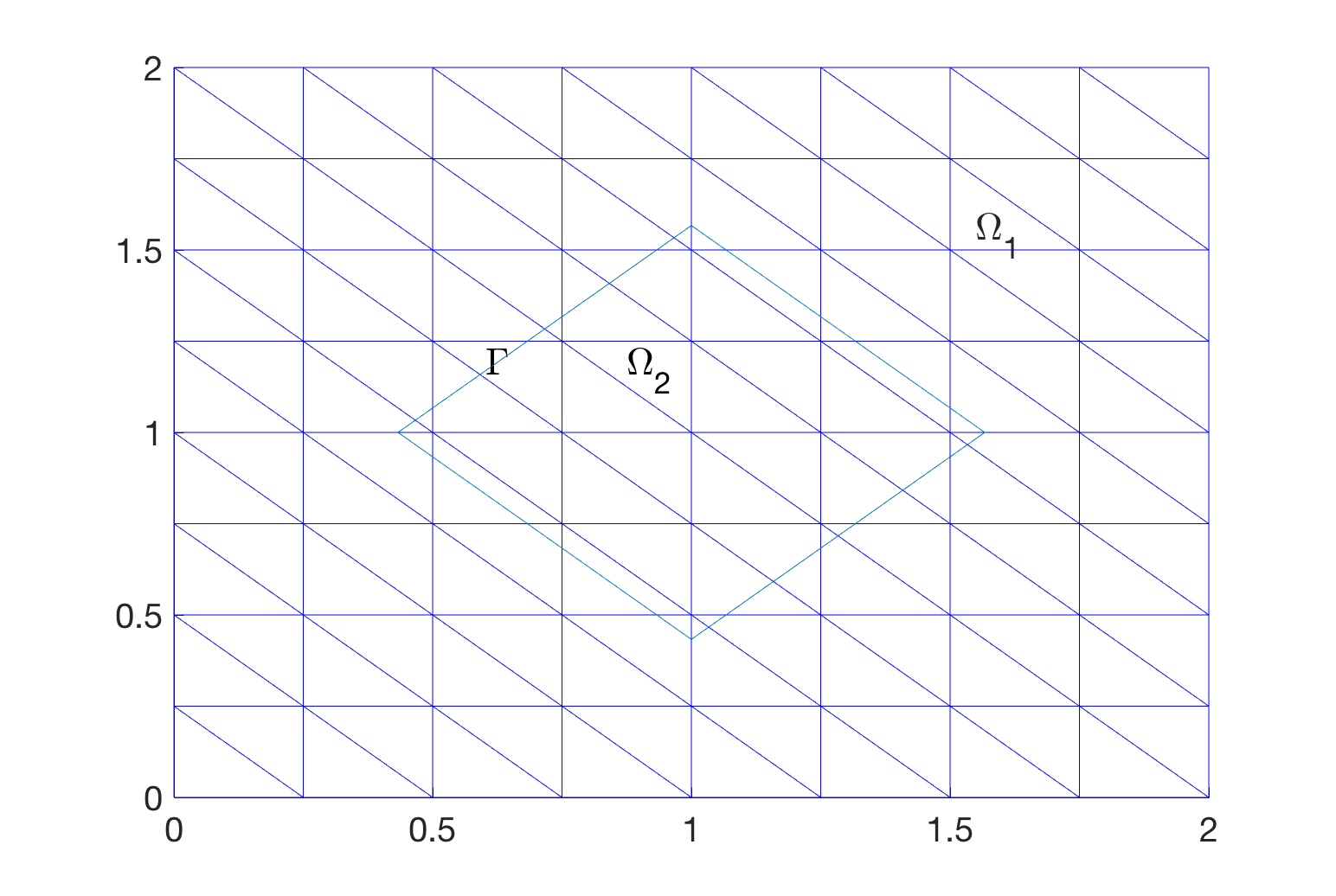} 			
	\tiny\caption{The domain with a polygonal interface: $8\times 8$ mesh.}\label{polygonaldomain}
\end{figure}

\begin{table}[H]
	\normalsize
	\caption{History of convergence for Example \ref{polygonalinterface1}:  X-HDG scheme (\ref{X-HDGscheme}) 
	}\label{polygonaltable12}
	\centering
	\footnotesize
				\subtable[ $k =1$]{
		\begin{tabular}{p{1.15cm}<{\centering}|p{1.2cm}<{\centering}|p{0.45cm}<{\centering}|p{1.2cm}<{\centering}|p{0.45cm}<{\centering}|p{1.2cm}<{\centering}|p{0.45cm}<{\centering}}
			\hline   
			\multirow{2}{*}{mesh}
			&\multicolumn{2}{c|}{$\frac{\lVert u-u_{h}\rVert_0}{\lVert u\rVert_0}$ }&\multicolumn{2}{c|}{$\frac{\lVert \bm{q}-\bm{q}_h \rVert_0}{\lVert \bm{q}\rVert_0}$}&\multicolumn{2}{c}{$\frac{\lVert \nabla u-\nabla_h u_h \rVert_0}{\lVert \nabla u\rVert_0}$}\cr\cline{2-7}  
			&error&order&error&order&error&order\cr  
			\cline{1-7}
			$8\times 8$	    &1.09E-02   & -- 	&6.15E-02	  &	  --  &6.43E-02	  &   --    \\
			\hline
			$16\times 16$	 &2.87E-03    &1.93	&3.06E-02	  &1.01		 &3.21E-02  &1.00 \\
			\hline
			$32\times 32$	&7.30E-04    &1.97	&1.60E-02	  &0.94   &1.63E-02	  &0.97   \\
			\hline
			$64\times 64$	&1.84E-04   &1.99	  &8.00E-03  &0.99	 &8.12E-03	  &1.00   \\
			\hline
			$128\times 128$	&4.60E-05    &2.00	&4.00E-03  &1.00 &4.05E-03	  &1.00	\\
			\hline
		\end{tabular}
	}
			\subtable[ $k =2$]{
		\begin{tabular}{p{1.15cm}<{\centering}|p{1.2cm}<{\centering}|p{0.45cm}<{\centering}|p{1.2cm}<{\centering}|p{0.45cm}<{\centering}|p{1.2cm}<{\centering}|p{0.45cm}<{\centering}}
			\hline   
			\multirow{2}{*}{mesh}
			&\multicolumn{2}{c|}{$\frac{\lVert u-u_{h}\rVert_0}{\lVert u\rVert_0}$ }&\multicolumn{2}{c|}{$\frac{\lVert \bm{q}-\bm{q}_h \rVert_0}{\lVert \bm{q}\rVert_0}$}&\multicolumn{2}{c}{$\frac{\lVert \nabla u-\nabla_h u_h \rVert_0}{\lVert \nabla u\rVert_0}$}\cr\cline{2-7}  
			&error&order&error&order&error&order\cr  
			\cline{1-7}
			$8\times 8$	    &3.39E-04    &--	&1.58E-03	  &--	  &8.21E-03	  &--	   \\
			\hline
			$16\times 16$	 &4.24E-05    &3.00	&3.93E-04	  &2.00  &2.10E-03	  &1.97	 \\
			\hline
			$32\times 32$	&5.65E-06    &2.91	&1.07E-04  &1.89 &5.51E-04	&1.93  \\
			\hline
			$64\times 64$	&7.11E-07    &2.99	&2.67E-05  &1.99 	&1.38E-04	  &1.99   \\
			\hline
			$128\times 128$	&8.90E-08    &3.00	&6.67E-06  &2.00 	&3.46E-05	  &2.00   \\
			\hline
		\end{tabular}
	}
\end{table}

\begin{table}[H]
	\normalsize
	\caption{History of convergence for Example \ref{polygonalinterface1}:  modified X-HDG scheme (\ref{X-HDGscheme1}) }
	\label{polygonaltable1}
	\centering
	\footnotesize
	\subtable[ $k =1$]{
			\begin{tabular}{p{1.15cm}<{\centering}|p{1.2cm}<{\centering}|p{0.45cm}<{\centering}|p{1.2cm}<{\centering}|p{0.45cm}<{\centering}|p{1.2cm}<{\centering}|p{0.45cm}<{\centering}}
			\hline   
			\multirow{2}{*}{mesh}
			&\multicolumn{2}{c|}{$\frac{\lVert u-u_{h}\rVert_0}{\lVert u\rVert_0}$ }&\multicolumn{2}{c|}{$\frac{\lVert \bm{q}-\bm{q}_h \rVert_0}{\lVert \bm{q}\rVert_0}$}&\multicolumn{2}{c}{$\frac{\lVert \nabla u-\nabla_h u_h \rVert_0}{\lVert \nabla u\rVert_0}$}\cr\cline{2-7}  
			&error&order&error&order&error&order\cr  
			\cline{1-7}
		$16\times 16$	    &2.85E-03   & -- 	&3.28E-02	  &	  --  &4.34E-02	  &   --   \\
		\hline
		$32\times 32$	 &7.84E-04    &1.86	&1.71E-02	  &0.94		 &2.43E-02  &0.84 \\
		\hline
		$64\times 64$	&1.97E-04    &1.99	&8.58E-03	  &0.99   &1.22E-02	  &0.98  \\
		\hline
		$128\times 128$	&4.93E-05   &2.00	  &4.29E-03  &1.00	 &6.14E-03	  &0.99   \\
		\hline
		$256\times 256$	&1.24E-05    &2.00	&2.15E-03  &1.00 &3.06E-03	  &1.00 \\
		\hline
		\end{tabular}
	}
			\subtable[ $k =2$]{
			\begin{tabular}{p{1.15cm}<{\centering}|p{1.2cm}<{\centering}|p{0.45cm}<{\centering}|p{1.2cm}<{\centering}|p{0.45cm}<{\centering}|p{1.2cm}<{\centering}|p{0.45cm}<{\centering}}
			\hline   
			\multirow{2}{*}{mesh}
			&\multicolumn{2}{c|}{$\frac{\lVert u-u_{h}\rVert_0}{\lVert u\rVert_0}$ }&\multicolumn{2}{c|}{$\frac{\lVert \bm{q}-\bm{q}_h \rVert_0}{\lVert \bm{q}\rVert_0}$}&\multicolumn{2}{c}{$\frac{\lVert \nabla u-\nabla_h u_h \rVert_0}{\lVert \nabla u\rVert_0}$}\cr\cline{2-7}  
			&error&order&error&order&error&order\cr  
			\cline{1-7}
			$16\times 16$	    &4.75E-05    &--	&4.08E-04	  &--	  &2.58E-03	  &--	   \\
			\hline
			$32\times 32$	 &6.60E-06    &2.85	&1.10E-04	  &1.88  &7.09E-04	  &1.87	 \\
			\hline
			$64\times 64$	&8.36E-07    &2.98	&2.77E-05  &2.00 &1.79E-04	&1.99  \\
			\hline
			$128\times 128$	&1.05E-07    &2.99	&6.92E-06  &2.00 	&4.49E-05	  &2.00   \\
			\hline
			$256\times 256$	&1.32E-08    &2.99	&1.49E-06  &2.00 	&1.12E-05	  &2.00   	\\
			\hline
		\end{tabular}
	}
\end{table}

Tables \ref{polygonaltable12}-\ref{polygonaltable1} give   the  numerical results obtained by the X-HDG scheme (\ref{X-HDGscheme}) and the modified X-HDG scheme  (\ref{X-HDGscheme1}) with $k=1,2$, and  Figure  \ref{polygonal1} shows 
the numerical solution $u_h$  at $128\times 128$ mesh with $k=1$. We can see that both of the schemes are of optimal convergence rates for the potential and flux approximations for all cases.

\begin{figure}[htbp]	
	\centering	
	\includegraphics[height = 5 cm,width=6.5 cm]{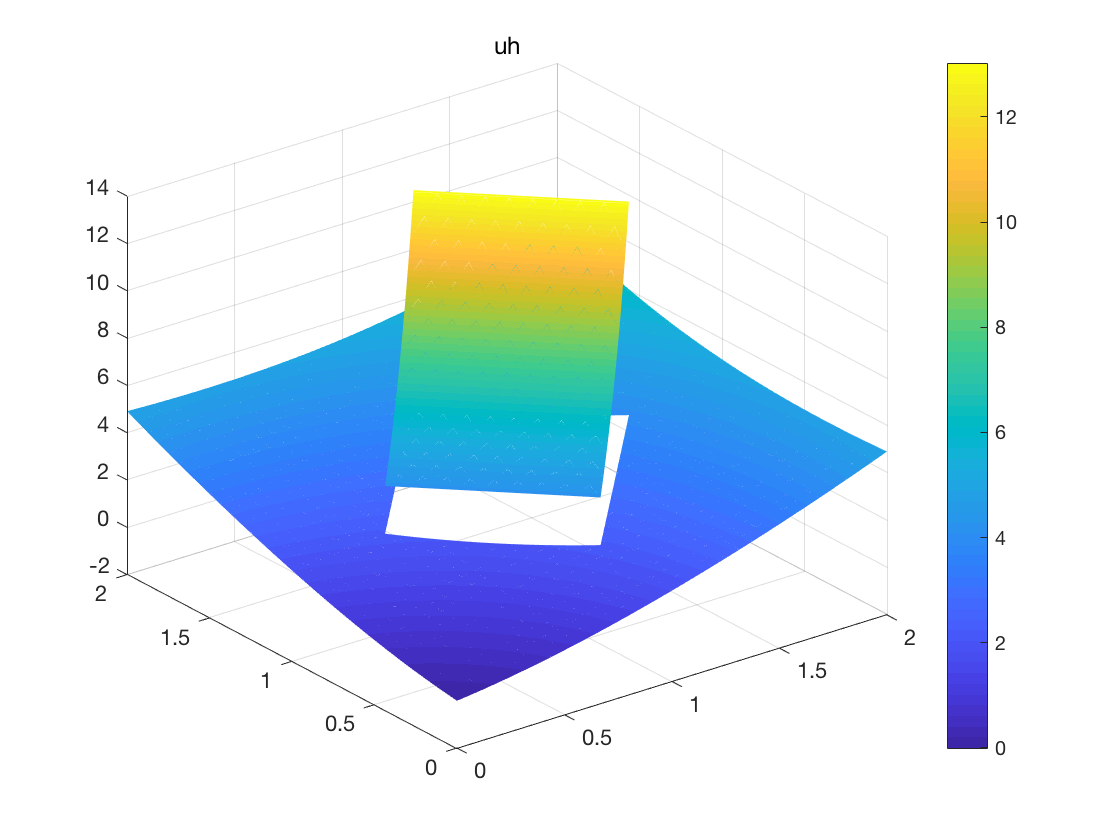} 			
	\includegraphics[height = 5 cm,width=6.5 cm]{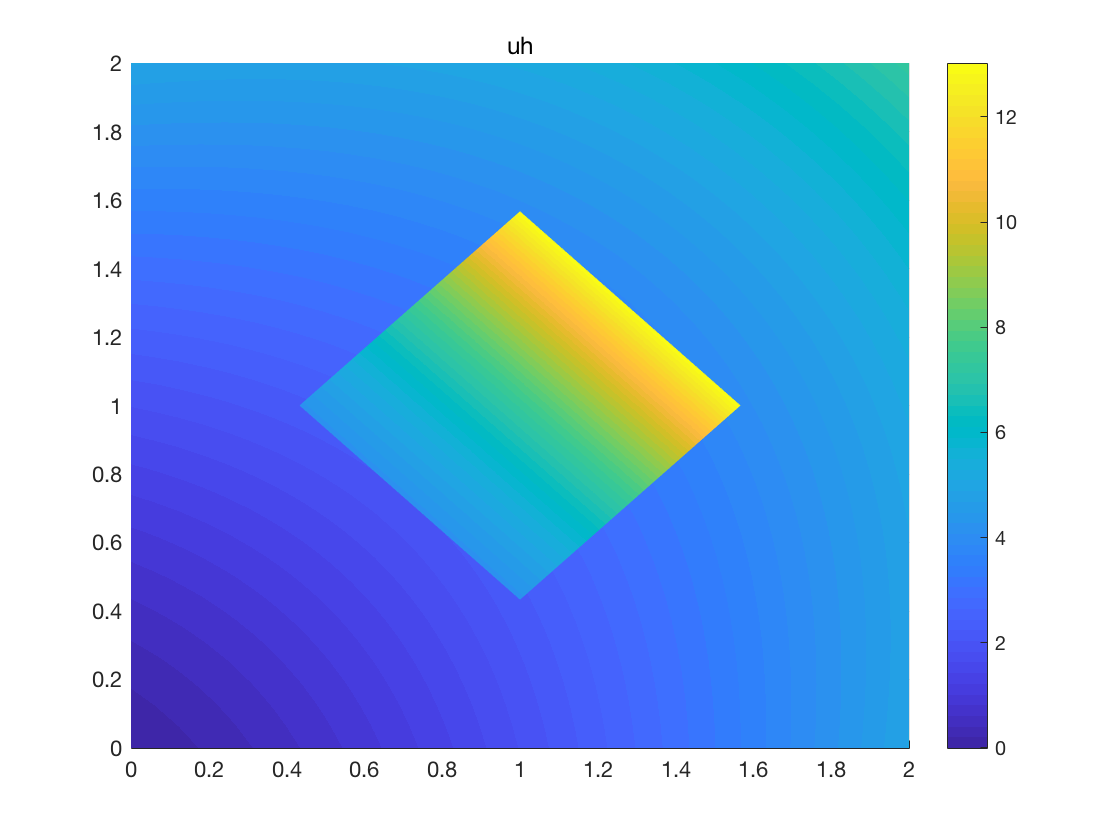} 			
	\tiny\caption{The X-HDG solution with $k=1$ for Example \ref{polygonalinterface1}.}\label{polygonal1}
\end{figure}

\section{Conclusions}
In this paper, we have proposed  two    arbitrary order  eXtended HDG methods for  the two and three dimensional second order elliptic interface problems.  Optimal error estimates have been derived for the flux and potential  approximations  without requiring   ``sufficiently large” stabilization parameters in the schemes.  Numerical experiments have  verified the theoretical results. 

In the future, we shall extend the eXtended HDG methods  to the Stokes and Brinkman interface problems, which have   important applications in the simulation of flow in porous media.

\footnotesize

\begin{thebibliography}{10}
	
	\bibitem{babuvska1970finite}
	I.~Babu{\v{s}}ka.
	\newblock The finite element method for elliptic equations with discontinuous
	coefficients.
	\newblock {\em Computing}, 5(3):207--213, 1970.
	
	\bibitem{babuska2011Stable}
	I.~Babu{\v{s}}ka and U.~Banerjee.
	\newblock Stable generalized finite element method ({SGFEM}).
	\newblock {\em Computer Methods in Applied Mechanics \& Engineering},
	201(1):91--111, 2011.
	
	\bibitem{babuvska1994special}
	I.~Babu{\v{s}}ka, G.~Caloz, and J.E. Osborn.
	\newblock Special finite element methods for a class of second order elliptic
	problems with rough coefficients.
	\newblock {\em SIAM Journal on Numerical Analysis}, 31(4):945--981, 1994.
	
	\bibitem{babuska2010Optimal}
	I.~Babu{\v{s}}ka and R.~Lipton.
	\newblock Optimal local approximation spaces for generalized finite element
	methods with application to multiscale problems.
	\newblock {\em Siam Journal on Multiscale Modeling \& Simulation},
	9(1):373--406, 2010.
	
	\bibitem{barrett1987fittedandunfitted}
	J.W. Barrett and C.M. Elliott.
	\newblock Fitted and unfitted finite element methods for elliptic equations
	with smooth interfaces.
	\newblock {\em IMA journal of numerical analysis}, 7(3):283--300, 1987.
	
	\bibitem{belytschko2009review}
	T.~Belytschko, R.~Gracie, and G.~Ventura.
	\newblock A review of extended/generalized finite element methods for material
	modeling.
	\newblock {\em Modelling and Simulation in Materials Science and Engineering},
	17(4):043001, 2009.
	
	\bibitem{Brambles96fin}
	J.H. Bramble and J.T. King.
	\newblock A finite element method for interface problems in domains with smooth
	boundaries and interfaces.
	\newblock {\em Advances in Computational Mathematics}, 6(1):109--138, 1996.
	
	\bibitem{Burman2012Fictitious}
	E.~Burman and P.~Hansbo.
	\newblock Fictitious domain finite element methods using cut elements: {II}.
	{A} stabilized {N}itsche method.
	\newblock {\em Applied Numerical Mathematics}, 62(4):328--341, 2012.
	
	\bibitem{Cai2017Discontinuous}
	Z.~Cai, C.~He, and S.~Zhang.
	\newblock Discontinuous finite element methods for interface problems: {R}obust
	a priori and a posteriori error estimates.
	\newblock {\em SIAM Journal on Numerical Analysis}, 55(1):400--418, 2017.
	
	\bibitem{Cai2011Discontinuous}
	Z.~Cai, X.~Ye, and S.~Zhang.
	\newblock Discontinuous {G}alerkin finite element methods for interface
	problems: A priori and a posteriori error estimations.
	\newblock {\em Society for Industrial and Applied Mathematics}, pages
	1761--1787, 2011.
	
	\bibitem{cao2015iterative}
	Y.~Cao, Y.~Chu, X.~He, and T.~Lin.
	\newblock An iterative immersed finite element method for an electric potential
	interface problem based on given surface electric quantity.
	\newblock {\em Journal of Computational Physics}, 281:82--95, 2015.
	
	\bibitem{chen2011mibpb}
	D.~Chen, Z.~Chen, C.~Chen, W.~Geng, and G.~Wei.
	\newblock {MIBPB}: a software package for electrostatic analysis.
	\newblock {\em Journal of computational chemistry}, 32(4):756--770, 2011.
	
	\bibitem{chen2015robust}
	H.~Chen, J.~Li, and W.~Qiu.
	\newblock Robust a posteriori error estimates for {HDG} method for
	convection--diffusion equations.
	\newblock {\em IMA Journal of Numerical Analysis}, 36(1):437--462, 2015.
	
	\bibitem{chen2014robust}
	H.~Chen, P.~Lu, and X.~Xu.
	\newblock A robust multilevel method for hybridizable discontinuous {G}alerkin
	method for the {H}elmholtz equation.
	\newblock {\em Journal of Computational Physics}, 264:133--151, 2014.
	
	\bibitem{chen2017superconvergent}
	H.~Chen, W.~Qiu, K.~Shi, and M.~Solano.
	\newblock A superconvergent {HDG} method for the {M}axwell equations.
	\newblock {\em Journal of Scientific Computing}, 70(3):1010--1029, 2017.
	
	\bibitem{Chen1998Finite}
	Z.~Chen and J.~Zou.
	\newblock Finite element methods and their convergence for elliptic and
	parabolic interface problems.
	\newblock {\em Numerische Mathematik}, 79(2):175--202, 1998.
	
	\bibitem{Cockburn2009Unified}
	B.~Cockburn, J.~Gopalakrishnan, and R.~Lazarov.
	\newblock Unified hybridization of discontinuous {G}alerkin, mixed, and
	continuous {G}alerkin methods for second order elliptic problems.
	\newblock {\em Siam Journal on Numerical Analysis}, 47(2):1319--1365, 2009.
	
	\bibitem{Cockburn2011HDGstokes}
	B.~Cockburn, J.~Gopalakrishnan, and N.C. Nguyen.
	\newblock Analysis of {HDG} methods for {S}tokes flow.
	\newblock {\em Mathematics of Computation}, 80(274):723--760, 2011.
	
	\bibitem{cockburn2010comparison}
	B.~Cockburn, N.C. Nguyen, and J.~Peraire.
	\newblock A comparison of {HDG} methods for {S}tokes flow.
	\newblock {\em Journal of Scientific Computing}, 45(1):215--237, 2010.
	
	\bibitem{cockburn2014divergence}
	B.~Cockburn and F-J. Sayas.
	\newblock Divergence-conforming {HDG} methods for {S}tokes flows.
	\newblock {\em Mathematics of Computation}, 83(288):1571--1598, 2014.
	
	\bibitem{Dong2018An}
	H.~Dong, B.~Wang, Z.~Xie, and L.L. Wang.
	\newblock An unfitted hybridizable discontinuous {G}alerkin method for the
	{P}oisson interface problem and its error analysis.
	\newblock {\em Ima Journal of Numerical Analysis}, 37(1):444--476, 2018.
	
	\bibitem{ewing1999immersed}
	R.E. Ewing, Z.~Li, T.~Lin, and Y.~Lin.
	\newblock The immersed finite volume element methods for the elliptic interface
	problems.
	\newblock {\em Mathematics and Computers in Simulation}, 50(1-4):63--76, 1999.
	
	\bibitem{Gupta2015Stable}
	V.~Gupta, C.~A. Duarte, I.~Babu{\v{s}}ka, and U.~Banerjee.
	\newblock Stable {GFEM} ({SGFEM}): Improved conditioning and accuracy of
	{GFEM/XFEM} for three-dimensional fracture mechanics.
	\newblock {\em Computer Methods in Applied Mechanics \& Engineering},
	289:355--386, 2015.
	
	\bibitem{G2016eXtended1}
	C.~Gürkan, M.~Kronbichler, and S.~Fernández-Méndez.
	\newblock e{X}tended hybridizable discontinuous {G}alerkin with heaviside
	enrichment for heat bimaterial problems.
	\newblock {\em Journal of Scientific Computing}, 72(2):1--26, 2016.
	
	\bibitem{G2016eXtendedvoid}
	C.~Gürkan, E.~Sala-Lardies, M.~Kronbichler, and S.~Fernández-Méndez.
	\newblock e{X}tended hybridizable discontinous {G}alerkin ({X-HDG}) for void
	problems.
	\newblock {\em Journal of Scientific Computing}, 66(3):1313--1333, 2016.
	
	\bibitem{hadley2002high}
	G.R. Hadley.
	\newblock High-accuracy finite-difference equations for dielectric waveguide
	analysis {II}: {D}ielectric corners.
	\newblock {\em Journal of lightwave technology}, 20(7):1219, 2002.
	
	\bibitem{hansbo2002unfitted}
	A.~Hansbo and P.~Hansbo.
	\newblock An unfitted finite element method, based on {N}itsche’s method, for
	elliptic interface problems.
	\newblock {\em Computer methods in applied mechanics and engineering},
	191(47-48):5537--5552, 2002.
	
	\bibitem{hesthaven2003high}
	J.S. Hesthaven.
	\newblock High-order accurate methods in time-domain computational
	electromagnetics: {A} review.
	\newblock In {\em Advances in imaging and electron physics}, volume 127, pages
	59--123. Elsevier, 2003.
	
	\bibitem{hou1997hybrid}
	T.Y. Hou, Z.~Li, S.~Osher, and H.~Zhao.
	\newblock A hybrid method for moving interface problems with application to the
	{H}ele--{S}haw flow.
	\newblock {\em Journal of Computational Physics}, 134(2):236--252, 1997.
	
	\bibitem{Huang2002Some}
	J.~Huang and J.~Zou.
	\newblock Some new a priori estimates for second-order elliptic and parabolic
	interface problems .
	\newblock {\em Journal of Differential Equations}, 184(2):570--586, 2002.
	
	\bibitem{Huang2012Uniform}
	J.~Huang and J.~Zou.
	\newblock Uniform a priori estimates for elliptic and static {M}axwell
	interface problems.
	\newblock {\em Discrete and Continuous Dynamical Systems-Series B (DCDS-B)},
	7(1):145--170, 2012.
	
	\bibitem{Huynh2013A}
	L.N.T. Huynh, N.C. Nguyen, J.~Peraire, and B.C. Khoo.
	\newblock A high-order hybridizable discontinuous {G}alerkin method for
	elliptic interface problems.
	\newblock {\em International Journal for Numerical Methods in Engineering},
	93(2):183--200, 2013.
	
	\bibitem{layton2009using}
	A.T. Layton.
	\newblock Using integral equations and the immersed interface method to solve
	immersed boundary problems with stiff forces.
	\newblock {\em Computers \& Fluids}, 38(2):266--272, 2009.
	
	\bibitem{leveque1994immersed}
	R.J Leveque and Z.~Li.
	\newblock The immersed interface method for elliptic equations with
	discontinuous coefficients and singular sources.
	\newblock {\em SIAM Journal on Numerical Analysis}, 31(4):1019--1044, 1994.
	
	\bibitem{Li-X2016analysis}
	B.~Li and X.~Xie.
	\newblock Analysis of a family of {HDG} methods for second order elliptic
	problems.
	\newblock {\em Journal of Computational and Applied Mathematics}, 307:37--51,
	2016.
	
	\bibitem{Li-X2016SIAM}
	B.~Li and X.~Xie.
	\newblock {BPX} preconditioner for nonstandard finite element methods for
	diffusion problems.
	\newblock {\em SIAM Journal on Numerical Analysis}, 54(2):1147--1168, 2016.
	
	\bibitem{Li-X-Z2016analysis}
	B.~Li, X.~Xie, and S.~Zhang.
	\newblock Analysis of a two-level algorithm for {HDG} methods for diffusion
	problems.
	\newblock {\em Communications in Computational Physics}, 19(5):1435--1460,
	2016.
	
	\bibitem{Li2010Optimal}
	J.~Li, M.J. Markus, B.I. Wohlmuth, and J.~Zou.
	\newblock Optimal a priori estimates for higher order finite elements for
	elliptic interface problems.
	\newblock {\em Applied Numerical Mathematics}, 60(1):19--37, 2010.
	
	\bibitem{lizhilin1998immersed}
	Z.~Li.
	\newblock The immersed interface method using a finite element formulation.
	\newblock {\em Applied Numerical Mathematics}, 27(3):253--267, 1998.
	
	\bibitem{lizhilin2006immersed}
	Z.~Li and K.~Ito.
	\newblock {\em The immersed interface method: numerical solutions of PDEs
		involving interfaces and irregular domains}, volume~33.
	\newblock Siam, 2006.
	
	\bibitem{lin2007error}
	T.~Lin, Y.~Lin, and W.~Sun.
	\newblock Error estimation of a class of quadratic immersed finite element
	methods for elliptic interface problems.
	\newblock {\em Discrete \& Continuous Dynamical Systems-B}, 7(4):807--823,
	2007.
	
	\bibitem{lin2015partially}
	T.~Lin, Y.~Lin, and X.~Zhang.
	\newblock Partially penalized immersed finite element methods for elliptic
	interface problems.
	\newblock {\em SIAM Journal on Numerical Analysis}, 53(2):1121--1144, 2015.
	
	\bibitem{Massjung2012An}
	R.~Massjung.
	\newblock An unfitted discontinuous {G}alerkin method applied to elliptic
	interface problems.
	\newblock {\em Siam Journal on Numerical Analysis}, 50(6):3134--3162, 2012.
	
	\bibitem{moes1999finite}
	N.~Mo{\"e}s, J.~Dolbow, and T.~Belytschko.
	\newblock A finite element method for crack growth without remeshing.
	\newblock {\em International journal for numerical methods in engineering},
	46(1):131--150, 1999.
	
	\bibitem{Cockburn2010HDGstokes}
	N.C. Nguyen, J.~Peraire, and B.~Cockburn.
	\newblock A hybridizable discontinuous {G}alerkin method for {S}tokes flow.
	\newblock {\em Computer Methods in Applied Mechanics and Engineering},
	199(9):582--597, 2010.
	
	\bibitem{nicaise2011optimal}
	S.~Nicaise, Y.~Renard, and E.~Chahine.
	\newblock Optimal convergence analysis for the extended finite element method.
	\newblock {\em International Journal for Numerical Methods in Engineering},
	86(4-5):528--548, 2011.
	
	\bibitem{Plum2003Optimal}
	M.~Plum and C.~Wieners.
	\newblock Optimal a priori estimates for interface problems.
	\newblock {\em Numerische Mathematik}, 95(4):735--759, 2003.
	
	\bibitem{strouboulis2000design}
	T.~Strouboulis, I.~Babu{\v{s}}ka, and K.~Copps.
	\newblock The design and analysis of the generalized finite element method.
	\newblock {\em Computer methods in applied mechanics and engineering},
	181(1-3):43--69, 2000.
	
	\bibitem{strouboulis2006generalized}
	T.~Strouboulis, I.~Babu{\v{s}}ka, and R.~Hidajat.
	\newblock The generalized finite element method for {H}elmholtz equation:
	theory, computation, and open problems.
	\newblock {\em Computer Methods in Applied Mechanics and Engineering},
	195(37-40):4711--4731, 2006.
	
	\bibitem{Wang2013Hybridizable}
	B.~Wang and B.C. Khoo.
	\newblock Hybridizable discontinuous {G}alerkin method ({HDG}) for {S}tokes
	interface flow.
	\newblock {\em Journal of Computational Physics}, 247(16):262--278, 2013.
	
	\bibitem{wang2016high}
	F.~Wang, Y.~Xiao, and J.~Xu.
	\newblock High-order e{X}tended finite element methods for solving interface
	problems.
	\newblock {\em arXiv preprint arXiv:1604.06171}, 2016.
	
	\bibitem{wangchen2014unfitted}
	Q.~Wang and J.~Chen.
	\newblock An unfitted discontinuous {G}alerkin method for elliptic interface
	problems.
	\newblock {\em Journal of Applied Mathematics}, 2014, 2014.
	
	\bibitem{wangtao2018nitsche}
	T.~Wang, C.~Yang, and X.~Xie.
	\newblock A {N}itsche-e{X}tended finite element method for distributed optimal
	control problems of elliptic interface equations.
	\newblock {\em arXiv preprint arXiv:1810.02271}, 2018.
	
	\bibitem{wu2010unfitted}
	H.~Wu and Y.~Xiao.
	\newblock An unfitted $ hp $-interface penalty finite element method for
	elliptic interface problems.
	\newblock {\em arXiv preprint arXiv:1007.2893}, 2010.
	
	\bibitem{xu2013estimate}
	J.~Xu.
	\newblock Estimate of the convergence rate of finite element solutions to
	elliptic equations of second order with discontinuous coefficients.
	\newblock {\em arXiv preprint arXiv:1311.4178}, 2013.
	
	\bibitem{yangcc2018interface}
	C.~Yang, T.~Wang, and X.~Xie.
	\newblock An interface-unfitted finite element method for elliptic interface
	optimal control problem.
	\newblock {\em Numerical mathmatics: Theory, Methods and Applications,
		accepted; arXiv:1805.04844v2}, 2018.
	
	\bibitem{zhang2004immersed}
	L.~Zhang, A.~Gerstenberger, X.~Wang, and W.K. Liu.
	\newblock Immersed finite element method.
	\newblock {\em Computer Methods in Applied Mechanics and Engineering},
	193(21-22):2051--2067, 2004.
	
	\bibitem{zhao2010high}
	S.~Zhao.
	\newblock High order matched interface and boundary methods for the {H}elmholtz
	equation in media with arbitrarily curved interfaces.
	\newblock {\em Journal of Computational Physics}, 229(9):3155--3170, 2010.
	
\end{thebibliography}

\end{document}